\documentclass[a4paper,12pt]{amsart}
\usepackage{amsfonts}
\usepackage{amsmath,amscd}
\usepackage{amssymb}
\usepackage{graphicx}
\usepackage{colonequals}
\usepackage{enumerate}
\usepackage{multicol}
\usepackage{mathrsfs}
\usepackage[usenames,dvipsnames]{pstricks}
\usepackage{epsfig}
\usepackage{pst-grad} 
\usepackage{pst-plot} 
\usepackage[hmargin=3cm,vmargin=2.5cm]{geometry}
\usepackage{rotating}
\usepackage[utf8]{inputenc}

\usepackage{mathtools}
\usepackage{amscd}
\usepackage{fancyhdr}
\usepackage{caption}
\usepackage{pdfpages}
\usepackage{lscape}

\usepackage{amsmath}
\usepackage{amscd}
\usepackage{latexsym}
\usepackage{amssymb}
\usepackage{graphicx}
\usepackage{amsthm}
\usepackage{delarray}
\newtheorem{theorem}{Theorem}[subsection]
\newtheorem{proposition}[theorem]{Proposition}
\newtheorem{corollary}[theorem]{Corollary}
\newtheorem{definition}[theorem]{Definition}
\newtheorem{lemma}[theorem]{Lemma}

\newtheorem{example}{Example}
\numberwithin{equation}{section}

\newcommand{\n}{\mathfrak n}
\newcommand{\z}{\mathfrak z}

\DeclareMathOperator{\diag}{diag}
\DeclareMathOperator{\End}{End}
\DeclareMathOperator{\C}{\mathcal{C}}

\DeclareMathOperator{\spn}{span}
\DeclareMathOperator{\la}{\langle}
\DeclareMathOperator{\ra}{\rangle}

\DeclareMathOperator{\Id}{Id}
\DeclareMathOperator{\Cl}{\rm{Cl}}
\DeclareMathOperator{\GL}{\rm{GL}}
\DeclareMathOperator{\SL}{\rm{SL}}
\DeclareMathOperator{\Aut}{\rm{Aut}}
\DeclareMathOperator{\Ad}{\rm{Ad}}

\DeclareMathOperator{\Pin}{Pin}
\DeclareMathOperator{\Spin}{Spin}

\DeclareMathOperator{\Orth}{O}

\DeclareMathOperator{\U}{U}
\DeclareMathOperator{\Sp}{Sp}

\DeclareMathOperator{\im}{Im}
\newcommand{\R}{\mathbb{R}}
\newcommand{\F}{\mathbb{F}}
\newcommand{\CC}{\mathbb{C}}
\newcommand{\HH}{\mathbb{H}}

\newcommand{\ii}{\mathbf{I}}
\newcommand{\jj}{\mathbf{J}}
\newcommand{\kk}{\mathbf{K}}

\DeclareMathOperator{\PP}{\mathbf{\Pi}}
\DeclareMathOperator{\QQ}{\mathbf{Q}}

\begin{document}
\title[Automorphism groups]{Automorphism groups of pseudo $H$-type algebras}

\author[Kenro Furutani, Irina Markina]  
{Kenro Furutani, Irina Markina}

\date{\today}
\thanks{
The first author was partially
supported by the Grant-in-aid for Scientific Research (C)
No. 17K05284, JSPS
and the National Center for Theoretical Science, National Taiwan University.
The second author was partially supported by 
the project Pure Mathematics in Norway, funded by the Trond Mohn Foundation}

\subjclass[2010]{Primary 17B60, 17B30, 17B70, 22E15}

\keywords{Clifford module, nilpotent 2-step Lie algebra, pseudo $H$-type algebras, Lie algebra automorphism, scalar product}

\address{K.~Furutani:  Department of Mathematics, Faculty of Science 
and Technology, Tokyo University of Science, 2641 Yamazaki, 
Noda, Chiba (278-8510), Japan}
\email{furutani\_kenro@ma.noda.tus.ac.jp}

\address{I.~Markina: Department of Mathematics, University of Bergen, P.O.~Box 7803,
Bergen N-5020, Norway}
\email{irina.markina@uib.no}

\begin{abstract}
In the present paper we determine the group of automorphisms of pseudo $H$-type Lie algebras, that are two step nilpotent Lie algebras closely related to the Clifford algebras $\Cl(\mathbb R^{r,s})$. 
\end{abstract}
\maketitle


\section{Introduction}

The pseudo $H$-type Lie algebras are two step nilpotent Lie algebras $\n_{r,s}(U)=(U\oplus\mathbb R^{r,s}, [.\,,.])$ endowed with a non-degenerate scalar product $\la.\,,.\ra_{U}+\la.\,,.\ra_{\R^{r,s}}$, where $U$ is the orthogonal complement to the centre $\mathbb R^{r,s}$ and the commutation relations are defined by 
$$
\la J_zu,v\ra_{U}=\la z,[u,v] \ra_{\R^{r,s}},\quad u,v\in U,\ z\in \R^{r,s}.
$$
Here $J_z\in\End (U)$, $J_z^2=-\la z,z\ra_{\R^{r,s}}\Id_{U}$ is the defining map for the representation $(J,U)$ of the Clifford algebra $\Cl(\R^{r,s})$. These Lie algebras are the natural generalisation of the $H$(eisenberg)-type algebras $\n_{r,0}(U)$, introduced in~\cite{Ka80,Kaplan81}, that are related to the Clifford algebras $\Cl(\R^{r,s})$ generated by a vector space endowed with the quadratic form of an arbitrary signature $(r,s)$.  The pseudo $H$-type Lie algebras were introduced in~\cite{Ci, GKM} and studied in~\cite{Ci97-1,Ci97-2,FM,FM1,FM2}. These type of algebras arise in study of parabolic subgroups with square integrable nilradicals~\cite{Wolf}, as maximal transitive prolongation of super Poincare algebras~~\cite{AC,AS1} and the nilpotent part of 2-gradings for semisimple Lie algebras~\cite{FGMMV,GKruglikovMV}. These algebras are some special examples of metric Lie algebras, studied in~\cite{AFMV,DeCoste,E,Eber03,Fischer}. The pseudo $H$-type Lie groups is a fruitful source for study of geometry with non-holonomic constrains or nilmanifolds~\cite{Barco, FigulaNagy, KOR}, symmetric spaces and harmonic spaces~\cite{Bell,BTV,CDKR,Pansu}, differential operators on Lie groups~\cite{Bauer,BieskeG,MullerSeeger,Ricci}. 

The main goal of the present paper is to describe the automorphism groups $\Aut(\n_{r,s}(U))$ of pseudo $H$-type algebras $\n_{r,s}(U)$ depending on the integer parameters $(r,s)$ and the structure of the representation $U$ of the Clifford algebra $\Cl(\R^{r,s})$. The automorphism groups preserving metric on $\n_{r,0}(U)$ were studied in~\cite{Riehm82,Riehm84} and the general automorphism groups of $\n_{r,0}(U)$ were described in~\cite{Barbano,KT,Saal,TamYosh}. Some attempt for study of $\Aut^0(\n_{0,1}(U))$ was done in~\cite{BarcoOvandoVittone}. An automorphism group $\Aut(\n_{r,s}(U))$ is decomposed into an abelian subgroup of dilatations, group ${\rm Hom}(U,\R^{r,s})$, the group generated by $\Pin(r,s)$, and a group $\Aut^0(\n_{r,s}(U))$ that acts trivially on the centre $\R^{r,s}$, see Section~\ref{GeneralAut}. The main goal is to determine the group $\mathbb A$ in terms of classical groups over $\R,\CC,\HH$, such that if $A\in \mathbb A\subset\GL(U)$, then
$A\oplus \Id\in \Aut^0(\n_{r,s}(U))$. 
The structure of the paper is the following. We recall the the necessary material about Clifford algebras and pseudo $H$-type Lie algebras in Sections 2 and 3. Section 5 is dedicated to the determination of the automorphism groups.The main result is contained in Table~\ref{tab:F-one}.


\section{Clifford algebras}



\subsection{Definition of Clifford algebras}


We denote by $\mathbb R^{r,s}$ the space $\mathbb R^{k}$, $r+s=m$, with
the non-degenerate quadratic form 
$
Q_{r,s}(z)=\sum_{i=1}^{r}z_{i}^{2}-\sum_{j=1}^{s}z_{r+j}^{2}$, $z\in \mathbb R^n
$ of the signature $(r,s)$.
The non-degenerate bi-linear form obtained from $Q_{r,s}$ by polarization
is denoted by $\la\cdot\,,\cdot\ra_{r,s}$. 
We call the form
$\la\cdot\,,\cdot\ra_{r,s}$ a {\it scalar product}. 
A vector $z\in\mathbb R^{r,s}$ is called
{\it positive} if $\la z,z\ra_{r,s}>0$,
{\it negative} if $\la z,z\ra_{r,s}<0$, and {\it null} 
if $\la z,z\ra_{r,s}=0$. 
We use the orthonormal basis
$\{z_1,\ldots,z_r,z_{r+1},\ldots,z_{r+s}\}$ 
for $\mathbb R^{r,s}$, where $\la z_i,z_i\ra_{r,s}=1$ for $i=1,\ldots,r$, and $\la z_j,z_j\ra_{r,s}=-1$ for $j=r+1,\ldots,r+s$. 

Let $\Cl_{r,s}$ be the real Clifford algebra 
generated by $\mathbb{R}^{r,s}$, that is the quotient of the tensor algebra 
$$
\mathcal{T}(\mathbb{R}^{r+s})=\mathbb{R}\oplus\left(\mathbb{R}^{r+s}\right)\oplus 
\left(\stackrel{2}\otimes\mathbb{R}^{r+s}\right)\oplus \left(\stackrel{3}\otimes\mathbb{R}^{r+s}\right)\oplus\cdots, 
$$
divided by the two-sided
ideal $I_{r,s}$ which is generated by the elements of the form
$z\otimes z+\la z,z\ra_{r,s}$, $z\in\mathbb{R}^{r+s}$.
The explicit determination of the Clifford algebras 
is given in~\cite{ABS} and they are isomorphic 
to matrix algebras 
$\mathbb{R}(n)$, 
$\mathbb{R}(n)\oplus\mathbb{R}(n)$, 
$\mathbb{C}(n)$, 
$\mathbb{H}(n)$ or 
$\mathbb{H}(n)\oplus\mathbb{H}(n)$ where 
the size $n$ is determined by $r$ and $s$, see~\cite{LawMich}. 

Given an algebra homomorphism $\widehat J\colon
\Cl_{r,s}\to\End(U)$, we call the space $U$ a Clifford module and
the operator $J_{\phi}$ a {\it Clifford action} or a {\it representation map}
of an element $\phi\in \Cl_{r,s}$. 
If there is a map
$$
\begin{array}{cccccc}
J\colon &\mathbb R^{r,s}&\to&\End(U)
\\
&z &\mapsto &J_z,
\end{array}
$$ 
satisfying $J^2_z=-\la z,z\ra_{r,s}\Id_{U}$ for an
arbitrary $z\in\mathbb R^{r,s}$, then $J$ can be uniquely extended to an algebra homomorphism 
$\widehat J$ by 
the universal property, see, for instance~\cite{Hu,Lam,LawMich}. We recommend to read~\cite{KobayashiYoshino} for wonderful introduction to the Clifford algebras $\Cl_{r,s}$.
Even though the representation matrices of the Clifford algebras
$\Cl_{r,s}$, and the Clifford modules $U$ are given over the fields $\mathbb{R}$, $\mathbb{C}$ or
$\mathbb{H}$, we refer to $\Cl_{r,s}$ as a real algebra and $U$ as a real vector space.

If $r-s\not\equiv 3 \,(\text{mod}~4)$, then $\Cl_{r,s}$ is a simple algebra. 
In this case there is only one
irreducible module $U=V_{irr}^{r,s}$ of dimension $n$. 
If $r-s\equiv 3 \,(\text{mod}~4)$, then the algebra $\Cl_{r,s}$ is not simple, 
and there are two non-equivalent irreducible modules.
They can be distinguished by the action of the ordered volume form 
$\Omega^{r,s}=\prod_{k=1}^{r+s}z_k$. In fact, the elements $\phi=\frac{1}{2}\big({\bf 1}\mp\Omega^{r,s}\big)$
act as an identity operator on the Clifford module, so 
$J_{\Omega^{r,s}}\equiv\pm \Id_{U}$. 
Thus we denote by $V^{r,s}_{irr;\pm}$ two non-equivalent irreducible Clifford modules 
on which the action of the volume form is given by $J_{\Omega^{r,s}}=\prod_{k=1}^{r+s}J_{z_k}\equiv\pm \Id$.

\begin{proposition}\cite[Proposition 4.5]{LawMich}\label{prop:reducibility}
Clifford modules are completely reducible; any Clifford module $U$ can be
decomposed into irreducible modules:
\begin{equation}\label{eq:isotypic}
U=
\begin{cases} 
\stackrel{p}\oplus V_{irr}^{r,s},\quad&\text{if}\quad r-s\not\equiv 3\,($\text{\em mod}$~4),
\\
\big(\stackrel{p_+}\oplus V_{irr;+}^{r,s}\big)\oplus
\big(\stackrel{p_-}\oplus V_{irr;-}^{r,s}\big),
\quad&\text{if}\quad r-s\equiv 3 \,($\text{\em mod}$~4).
\end{cases}
\end{equation}
The numbers $p$, $p_+,p_-$ are uniquely determined by the dimension of $U$.
\end{proposition}
The module $U= \stackrel{p}\oplus V_{irr}^{r,s}$ is called {\it isotypic} and the second one in~\eqref{eq:isotypic} is non-isotypic.
The Clifford algebras possess  the periodicity properties:
\begin{equation}\label{eq:Bott_periodicity}
\Cl_{r,s}\otimes \Cl_{0,8}\cong \Cl_{r,s+8},\quad
\Cl_{r,s}\otimes \Cl_{8,0}\cong \Cl_{r+8,s},\quad
\Cl_{r,s}\otimes \Cl_{4,4}\cong \Cl_{r+4,s+4},
\end{equation}
where the last one follows from $\Cl_{r,s}\otimes \Cl_{1,1}\cong
\Cl_{r+1,s+1}$, see~\cite{ABS}.
The Clifford algebras
$\Cl_{\mu,\nu}$ for $(\mu,\nu)\in\{(8,0),(0,8),(4,4)\}$
are isomorphic to $\mathbb{R}(16)$. 


\subsection{Admissible modules}


\begin{definition}\cite{Ci} A module $U$ of the Clifford algebra $\Cl_{r,s}$ 
is called admissible if there is a scalar product $\la\cdot\,,\cdot\ra_{U}$ on
$U$ such that
\begin{equation}\label{admissibility condition 1}
\la J_{z}x,y\ra_{U}+\la x,J_{z}y\ra_U=0,\quad\text{for all}\  \ x,y\in U\ \text{and}\  z\in \mathbb{R}^{r,s}.
\end{equation}
\end{definition}
We write $(U,\la \cdot\,,\cdot\ra_{U})$ for
an admissible module to emphasise the scalar product $\la\cdot\,,\cdot\ra_{U}$
and call it an {\it admissible scalar product}. 
We collect properties of admissible modules in several propositions.
\begin{proposition}\label{properties of admissible module} Let $\Cl_{r,s}$ be the Clifford algebra generated by the space $\mathbb R^{r,s}$.
\begin{itemize}
\item[(1)] {If $\la
\cdot\,,\cdot\ra_{U}$ 
is an admissible scalar product for $\Cl_{r,s}$, then $K\la \cdot\,,\cdot\ra_{U}$ 
is also admissible for any constant $K\neq 0$. We can assume that $K=\pm 1$ by normalisation of the scalar products.
}
\item[(2)] {Let $(U,\la\cdot\,,\cdot\ra_U)$ be an admissible module for $\Cl_{r,s}$ and let $(U_1,\la\cdot\,,\cdot\ra_{U_1})$ be such that $U_1$ is a submodule of $U$ and $\la\cdot\,,\cdot\ra_{U_1}$ is a non-degenerate
restriction of $\la\cdot\,,\cdot\ra_U$ to $U_{1}$. Then the orthogonal complement
${U_{1}}^{\perp}=\{x\in U\mid\ \la x,y\ra_{U}=0,\ \text{for all}\  y\in U_{1}\}$
with the scalar product obtained by the restriction of $\la\cdot\,,\cdot\ra_U$ to ${U_{1}}^{\perp}$ 
is also an admissible module.
}
\item[(3)] {Condition~\eqref{admissibility condition 1} and the
property $J^{\,2}_{z}=-\la z,z\ra_{r,s}\Id_U$ imply
\begin{equation}\label{admissibility condition 2}
\la J_{z}x,J_{z}y\ra_{U}=\la z,z\ra_{r,s}\la x,y\ra_{U}.
\end{equation}
}
\item[(4)] {If $s>0$, then any admissible module $(U,\la \cdot\,,\cdot\ra_{U})$ of $\Cl_{r,s}$
is {\it neutral}, i.e., $\dim U=2l$, $l\in\mathbb N$, and $U$ it is isometric to
$\mathbb{R}^{l,l}$, see~{\em \cite[Proposition 2.2]{Ci}}.
}
\item[(5)] {If $s=0$, then any Clifford module of $\Cl_{r,0}$
    can be made into admissible with positive definite or negative
    definite scalar product, see~\cite[Theorem 2.4]{Hu}.
}
\end{itemize}
\end{proposition}

Proposition~\ref{r-s = 0,1,2 mod 4} describes the relation between irreducible and admissible modules. An admissible module of the minimal possible dimension is called a {\it minimal admissible module}. 

\begin{proposition}\label{r-s = 0,1,2 mod 4}\cite[Theorem 3.1]{Ci}\cite[Proposition 1]{FM1}
Let $\Cl_{r,s}$ be the Clifford algebra generated by the space $\mathbb R^{r,s}$.
\begin{itemize}
\item[ (1)] {If $s=0$, then any irreducible Clifford module is minimal
admissible with respect to a positive definite or a 
negative definite scalar product. 
}
\item[(2)] {If $r-s\equiv 0,1,2\mod 4$, $s>0$,
then a unique irreducible module $V^{r,s}_{irr}$ is not necessary admissible. The following situations are possible: 
\begin{itemize}
\item[(2-1)] {The irreducible module $V^{r,s}_{irr}$ is minimal admissible or,}
\item[(2-2)] {The irreducible module $V^{r,s}_{irr}$ is not admissible, but the
direct sum $V^{r,s}_{irr}\oplus V^{r,s}_{irr}$ is minimal admissible.}
\end{itemize}
}
\item[(3)] {
If $r-s\equiv 3 \mod 4$, $s>0$, then for two non-equivalent irreducible modules $V_{irr;\pm}^{r,s}$ the following can occur:
\begin{itemize}
\item[(3-1)] { If $r\equiv 3 \mod 4$, $s\equiv 0 \mod 4$, or 
  
\noindent\quad $r\equiv 1\mod 8$, $s\equiv 6\mod 8$,
or 

\noindent\quad $r\equiv 5\mod 8$, $s\equiv 2\mod 8$

\noindent then each irreducible module $V_{irr;\pm}^{r,s}$ is minimal admissible.}
\item[(3-2)] {Otherwise none of the irreducible modules $V_{irr;\pm}^{r,s}$ is
admissible. 
\begin{itemize}
\item[(3-2-1)] {If $r\equiv 1\mod 8$, $s\equiv 2\mod 8$,
or 

\noindent\quad $r\equiv 5\mod 8$, $s\equiv 6\mod 8$

\noindent then $V_{irr;+}^{r,s}\oplus  V_{irr;+}^{r,s}$, $V_{irr;-}^{r,s}\oplus V_{irr;-}^{r,s}$ are minimal admissible modules, and the module $V_{irr;+}^{r,s}\oplus V_{irr;-}^{r,s}$ is not admissible.}
\item[(3-2-2)]{ If $s$ is odd, then the module $V_{irr;+}^{r,s}\oplus V_{irr;-}^{r,s}$ is minimal admissible and neither $V_{irr;+}^{r,s}\oplus V_{irr;+}^{r,s}$ nor $V_{irr;-}^{r,s}\oplus V_{irr;-}^{r,s}$ is admissible.}
\end{itemize}
}
\end{itemize}
}
\end{itemize}
\end{proposition}

The dimensions of minimal admissible modules need to be determined only for basic cases
\begin{equation}\label{basic cases}
\begin{array}{l}
(r,s)~\text{for}~0\leq r\leq 7~\text{and $0\leq s\leq 3$},\\
(r,s)~\text{for}~0\leq r\leq 3~\text{and $4\leq s\leq 7$}, ~\text{and}\\
(r,s)~\in \{(8,0), (0,8), (4,4)\}.
\end{array}
\end{equation} 
We use periodicity property~\eqref{eq:Bott_periodicity} to find the dimension of a minimal admissible module $\dim(V^{r+\mu,s+\nu}_{min})=\dim(V^{r,s}_{min})\cdot\dim(V^{\mu,\nu}_{min})=16\dim(V^{r,s}_{min})$ provided that 
$V^{r,s}_{min}$ is minimal admissible and
$(\mu,\nu)\in\{(8,0),(0,8),(4,4)\}$. Moreover $\dim(V^{r,s}_{min})=2^{r+s-p_{r,s}}$ where $p_{r,s}$ is the maximal number of symmetric positive mutually commuting involutions $PI_{r,s}$ that admits $V^{r,s}_{min}$, see Section~\ref{sec:Ers}, or~\cite[Section 2.3]{FM2}. We describe the number and the dimension
of minimal admissible modules $V^{r,s}_{min}$ in
Table~\ref{t:dim}. We indicate whether 
the scalar product restricted to the common 1-eigenspaces $E_{r,s}\subset V_{min}^{r,s}$ of the
involutions from $PI_{r,s}$ is neutral or sign definite, see Section~\ref{sec:Ers} or
~\cite[Section 2.6]{FM2} for details of the proof.
\begin{table}[htbp]
	\center
	\caption{Dimensions of minimal admissible modules}
\vskip0.1cm	
\scalebox{0.8}
{
\begin{tabular}{|c||c|c|c||c||c|c|c||c||c|}
\hline
		${\text{\small 8}} $&$ {\text{\small{16}}^{\pm}}$&${\text{\small 32}}^{\pm}$&${\text{\small{64}}}^{\pm}
		$&${\text{\small{64}$_{\times 2}^{\pm}$}}$&${\text{\small{128}}^{\pm}}$&${\text{\small{128}}^{\pm}}$&${\text{{\small{128}}}^{\pm}}$&$
		{\text{{\small{128}$_{\times 2}^{\pm}$}}} $&${\text{\small{256}}^{\pm}}$
\\
\hline\hline
		${\text{\small 7}}$ &$ {\text{\small{16}}^{N}}$&${\text{\small{32}}^{N}}$&$
		{\mathbf{\small{64}}^{N}} $&${\text{\small{64}}}^{\pm}
		$&${\mathbf{\small{128}}^{N}}$&${\mathbf{\small{128}}^{N}}  $&${\mathbf{\small{128}}^{N}}$&$ {\text{\small{128}}^{\pm}} $&${\text{\small{256}}^{N}}$
\\
\hline
		${\text{\small 6}}$ &${\text{\small{16}}^{N}}$&${\text{\small{16}$_{\times 2}^{N}$}}$&${\text{\small{32}}^{N}}$&${\text{\small{32}}^{\pm}}$&${\mathbf{\small{64}}^{N}}
		$&${\mathbf{\small 64}}^N_{\times 2}$&${\mathbf{\small{128}}^{N}} $&${\text{\small{128}}^{\pm}}$&$ {\text{\small{256}}^{N}} $
\\
\hline
		${\text{\small 5}} $&${\mathbf{\small 16}^{N}}$&${\text{\small 16}^{N}}$&${\text{\small 16}^{N}}$&${\text{\small 16}}^{\pm}$&${\mathbf{\small 32}}^{N}$&${\mathbf{\small 64}}^{N} $&${\mathbf{\small{128}}^{\pm}}$&${\text{\small{128}}^{N}} $&$\mathbf{\small{256}}^{N}$
\\
\hline\hline
		${\text{\small 4}} $&$  {\text{\small 8}^{\pm}}$&$ {\text{\small 8}^{\pm}}$&$
		{\text{\small 8}^{\pm}}$&$ 8_{\times 2}^{{\pm}}$&$16^{\pm}$&${\text{\small 32}}^{\pm}$&${\text{\small 64}}^{\pm}
		$&${\text{\small 64}_{\times 2}^{\pm}} $&${\text{\small{128}}^{\pm}}$
\\
\hline\hline
		${\text{\small 3}}$&${\mathbf{\small 8}}^{N}$&${\mathbf{\small 8}}^{N}$&${\mathbf{\small 8}^{N}}$&$8^{\pm}$&$16^{N}$&$32^{N}$
		&${\mathbf{\small 64}^{N}}$&$64^{\pm}$&${\mathbf{\small 128}}^{N}$
\\
\hline
		${\text{\small 2}}$&${\mathbf{\small 4}^{N}}$&$
		{\mathbf 4}_{\times 2}^{N}$&${\mathbf 8^{N}}$&$ 8^{\pm}$&$16^{N}$&$16_{\times 2}^{N}$&$32^{N}$&$32^{\pm} $&${\mathbf{\small 64}}^{N}$
\\
\hline
		${\text{\small 1}}$ &${\mathbf 2^{N}}$&${\mathbf 4^{N}}$&${\mathbf 8^{N}}$& $8^{\pm}$&$\mathbf {16}^{N}$&$16^{N}$&$16^{N}$&$16^{\pm}$&${\mathbf{\small 32}}^{N}$
\\
\hline\hline
		${\text{\small 0}} $&$  1^{\pm}$&$ 2^{\pm}$&$4^{\pm}$&$ 4_{\times 2}^{\pm}$&$ 8^{\pm}$&$ 8^{\pm}$&$ 8^{\pm}$&$ 8_{\times 2}^{\pm}$&$16^{\pm}$
		\\
		\hline\hline
		{s/r}&  {\text{\small 0}}& {\text{\small 1}}& 
		{\text{\small 2}}&{\text{\small 3}} & {\text{\small 4}}& {\text{\small 5}}& {\text{\small 6}}& {\text{\small 7}}& {\text{\small 8}}
		\\
		\hline
	\end{tabular}\label{t:dim}
	}
\end{table}
We make the following comments to Table~\ref{t:dim}:
\begin{itemize}
\item[(1)] {We use the black colour  when $\dim(V^{r,s}_{\min})=2\dim(V^{r,s}_{irr})$, see Proposition~\ref{r-s = 0,1,2 mod 4}, 
statements (2-2) and (3-2).
}
\item[(2)] {Writing the subscript "${\times 2}$" we show that the Clifford algebra has two minimal admissible modules corresponding to the non-equivalent irreducible modules, see Proposition~\ref{r-s = 0,1,2 mod 4}.
statements (3-1) and (3-2-1).}
\item[(3)] {The upper index "${N}$" means that the scalar product 
restricted to the common $1$-eigenspace $E_{r,s}$ is neutral. 
The fact that $E_{r,s}$ is a neutral space does not depend on the
choice of the scalar product on $V^{r,s}_{min}$.
}
\item[(4)] { The upper index "${\pm}$" shows that the scalar product
    restricted to the common 1-eigenspace $E_{r,s}$ of the system
    $PI_{r,s}$ is sign definite. 
The sign of the scalar product on $E_{r,s}$ depends on the choice of the admissible scalar product on the module $V^{r,s}_{min}$.
}
\end{itemize}


\subsection{System of involutions $PI_{r,s}$ and common 1-eigenspace $E_{r,s}$}\label{sec:Ers}



\subsubsection{Mutually commuting isometric involutions}


Recall that a linear map $\Lambda$ defined on a vector space $U$ with
a scalar
product $\la\cdot\,,\cdot\ra_{U}$ is called 
{\it symmetric} with respect to the scalar product $\la\cdot\,,\,\cdot\ra_{U}$, 
if $\la \Lambda x,y\ra_{U}=\la x,\Lambda y\ra_{U}$. We say that $\Lambda$ is
{\it positive} if it maps positive vectors to
positive vectors and negative vectors to negative vectors
and $\Lambda$ is {\it negative} if it reverses the positivity and negativity of the vectors.
Let $J_{z_i}$ be representation maps for an orthonormal basis
$\{z_1,\ldots,z_{r+s}\}$ of~$\mathbb R^{r,s}$. 
The simplest positive involutions, written as a product of the maps
$J_{z_i}$, are one of the following forms:
\begin{equation*}\label{eq:basicInvol}
\begin{cases}
\text{type (1)}:\  \mathcal{P}_1=J_{z_{i_1}} J_{z_{i_2}} J_{z_{i_3}} J_{z_{i_4}},\ 
\text{all}\ z_{i_k}~\text{are positive},
  \\
\text{type (2)}:\  \mathcal{P}_1=J_{z_{i_1}} J_{z_{i_2}} J_{z_{i_3}} J_{z_{i_4}},\ 
  \text{all}\ z_{i_k}~\text{are negative},
  \\
\text{type (3)}:\  \mathcal{P}_{2}= J_{z_{i_1}} J_{z_{i_2}} J_{z_{i_3}} J_{z_{i_4}},\ \text{two $z_{i_l}$ are positive and two are negative},
\\
\text{type (4)}:\  \mathcal{P}_{3}=J_{z_{i_1}} J_{z_{i_2}} J_{z_{i_3}},\qquad \text{all three
 $z_{i_{k}}$ are positive},
 \\
\text{type (5)}:\  \mathcal{P}_{4}=J_{z_{i_1}} J_{z_{i_2}} J_{z_{i_3}},\qquad\text{one $z_{i_l}$ is positive
  and two are negative}.
  \end{cases}
\end{equation*}
For a given minimal admissible module $V_{min}^{r,s}$, 
we denote by $PI_{r,s}$ a set of the maximal
number of mutually commuting symmetric positive 
involutions of types (1)-(5) 
and such that none of them is a product
of other involutions in $PI_{r,s}$. 
The set $PI_{r,s}$ is not unique, while 
the number of involutions $p_{r,s}=\#\{PI_{r,s}\}$ in $PI_{r,s}$ is unique for the given signature~$(r,s)$.
The ordering on the set $PI_{r,s}$ can be made, if necessary, in such a way 
that at most one involution of the type (4) or (5) is included in $PI_{r,s}$ and it is the last one. We denote by $PI^*_{r,s}$ the reduced system of involution, that contains only involutions of type (1)-(3). In the case when there are no involutions of type (4) or (5), we have $PI_{r,s}=PI_{r,s}^*$ and we write  $PI_{r,s}$ if no confusion arises.

We define the subspace $E_{r,s}$ of a minimal admissible module $V_{min}^{r,s}$ by
\begin{eqnarray*}\label{common 1-eigenspaceNco1}
E_{r,s}=\{v\in V^{r,s}_{min}&\mid&\ P_iv=v,\ \ i\leq p_{r,s}, 
\\
&&r-s \neq 3\mod 4,\ \text{or}\ r-s\equiv 3\mod 4 \ \text{with odd}\ s\}
\end{eqnarray*}
$$
E_{r,s}=\{v\in V^{r,s}_{min}\mid\
  P_iv=v,\ \ i\leq p_{r,s}-1,\ \ r-s\equiv 3\mod 4 \ \text{with even}\ s\}.
$$
We call $E_{r,s}$ the ``{\it common 1-eigenspace}'' for the system of involutions $PI_{r,s}$. The space $E_{r,s}$ is the minimal subspace of $V_{min}^{r,s}$ that is invariant under the action of all involutions from $PI_{r,s}$.
The system of involutions $PI_{r,s}$ does not depend on the scalar product on the admissible modules $V_{min}^{r,s}=(V_{min}^{r,s},\la\cdot\,,\cdot\ra_{V_{min}^{r,s}})$ 
and $V_{min}^{r,s}=(V_{min}^{r,s},-\la\cdot\,,\cdot\ra_{V_{min}^{r,s}})$. 
Nevertheless, the restrictions of the admissible scalar products on the respective $E_{r,s}$ will have opposite signs. It is indicated in Table~\ref{t:dim} that
\begin{itemize}
\item[(1)]{the restriction of the admissible scalar product on $E_{r,s}$ is sign definite for $r\equiv 0,1,2\mod 4$ and
$s\equiv 0 \mod 4$ or for $r\equiv 3\mod 4$ and
arbitrary $s$;}
\item[(2)]{otherwise the restriction of the admissible scalar product on the common 1-eigenspaces $E_{r,s}$ is neutral.}
\end{itemize}
From now on we use $\pm$ or $N$ as the upper index and write
$V_{min}^{r,s;+}$ ($V_{min}^{r,s;-}$) or $V_{min}^{r,s;N}$ if the restriction of the admissible scalar product on $E_{r,s}$ is positive (negative) definite or neutral.
We also use the lower index $\pm$ 
to distinguish the minimal admissible modules, corresponding to a choice of non equivalent irreducible modules that were mentioned in Proposition~\ref{r-s = 0,1,2 mod 4},
statements (3-1) and (3-2-1).  

According to these agreements any admissible module can be
decomposed into the orthogonal sum of minimal admissible modules, see
Proposition~\ref{properties of admissible module}, statement~(2). We distinguish the following possibilities. 
\begin{itemize}
\item[$$] {If $r-s\not\equiv 3 \mod 4$ and $s$ is arbitrary or $r-s\equiv 3 \mod 4$ and $s$ is odd then 
\begin{equation}\label{eq:rsneq3}
U= \big(\stackrel{p^+}\oplus V_{min}^{r,s;+}\big)\bigoplus\big(\stackrel{p^-}\oplus V_{min}^{r,s;-}\big).
\end{equation}
}
\item[$$] {If $r-s\equiv 3\mod 4$ and $s$ is even, then 
\begin{equation}\label{eq:rseq3}
U= \big(\stackrel{p^+_+}\oplus V_{min;+}^{r,s;+}\big)\bigoplus \big(\stackrel{p^-_+}\oplus V_{min;+}^{r,s;-}\big)
\bigoplus
\big(\stackrel{p^+_-}\oplus V_{min;-}^{r,s;+}\big)\bigoplus \big(\stackrel{p^-_-}\oplus V_{min;-}^{r,s;-}\big).
\end{equation}
}
\end{itemize} 

Since the involutions in $PI_{r,s}$ are symmetric, 
the eigenspaces of involutions are mutually orthogonal. 
The involutions commute, therefore, they decompose the eigenspaces of
other involutions into smaller (eigen)-subspaces. 
We give an example, that is crucial for the paper.

\begin{example}\label{ex:inv44}
The set $PI_{\mu,\nu}$ for $(\mu,\nu)\in\{(8,0),(0,8),(4,4)\}$ is given by
\begin{equation*}\label{mcpi2}
T_{1}=J_{\zeta_1}J_{\zeta_2}J_{\zeta_3}J_{\zeta_4},\, T_{2}=J_{\zeta_1}J_{\zeta_2}J_{\zeta_5}J_{\zeta_6},
\,T_{3}=J_{\zeta_1}J_{\zeta_2}J_{\zeta_7}J_{\zeta_8},\,T_{4}=J_{\zeta_1}J_{\zeta_3}J_{\zeta_5}J_{\zeta_7}.
\end{equation*}
The module $V_{min}^{\mu,\nu}$
is decomposed into $16$ one dimensional common eigenspaces of 
four involutions $T_{i}$. Let
$v\in E_{\mu,\nu}$
and $|\la v, v\ra_{V_{min}^{\mu,\nu}}|=1$. Then other common
eigenspaces
are spanned by $J_{z_{i}}v$, $i=1,\ldots, 8$, 
and $J_{z_1}J_{z_j}v$, $j=2,\ldots,8$. Hence we have
\begin{equation}\label{eq:decom08}
V_{min}^{\mu,\nu}=E_{\mu,\nu}\bigoplus_{i=1}^{8}\,J_{\zeta_i}(E_{\mu,\nu})
\bigoplus_{j=2}^{8} J_{\zeta_1}J_{\zeta_{j}}(E_{\mu.\nu}).
\end{equation}
The value $\la v, v\ra_{V_{min}^{\mu,\nu}}$ can be $\pm 1$
according to the admissible scalar product, however 
we may assume $\la v, v\ra_{V_{min}^{\mu,\nu}}=1$, see~\cite[Example 1, Lemma 3.2.5]{FM2}. 
\end{example}


\section{Pseudo $H$-type Lie algebras}



\subsection{Definitions of pseudo $H$-type Lie algebras and their Lie groups}\label{subsec:Defin.Pseudo}


Let $(U,\la\cdot\,,\cdot\ra_{U})$ be an admissible module of a Clifford
algebra $\Cl_{r,s}$.
We define a vector valued skew-symmetric bi-linear form 
\[
\begin{array}{cccccc}
[\cdot\,,\cdot]\colon &U\times U &\to &\mathbb{R}^{r,s}
\\
&(x,y)&\longmapsto &[x,y]
\end{array}
\]
by the relation
\begin{equation}\label{def of skew-symmetric bi-linear form}
\la J_{z}x,y\ra_{U}=\la z,[x,y]\ra_{r,s}.
\end{equation}

\begin{definition}{\em \cite{Ci}}
The space $U\oplus \mathbb{R}^{r,s}$ endowed with the Lie bracket
\[
[(x,z),(y,w)]=(0,[x,y])
\]
is called a pseudo $H$-type Lie algebra and it is denoted by $\n_{r,s}(U)$.
\end{definition}
A pseudo $H$-type Lie algebra $\n_{r,s}(U)$ is 2-step nilpotent, the space $\mathbb{R}^{r,s}$ is the centre, and the direct sum $U\oplus\mathbb{R}^{r,s}$ is orthogonal with respect to $\la\cdot\,,\cdot\ra_{U}+\la\cdot\,,\cdot\ra_{r,s}$.

The Baker-Campbell-Hausdorff formula allows us to define the Lie group structure on the space
$U\oplus\mathbb{R}^{r,s}$ by
\[
(x,z)*(y,w)=\bigr(x+y,z+w+\frac{1}{2}[x,y]\bigl).
\]
The Lie group is denoted by ${N}_{r,s}(U)$ and is called the pseudo $H$-type Lie group. 
Note that the scalar product $\la\cdot\,,\cdot\ra_{U}$ is implicitly
included in the definitions of the $H$-type Lie algebra and the
corresponding Lie group.
In general, the Lie algebra structure might change 
if we replace the admissible scalar product on $U$, see~\cite{AFMV, E, Eber03}. 


\subsection{General structure of the group $\Aut(\n_{r,s}(U))$}\label{GeneralAut}


In the present section all the matrix groups are considered over the field $\mathbb R$. Let $\n=(U\oplus \z,[.\,,.])$ be a real 2-step nilpotent graded Lie algebra with the centre $\z$ and $\Aut(\n)$ be a group of automorphisms of this Lie algebra. We use the identification $U\cong\mathbb R^n$ and $\z\cong\mathbb R^m$. An automorphism has to preserve the centre and therefore an element $\phi\in\Aut(\n)$ has to be of the form
$$
\phi=\begin{pmatrix}A&0\\B&C\end{pmatrix},\quad A\in\GL(n,\R),\ \ C\in\GL(m,\R),\ \ B\in {\rm Hom}(\mathbb R^n,\mathbb R^m).
$$
where $C([u,v])=[Au,Av]$. The subgroup 
$$
\mathcal B(\n)=\Big\{ \begin{pmatrix}t\Id_n&0\\B&t^2\Id_m\end{pmatrix},\quad B\in {\rm Hom}(\mathbb R^n,\mathbb R^m),\ \ t\neq 0\Big\}
$$ 
is a normal subgroup of $\Aut(\n)$. 
The factor group 
$$
\Aut(\n)/\mathcal B(\n)=\Big\{\begin{pmatrix}A&0\\0&C\end{pmatrix},\ A\in \SL(n,\R),\ C\in \GL(m,\R),\ C([u,v])=[Au,Av]\Big\}
$$
is a subgroup of $\Aut(\n)$ will be denoted by $\C(\n):=\Aut(\n)/\mathcal B(\n)$. Thus the group $\Aut(\n)$ is a semi-direct product
  of $\mathcal{B}(\n)$ and $\C(\n)$, and it is enough to determine the
  group $\C(\n)$. From now on we write $A\oplus C$ for elements of the group $\C(\n)$. The group 
$$
\Aut^0(\n)=\Big\{A\oplus \Id_m,\quad A\in\SL(n,\R),\quad [u,v]=[Au,Av]\Big\}
$$
is a normal subgroup in $\C(\n)$. 

Let us assume now, that $\n$ is a pseudo $H$-type Lie algebra
$\n_{r,s}(U)=U\oplus\z$ 
with $\z=\mathbb R^{r,s}$, $r+s=m$. 
Then the group  $\Aut^0(\n_{r,s}(U))$ can be written as follows
\begin{equation}\label{eq:Aut0}
\Aut^0(\n_{r,s}(U))=\Big\{A\oplus\Id_m,\ A\in \SL(n,\R),\ A^{\tau}J_zA=J_z\ \text{for any}\ z\in\mathbb R^{r,s}\Big\}.
\end{equation}

\begin{lemma}\cite[Theorem 2]{FM1}\label{lem:orth_rs}
The subgroup of the maps $C\in\GL(m,\R)$ such that $A\oplus C\in \C(\n_{r,s}(U))$ is contained in $\Orth(r,s)$, $r+s=m$. 
\end{lemma}
Due to Lemma~\ref{lem:orth_rs}, we conclude that 
$$
\C(\n_{r,s}(U))=\Big\{A\oplus C,\ A\in \SL(n,\R),\ C\in \Orth(r,s),\ A^{\tau}J_zA=J_{C^{\tau}(z)}\ \ z\in\mathbb R^{r,s}\Big\},
$$
where the transpositions $A^{\tau}$ and $C^{\tau}$ are understood with respect to the corresponding scalar products on $U$ and on $\mathbb R^{r,s}$.
In the next step we show that the map
$$
\C(\n_{r,s}(U))\to\Orth(r,s):\quad A\oplus C\mapsto C
$$
is surjective.
To achieve the goal we recall the notion of the $\Pin(r,s)$ group. The map
$$
\mathbb{R}^{r,s}\ni z\mapsto -z\in \mathbb{R}^{r,s}\subset \Cl_{r,s}
$$ is extended to the Clifford algebra automorphism $\alpha\colon\Cl_{r,s} \to \Cl_{r,s}$
by the universal property of the Clifford algebras. Note two properties of the map $\alpha$:
$$
\alpha^2=\Id,\qquad \alpha(\phi_1\phi_2)=\alpha(\phi_1)\alpha(\phi_2),\quad \phi_1,\phi_2\in\Cl_{r,s}.
$$
We denote by $\Cl_{r,s}^{\times}$ the group of invertible elements in
$\Cl_{r,s}$ and in particular ${\mathbb{R}^{r,s}}^{\times}=\{z\in\mathbb R^{r,s}\vert\ \la z,z\ra_{r,s}\not=0\}$. The representation 
$\widetilde{\Ad}\colon{\mathbb R^{r,s}}^{\times}\to\End(\mathbb R^{r,s})$, 
is defined as
$$
\widetilde{\Ad}_{z}(w)=-z wz^{-1}
=\left(w-2\frac{\la w,z\ra_{r,s}}{\la z,z\ra_{r,s}}z\right)\in\mathbb R^{r,s}\quad 
\text{for}\quad w\in\mathbb R^{r,s},\ \ z\in {\mathbb{R}^{r,s}}^{\times}.
$$
The map $\widetilde{\Ad}_{z}\colon \mathbb R^{r,s}\to\mathbb R^{r,s}$ is the reflection of the vector $w\in\mathbb R^{r,s}$ with respect to the hyperplane orthogonal to the vector $z\in\mathbb R^{r,s}$.
Then it extends to the twisted adjoint representation 
$\widetilde{\Ad}\colon\Cl_{r,s}^{\times}\to\GL(\Cl_{r,s})$ by setting
\begin{equation}\label{twisted adjoint representation}
\Cl_{r,s}^{\times}\ni \varphi\longmapsto \widetilde{\Ad}_{\varphi},\quad
\widetilde{\Ad}_{\varphi}(\phi)=\alpha(\varphi) \phi \varphi^{-1}, \quad \phi\in \Cl_{r,s}.
\end{equation}
The map $\widetilde{\Ad}_{z}$ for $z\in{\mathbb{R}^{r,s}}^{\times}$, leaving the space
$\mathbb{R}^{r,s}\subset \Cl_{r,s}$ invariant, is also an isometry:
$
\la\widetilde{\Ad}_{z}(w),\widetilde{\Ad}_{z}(w)\ra_{r,s}=\la
w, w\ra_{r,s}$. 
Moreover, the properties to preserve the space $\mathbb R^{r,s}$ and the bilinear symmetric form $\la.\,,.\ra_{r,s}$ are fulfilled for the group
\begin{equation}\label{eq:Cliff}
{\rm P}(\mathbb R^{r,s})=\{v_{1}\cdots v_{k}\in \Cl_{r,s}^{\times}\mid\ \la v_{i},v_{i}\ra_{r,s}\neq 0\}.
\end{equation}
The map $\widetilde{\Ad}\colon {\rm P}(\mathbb R^{r,s})\to \Orth(r,s)$ is a surjective homomorphism~\cite[Theorem 2.7]{LawMich}.
It particularly implies $\widetilde{\Ad}_{\varphi^{-1}}=\widetilde{\Ad}_{\varphi}^{\tau}$.
Subgroups of $P(\mathbb R^{r,s})\subset \Cl_{r,s}^{\times}$ defined by
\begin{align*}
&\Pin(r,s)=\{v_{1}\cdots v_{k}\in \Cl_{r,s}^{\times}\mid\ \la v_{i},v_{i}\ra_{r,s}=\pm 1\},
\\
&\Spin(r,s)=\{v_{1}\cdots v_{k}\in \Cl_{r,s}^{\times}\mid\ k~\text{is even},\la v_{i},v_{i}\ra_{r,s}=\pm 1\},
\end{align*}
are called pin and spin groups, respectively. Recall the following.
\begin{proposition}\label{prop:coverings}\cite{ABS, LawMich}
The map 
$
\widetilde{\Ad}\colon\Pin(r,s)\to \Orth(r,s) 
$
is the double covering map.
\end{proposition}
Let us introduce the norm mapping $N\colon \Cl_{r,s}\to\Cl_{r,s}$ defined by $N(\phi)=\phi\cdot\alpha(\phi^T)$. It is easy to see that $N(z)=\la z,z\ra_{r,s}$ for any $z\in\mathbb R^{r,s}$.

\begin{proposition}\label{lem:phiArbitrary} 
Let $J\colon\Cl_{r,s}\to\End(U)$ be a Clifford algebra representation and $\varphi\in \Pin(r,s)$. 
Then the map $\mathcal P\colon \Pin(r,s)\to \C(\n_{r,s}(U))$ defined by
$$
\varphi\mapsto \mathcal P(\varphi)=\begin{pmatrix}J_{\varphi}&0\\ 0&(-1)^nN(\varphi)\widetilde{\Ad}_{\varphi}\end{pmatrix},
$$
where $N(\varphi)=\varphi\cdot \alpha(\varphi^T)$, is the group homomorphism.
\end{proposition}
\begin{proof}
First we show that $P(\varphi)\in \C(\n_{r,s}(U))$ that is 
$$
J_{\varphi}^{\tau}J_zJ_{\varphi}=J_{(-1)^nN(\varphi)\widetilde{\Ad}_{\varphi}^{\tau}}.
$$ 
Let $\varphi\in\Pin(r,s)$. 
Then $\widetilde{\Ad}_{\varphi}^\tau\in \Orth(r,s)$. 
Moreover, $\widetilde{\Ad}_{\varphi}^\tau(z)=\widetilde{\Ad}_{\varphi^{-1}}(z)=\alpha(\varphi^{-1})z\varphi$. Thus for any $\varphi=\prod_{k=1}^n x_k\in\Pin(r,s)$ we obtain
$$
\varphi^T=(x_1\cdot\ldots\cdot x_n)^T=(x_n\cdot\ldots\cdot x_1)\quad\text{and}\quad 
N(\varphi)=\varphi\cdot \alpha(\varphi^T)=\prod_{k=1}^{n}\langle x_k,x_k\rangle_{r,s}.
$$ 
The properties of the maps $*^T$ and $N$ can be found in~\cite[Page 15]{LawMich}. Then since $x_k^{-1}=\frac{\alpha(x_k)}{\langle x_k,x_k\rangle_{r,s}}$, $k=1,\ldots,n$, and $J_{x_k}^{\tau}=-J_{x_k}$ we have 
$\alpha(\varphi^{-1})=N^{-1}(\varphi)\varphi^T$ and $J_{\varphi^T}=(-1)^nJ_{\varphi}^{\tau}$.
Thus
$$
J_{\widetilde{\Ad}_{\varphi}^{\tau}(z)}=J_{\widetilde{\Ad}_{\varphi^{-1}}(z)}=J_{\alpha(\varphi^{-1})}J_zJ_{\varphi}=\frac{(-1)^n}{N(\varphi)}J_{\varphi}^{\tau}J_zJ_{\varphi}.
$$
It proves the proposition.
\end{proof}

\begin{proposition}\label{automorohisms from Clifford algebra} 
Let $J\colon\Cl_{r,s}\to\End(U)$ be a Clifford algebra representation and $\varphi\in \Pin(r,s)$. The both lines in the following diagram
\begin{equation}\label{CD11}
\begin{CD}
\{\Id\}@>>>\Aut^0(\n_{r,s}(U))@>{\iota}>> \C(\n_{r,s}(U)) @>{\psi
}>> \Orth(r,s)@>>>{\{\Id\}}\\
@.@AAA{\mathcal{P}}@AAA@AA{\Id}A\\
\{\Id\}@>>>\mathbb Z_{2}@>>> \Pin(r,s) @ >{\widetilde{\Ad}}>>\Orth(r,s)@>>>{\{\Id\}}.
\end{CD}
\end{equation}
are short exact sequences. The kernel $\Aut^0(\n_{r,s}(U))$ is defined in~\eqref{eq:Aut0}. 
\end{proposition}

\begin{proof}
It is well known fact that the second line is a short exact sequence, see, for instance~\cite{LawMich}. Let $C\in\Orth(r,s)$ and $\varphi$ be any element of $\Pin(r,s)$ such that $\widetilde \Ad_{\varphi}=C$. Then to any $C\in \Orth(r,s)$ there is $\psi^{-1}(C)=\mathcal P(\varphi)\in \C(\n_{r,s}(U))$, given by Proposition~\ref{lem:phiArbitrary}. It shows that the map $\psi$ is surjective.
\end{proof}

\begin{lemma}\label{lem:semidirect} Let
\begin{equation}\label{eq:exactGen}
\{\Id\}\ \ \longrightarrow\ \ N\ \ {\overset {\iota }{\longrightarrow}} \ \ G\ \  {\overset {\psi }{\longrightarrow}} \ \ H\ \ \longrightarrow\ \ \{\Id\}
\end{equation}
be a short exact sequence of groups.
We assume that $K$ is a subgroup in $G$ such that $\psi\vert_{K}$ is surjective. Then there is a group homomorphism $\rho\colon N\rtimes_{\phi}K\to G$ with $
\ker \rho=\{(n,n^{-1})\mid\ n\in K\cap N\}$.
\end{lemma}
\begin{proof}
Since $N$ is a normal subgroup of $G$, the subgroup $K$ acts on $N$ by conjugation
\[
\phi\colon K\to \Aut(N),\quad  \phi_{k}(n)=knk^{-1},\quad\text{for}\quad n\in N,\  k\in K.
\]
Then we have a surjective group homomorphism
\[
\rho:N\rtimes_{\phi} K\ni (n,k)\mapsto nk\in G,
\]
In fact
$
\rho( (n,k)\cdot (n',k'))=
nkn'k^{-1}kk'=
\rho((n,k))\rho((n',k'))$.
The kernel of $\rho$
is 
\[
\ker \rho=\{(n,k)\mid\ nk=e,\ n\in N,\ k\in K\}
=\{(n,n^{-1})\mid\ n\in K\cap N\},
\]
where $e$ is the unit element in $G$. Consequently,
  $(N\rtimes_{\phi} K)/\ker \rho\cong G$.
\end{proof} 

We set $G=\C(\n_{r,s}(U))$, $K=\mathcal P\big(\Pin(r,s)\big)$ and $N=\Aut^0(\n_{r,s}(U))$ in Lemma~\ref{lem:semidirect}. The kernel $\rho$ consists of $\Phi\in
\Aut^0(\n_{r,s}(U))\cap \mathcal P\big(\Pin(r,s)\big)$.
Now we determine the order of $\Aut^0(\n_{r,s}(U))\cap\mathcal P\big(\Pin(r,s)\big)$ for different types of admissible modules and all the pairs $(r,s)$. 

\begin{theorem}\label{th:kernel} In the notations above, we have 

$(1)\ \Aut^0(\n_{r,s}(U))\cap \mathcal P\big(\Pin(r,s)\big)=
\{\pm\Id\oplus \Id\}$
in the following cases
\begin{itemize}
\item[(1a)]{$r$ is even, $s$ is arbitrary;}
\item[(1b)]{$r=1\mod 4$, $s=1,2 \mod 4$;}
\item[(1c)]{$r=3\mod 4$, $s=0,3 \mod 4$ and the admissible module is isotypic;}
\end{itemize}

$(2) \ \Aut^0(\n_{r,s}(U))\cap \mathcal P\big(\Pin(r,s)\big)=\{
\pm\Id\oplus\Id;\ 
\pm J_{\Omega^{r,s}}\oplus\Id\}$
in the following cases
\begin{itemize}
\item[(2a)]{$r=1\mod 4$, $s=0,3 \mod 4$;}
\item[(2b)]{$r=3\mod 4$, $s=1,2 \mod 4$;}
\item[(2c)]{$r=3\mod 4$, $s=0,3 \mod 4$ and the admissible module is non-isotypic.}
\end{itemize}
\end{theorem}
\begin{proof}
To prove the theorem we need to find $\phi\in\Pin(r,s)$ such that
$$
\Psi(\phi)=(-1)^nN(\phi)\widetilde{\Ad}_{\phi}=\Id.
$$
Then $\pm J_{\phi}\oplus \Id$ will belong to $\Aut^0(\n_{r,s}(U))\cap \mathcal P\big(\Pin(r,s)\big)$.
Note that $\Psi(\pm\mathbf 1)=\Id$. We also note that 
$$
N(\Omega^{r,s})=\alpha(\Omega^{r,s})(\Omega^{r,s})^T=(-1)^s
\quad\text{and}\quad
z\Omega^{r,s}=(-1)^{r+s-1}\Omega^{r,s}z\ \text{for any}\ z\in\mathbb R^{r,s}.
$$
Thus
$$
(-1)^{r+s}N(\Omega^{r,s})\widetilde\Ad_{\Omega^{r,s}}z=(-1)^{2r+4s}
(-1)^{r-1}z\Omega^{r,s}(\Omega^{r,s})^{-1}=(-1)^{r-1}z
$$
Hence $\Psi (\Omega^{r,s})=\Id$ for odd values of $r$ and arbitrary values of $s$. Thus, if $r$ is even then for arbitrary $s$ the only elements in $\Aut^0(\n_{r,s}(U))\cap \mathcal P\big(\Pin(r,s)\big)$ are $\pm\Id\oplus\Id$. Moreover, in this case all the modules are isotypic. This shows $(1a)$.

Before we proceed, we remind some properties of the volume form:
\begin{equation}\label{eq:Omega2}
(\Omega^{r,s})^2=
\begin{cases}
(-1)^s,\quad&\text{if}\quad r+s=3,4\mod 4,
\\
(-1)^{s+1},\quad&\text{if}\quad r+s=1,2\mod 4;
\end{cases}
\end{equation}
We need to check the values $r=1,3\mod 4$.
\\

Let $r=1\mod 4$. In this case all admissible modules are isotypic. Moreover~\eqref{eq:Omega2} implies
\begin{equation*}\label{eq:Om2}
(\Omega^{r,s})^2=
\begin{cases}
1,\quad&\text{if}\quad s=1,2\mod 4,
\\
-1,\quad&\text{if}\quad s=0,3\mod 4;
\end{cases}
\end{equation*}
Thus, if $r=1\mod 4$ and $s=1,2\mod 4$, then we have $J_{\Omega^{r,s}}=\pm\Id$ and it proves (1b). In the case $r=1\mod 4$ and $s=0,3\mod 4$ we obtain (2a).
\\

Let $r=3\mod 4$. In this cases we need to distinguish isotypic and non-isotypic admissible modules. The property~\eqref{eq:Omega2} implies
\begin{equation*}\label{eq:Om2}
(\Omega^{r,s})^2=
\begin{cases}
1,\quad&\text{if}\quad s=0,3\mod 4,
\\
-1,\quad&\text{if}\quad s=1,2\mod 4;
\end{cases}
\end{equation*}
Thus if $r=3\mod 4$ and $s=1,2\mod 4$ we obtain (2b). If $r=3\mod 4$, $s=0,3\mod 4$ and module is isotypic then $J_{\Omega^{r,s}}=\pm\Id$, that shows (1c). In the case $r=3\mod 4$, $s=0,3\mod 4$ with a non-isotypic module we obtain (2c).

At the end we notice that the cases~(1a), (1c), and (2c) contains the result of~\cite{Saal}. 
\end{proof}

We conclude that any element of 
$\C(n_{r,s}(U))$ has the form
$AJ_{\varphi}\oplus (-1)^n N(\varphi)\widetilde{\Ad}_{\varphi}$.
Thus the only thing that we need to define is the subgroup $\mathbb A\in \SL(n,\R)$ containing maps $A$ such that \begin{equation}\label{isomorphism relation}
A^\tau J_z A=J_z\quad\text{for all}\quad z\in\mathbb R^{r,s}. 
\end{equation}  
 
\subsection{Relation between the structure of involutions $PI_{r,s}$ and $\Aut^0\big(\n_{r,s}(U)\big)$}


In this section we show that the isomorphism group is closely related to the structure of the set of involutions $PI_{r,s}^*$ of types (1)-(3) and it is defined by its behaviour on the common 1-eigenspace $E_{r,s}^*$.
We give, in addition to~\eqref{isomorphism relation}
some relations between the automorphism $A$ and the Clifford actions $J_z$ that we use in the present work. The proof follows from~\eqref{isomorphism relation} by induction and can be found in~\cite[Lemma 3]{FM1} for the product of any number of $J_{z_k}$.

\begin{lemma}\label{relation between volume form}
Let $\{z_i\}_{i=1}^{r+s}$ be an orthonormal basis for
$\mathbb{R}^{r,s}$
and let $\Phi=A\oplus \Id\in\Aut^0(n_{r,s}(U))$. 
Then the following relations hold:
\begin{equation}
AJ_{z_k}J_{z_l}=J_{z_k}J_{z_l}A,\qquad
AJ_{z_k}J_{z_l}J_{z_m}J_{z_n}=J_{z_k}J_{z_l}J_{z_m}J_{z_n}A.
\end{equation}
\begin{equation}
A^{\tau}J_{z_j}J_{z_k} J_{z_l}  A=J_{z_j}J_{z_k}J_{z_l}.
\end{equation}
\end{lemma}

\begin{lemma}\label{lem:2to1}
Let $\{z_i\}_{i=1}^{r+s}$ be an orthonormal basis for
$\mathbb{R}^{r,s}$ and $U$ an admissible module. If a linear map $A\colon U\to U$ satisfies the conditions
\begin{equation}\label{eq:k0}
\begin{array}{ll}
&A^{\tau}J_{z_{k_0}}A=J_{z_{k_0}}\quad\text{for one index}\quad {k_0}\in\{1,\ldots r+s\},
\quad\text{and}
\\
&AJ_{z_{k_0}}J_{z_l}=J_{z_{k_0}}J_{z_l}A \quad\text{for all indices}\quad l=1,\ldots r+s,
\end{array}
\end{equation}
then 
$\Phi=A\oplus \Id\in\Aut^0(n_{r,s}(U))$.
\end{lemma}
\begin{proof}
We only need to show~\eqref{isomorphism relation} for all $z=z_l$ for $l=1,\ldots r+s$. Let~\eqref{eq:k0} is fulfilled.  
Then
$$
A^{\tau}J_{z_l}A=\pm A^{\tau}J_{z_{k_0}}J_{z_{k_0}}J_{z_l}A=\pm A^{\tau}J_{z_{k_0}}AJ_{z_{k_0}}J_{z_l}=\pm J_{z_{k_0}}^2J_{z_l}=J_{z_l}.
$$
\end{proof}

\begin{corollary}
$\Phi=A\oplus\Id\in \Aut^0(n_{r,s}(U))$ if and only if~\eqref{eq:k0} holds.
\end{corollary}
Let $(V^{r,s}_{min}, J)$ be a minimal admissible module of $\Cl_{r,s}$.
Let $P\colon V^{r,s}_{min}\to V^{r,s}_{min}$ be an involution from the set $PI_{r,s}^*$ that is the product of four generators. 
We denote by $E^{k}_{P}$, $k\in\{1,-1\}$
the eigenspace of the involution $P$ with the eigenvalue $k=\pm 1$.
In order to denote the intersection of
eigenspaces of several  
involutions $P_l\in PI_{r,s}^*$, $l=1,\ldots,N=\#(PI_{r,s}^{*})$, we use the multi-index 
$I=(k_1,\ldots,k_N), ~k_{l}=\pm 1$ and 
write $E^I=\cap_{l=1}^{N}E^{k_l}_{P_l}$. 
Assume that 
$\Phi=A\oplus \Id\in\Aut^0(\n_{r,s}(V_{min}^{r,s}))$. Then

\noindent 1. $A=\oplus A_{I}$, 
where $A_{I}\colon E^{I}\to E^{I}$ for any choice of $I=(k_1,\ldots,k_N)$;

\noindent 2. if $J_{z_j},\ J_{z_j}J_{z_k},\  J_{z_j}J_{z_k}J_{z_m} \colon E^{I}\to E^{I}$ for some $I$, then 
$$
A_{I}J_{z_j}(x_I)=J_{z_j}(A^{\tau}_{I})^{-1}(x_{I}),\quad
A_{I}J_{z_j}J_{z_k}J_{z_m}(x_I)=J_{z_j}J_{z_k}J_{z_m}(A^{\tau}_{I})^{-1}(x_{I})
$$
$$
A_{I}J_{z_j}J_{z_k}(x_I)=J_{z_j}J_{z_k}A_{I}(x_I), 
$$
and for the transposed operators
$$ 
A^{\tau}_{I}J_{z_j}(x_I)=J_{z_j} (A_{I})^{-1}(x_I), \quad
A^{\tau}_{I}J_{z_j}(x_I)J_{z_k}J_{z_m}=J_{z_j}J_{z_k}J_{z_m} (A_{I})^{-1}(x_I),
$$
$$
A^{\tau}_{I}J_{z_j}J_{z_k}(x_I)=J_{z_j}J_{z_k} A^{\tau}_{I}(x_I).
$$

\begin{proof}
The first statement follows from the fact that $AP_l=P_l A$ for all $l=1,\ldots,N$.

The second statement is the direct consequence of~\eqref{isomorphism relation} and Lemma~\ref{relation between volume form}.
\end{proof}

Thus the  
construction of the map $A\colon V^{r,s}_{min}\to V^{r,s}_{min}$ can be reduced to the
construction of 
the maps $A_I\colon E^I\to E^I$ and setting $A=\oplus A_I$. Theorem~\ref{th:general} states that, under some conditions, the construction of all maps $A_I$ can be obtained from the map $A_1\colon E^1\to E^1$, where we denote $E^1=\bigcap_{l=1}^{N}E^{1}_{P_l}$. Note that $E^1$ it is exactly the subspace $E_{r,s}^*\subset V^{r,s}_{min}$ that is the common $1$-eigenspace of involutions from $PI_{r,s}^*$ that are of types (1)-(3).

\begin{theorem}\label{th:general} Under the previous notations we assume that 

\noindent $(a)$ there are maps $G_I\colon E^1\to E^I$ for all multi-indices $I$ of the form either $G_I=J_{z_i}$ or $J_{z_i}J_{z_k}$ for some $i,k=1,\ldots,r+s$, and

\noindent $(b)$ there exists a linear map $A_1\colon E^1\to E^1$ such that if $J_{z_j}, J_{z_j}J_{z_k}, J_{z_j}J_{z_k}J_{z_m}\colon E^1\to E^1$, then the map $A_1$ satisfies
\begin{equation}\label{eq:A11}
A_{1}J_{z_j}=J_{z_j}(A^{\tau}_{1})^{-1},\quad 
A_{1}J_{z_j}J_{z_k}J_{z_m}=J_{z_j}J_{z_k}J_{z_m}(A^{\tau}_{1})^{-1}
\end{equation}
and the same for any other product of odd number of generators $J_{z_l}$, leaving the space $E^1$ invariant; and also
\begin{equation}\label{eq:A111}
A_{1}J_{z_j}J_{z_k}=J_{z_j}J_{z_k}A_1
\end{equation}
and the same for any other product of even number of generators $J_{z_l}$, leaving the space $E^1$ invariant.

Then the map $A\colon V^{r,s}_{min}\to V^{r,s}_{min}$, $A=\oplus A_{I}$ with
$A_I\colon E^I\to E^I$ such that 
\begin{equation}~\label{eq:AII}
A_I=
\begin{cases}
G_I(A_1^{-1})^{\tau}G_I^{-1},\ \ &\text{if}\ \  G_I=J_{z_i}\ \ \text{for some}\ \ i=1,\ldots,r+s, 
\\
G_IA_1G_I^{-1},\ \ &\text{if}\ \  G_I=J_{z_i}J_{z_k}\ \ \text{for some}\ \ i,k=1,\ldots,r+s.
\end{cases}
\end{equation}
uniquely defines the automorphism $\Phi=A\oplus \Id\in\Aut^0(n_{r,s}(V_{min}^{r,s}))$.
\end{theorem}

\begin{proof}
The spaces $E^I$ are mutually orthogonal because all the involutions in $PI_{r,s}^*$ are symmetric.
Thus $V^{r,s}_{min}=\oplus E^I$, where the direct sum is orthogonal. 
For the convenience we also write the maps defining $A_I^{\tau}$:
\begin{equation}~\label{eq:AIt}
A_I^{\tau}=
\begin{cases}
G_IA_1^{-1}G_I^{-1},\ \ &\text{if}\ \  G_I=J_{z_i},
\\
G_IA_1^{\tau}G_I^{-1},\ \ &\text{if}\ \  G_I=J_{z_i}J_{z_k}.
\end{cases}
\end{equation}
Then we set $A=\oplus A_I$. We only need to check the condition $AJ_{z_{j_0}}A^{\tau}=J_{z_{j_0}}$ for any $z_{j_0}$ in the orthonormal basis for $\mathbb R^{r,s}$. 

We choose $y\in V^{r,s}_{min}=\oplus E^I$. Then we write $y=\oplus y_I$ with $y_I\in E^I$. Thus we distinguish the cases when the map $G_I$ is the product of an odd or even number of maps $J_{z_i}$. Moreover, we find a multi-index $K$ for the multi-index $I$, such that $G_K^{-1}J_{z_{j_0}}G_I$ leaves the space $E^1$ invariant. Since $G_K$ can also be the product of an even or odd number of $J_{z_k}$, we distinguish the following cases: $AJ_{z_{j_0}}A^{\tau}y_I=A_KJ_{z_{j_0}}A^{\tau}_Iy_I$
\begin{eqnarray*}
&=&
\begin{cases}
G_K(A_1^{-1})^{\tau}G_K^{-1}J_{z_{j_0}}G_IA_1^{-1}G_I^{-1}y_I\ \ &\text{if}\ G_I=J_{z_i},\ G_K=J_{z_l},
\\
G_KA_1G_K^{-1}J_{z_{j_0}}G_IA_1^{-1}G_I^{-1}y_I\ \ &\text{if}\ G_I=J_{z_i},\ G_K=J_{z_k}J_{z_l},
\\
G_K(A_1^{-1})^{\tau}G_K^{-1}J_{z_{j_0}}G_IA_1^{\tau}G_I^{-1}y_I\ \ &\text{if}\ G_I=J_{z_i}J_{z_m},\ \ G_K=J_{z_l},
\\
G_KA_1G_K^{-1}J_{z_{j_0}}G_IA_1^{\tau}G_I^{-1}y_I\ \ &\text{if}\ G_I=J_{z_i}J_{z_m},\ \ G_K=J_{z_k}J_{z_l},
\end{cases}
\end{eqnarray*}
by definitions~\eqref{eq:AII} and~\eqref{eq:AIt}  of $A_K$ and $A_I^{\tau}$. We only check the first condition, since the others can be verified similarly. The condition that $G_K^{-1}J_{z_{j_0}}G_I$ leaves the space $E^1$ invariant, reads as
$
(A_1^{-1})^{\tau}G_K^{-1}J_{z_{j_0}}G_IA_1^{-1}=G_K^{-1}J_{z_{j_0}}G_I$.
Indeed from~\eqref{eq:A11} we have 
\begin{eqnarray*}
(A_1^{-1})^{\tau}G_K^{-1}J_{z_{j_0}}G_IA_1^{-1}&=&
(A_1^{-1})^{\tau}J_{z_l}^{-1}J_{z_{j_0}}J_{z_i}A_1^{-1}=J_{z_l}^{-1}J_{z_{j_0}}J_{z_i}=G_K^{-1}J_{z_{j_0}}G_I.
\end{eqnarray*}
We calculate
\begin{eqnarray*}
G_K(A_1^{-1})^{\tau}G_K^{-1}J_{z_{j_0}}G_IA_1^{-1}G_I^{-1}y_I=G_KG_K^{-1}J_{z_{j_0}}G_IG_I^{-1}=J_{z_{j_0}}.
\end{eqnarray*}
Thus, $AJ_{z_{j_0}}A^{\tau}y_I=J_{z_{j_0}}y_I$, which proves the theorem.

Now we show the uniqueness. Let us assume that $G_{I}, G_{\tilde I}\colon E^1\to E^I$ and both $G_{I}, G_{\tilde I}$ are the product of even number of $J_{z_k}$. Then
$
A_{1}G_I=G_IA_1$, and $A_{1}G_{\tilde I}=G_{\tilde I}A_1$.
It implies
$$
A_{I}\circ A_{\tilde I}^{-1}=G_IA_1G_I^{-1}G_{\tilde I}A_1^{-1}G_{\tilde I}^{-1}=G_IG_I^{-1}G_{\tilde I}A_1A_1^{-1}G_{\tilde I}^{-1}=\Id,
$$
since $G_I^{-1}G_{\tilde I}$ is the product of even number of $J_{z_k}$ and we can apply~~\eqref{eq:A111}. Thus $A_{I}= A_{\tilde I}^{-1}$.

Let now $G_{I}, G_{\tilde I}\colon E^1\to E^I$ and both of them is the product of odd numbers of generators. Then by making use of~\eqref{eq:AII} and~\eqref{eq:A111} we obtain
$$
A_{I}\circ A_{\tilde I}^{-1}=G_I(A_1^{\tau})^{-1}G_I^{-1}G_{\tilde I}A_1^{\tau}G_{\tilde I}^{-1}=
G_I(A_1^{\tau})^{-1}A_1^{\tau}G_I^{-1}G_{\tilde I}G_{\tilde I}^{-1}=\Id
$$
since $G_I^{-1}G_{\tilde I}$ is the product of even number of generators.

Finally, if $G_{I}\colon E^1\to E^I$ is the product of odd number of generators and $G_{\tilde I}\colon E^1\to E^I$ is the product of even number of generators, then we obtain 
$$
A_{I}\circ A_{\tilde I}^{-1}=G_I(A_1^{\tau})^{-1}G_I^{-1}G_{\tilde I}A_1^{-1}G_{\tilde I}^{-1}=
G_IG_I^{-1}G_{\tilde I}A_1A_1^{-1}G_{\tilde I}^{-1}=\Id
$$
by~\eqref{eq:AII}. Here we used the fact that $G_I^{-1}G_{\tilde I}$ is the product of odd number of generators and~\eqref{eq:A11}.
\end{proof}

 
\subsection{Classification of pseudo $H$-type Lie algebras $\n_{r,s}(U)$}


We start from the necessary condition of isomorphism between two $H$-type Lie algebras.
\begin{theorem}\label{property of isomorphism}\cite[Theorem 2]{FM1}.
Let $(V^{r,s},\, \la\cdot\,,\,\cdot\ra_{V^{r,s}})$  and $(V^{\tilde r,\tilde s},\,\la\cdot\,,\,\cdot\ra{V^{\tilde{r},\tilde{s}}})$ 
be admissible modules 
of the Clifford algebras $\Cl_{r,s}$ and $\Cl_{\tilde r,\tilde s}$,
respectively. 
Assume that $r+s=\tilde r+\tilde s$, $\dim (V^{r,s})=\dim(V^{\tilde r,\tilde s})$, 
and that the Lie algebras  
$\n_{r,s}(V^{r,s})$ and $\n_{\tilde r,\tilde s}(V^{\tilde r,\tilde s})$ are isomorphic. 
Then, either $(r,s)=(\tilde r,\tilde s)$ or $(r,s)=(\tilde s,\tilde r)$. 
\end{theorem}

The classification of the pseudo $H$-type algebras $\n_{r,s}(V^{r,s}_{min})$, constructed from the minimal admissible modules was done in~\cite{FM1}. We summarise the results of the classification in Table~\ref{t:step1}.

\begin{table}[h]
\center\caption{Classification result for minimal admissible modules}
\scalebox{0.8}{
\begin{tabular}{|c||c|c|c|c|c|c|c|c|c|}
\hline
$8$&$\cong$&&&&&&&&
\\
\hline
7&d&d&d&$\not\cong$&&&&&
\\
\hline
6&d&$\cong$&$\cong$&h&&&&&
\\
\hline
5&d&$\cong$&$\cong$&h&&&&&
\\
\hline
4&$\cong$&h&h&h&&&&&\\\hline
3&d&$\not\cong$&$\not\cong$&&d&d&d&$\not\cong$&\\\hline
2&$\cong$&h&&$\not\cong$&d&$\cong$&$\cong$&h&$$\\\hline
1&$\cong$&&d&$\not\cong$&d&$\cong$&$\cong$&h&$$\\\hline
0&&$\cong$&$\cong$&h&$\cong$&h&h&h&$\cong$\\\hline\hline
$s/r$&0&1&2&3&4&5&6&7&8\\
\hline
\end{tabular}\label{t:step1}
}
\end{table}
Here ``d'' stands for ``double'', meaning that $\dim V^{r,s}_{min}=2\dim
V^{s,r}_{min}$ and
``h'' (half) means that $\dim V^{r,s}_{min}=\frac{1}{2}\dim V^{s,r}_{min}$. 
The corresponding pairs are trivially non-isomorphic due to the 
different dimension of minimal admissible modules. 
The symbol $\cong$ denotes the Lie algebra $\n_{r,s}(V^{r,s}_{min})$ having 
isomorphic counterpart $\n_{s,r}(V^{s,r}_{min})$, the symbol $\not\cong$ shows that the Lie algebra $n_{r,s}(V^{r,s}_{min})$ is not  
isomorphic to $\n_{s,r}(V^{s,r}_{min})$. 

The result of the classification for the cases when the Lie algebras has the same signature $(r,s)$ of the scalar product on the centre and arbitrary admissible modules is contained in~\cite[Theorems 4.1.1-4.1.3]{FM2}. We summarise the result here.

\begin{theorem}\label{main theorem 1} Let $U=(U,\la\cdot\,,\cdot\ra_{U})$ and
$\tilde U=(\tilde{U},\la\cdot\,,\cdot\ra_{\tilde{U}})$
be admissible modules of a Clifford algebra $\Cl_{r,s}$.
\begin{itemize}
\item[1.]{If $r=0,1,2\mod 4$, 
then $\n_{r,s}(U)\cong \n_{r,s}(\tilde U)$, if and only if $\dim(U)=\dim(\tilde U)$}
\item[2.] {
Let $r=3\mod 4$ and $s=0\mod 4$ 
and let the admissible modules be decomposed into the direct sums of the type~\eqref{eq:rseq3}: 
Then the Lie algebras $\n_{r,s}(U)$ and $\n_{r,s}(\tilde{U})$ are isomorphic, if and only if,
$$
p=p^+_++p^{-}_-=\tilde{p}^+_++\tilde{p}^{-}_-=\tilde p\ \ \text{and}\ \ q=p^{-}_++ p^+_-=\tilde{p}^{-}_++\tilde{p}^+_-=\tilde q,\qquad\text{or}
$$
$$
p=p^+_++p^{-}_-=\tilde{p}^{-}_++\tilde{p}^+_-=\tilde q\ \ \text{and}\ \ q=p^{-}_++ p^+_-=\tilde{p}^+_++\tilde{p}^{-}_-=\tilde p.
$$
}
\item[3.]{
Let $r=3\mod 4$ and 
$s=1,2,3\mod 4$ 
and let $U$ and $\tilde U$ be decomposed into the direct sums~\eqref{eq:rsneq3} 
Then $\n_{r,s}(U)\cong \n_{r,s}(\tilde{U})$, if and only if
$$
p=p^+=\tilde{p}^+=\tilde p\ \text{ and }\ q=p^-=\tilde{p}^-=\tilde q,\qquad\text{or}
$$
$$
p=p^+=\tilde{p}^-=\tilde q\ \text{ and }\ q=p^-=\tilde{p}^+=\tilde p.
$$ 
}
\end{itemize}
\end{theorem}

According to Theorem~\ref{main theorem 1} in the cases 
$r= 3\mod 4$ and $s=0\mod 4$ we can substitute $\oplus^{p_-^+}V^{r,s;+}_{min;-}$ by $\oplus^{p_+^-}V^{r,s;-}_{min;+}$ if $p_-^+=p_+^-$. Analogously we replace $\oplus^{p_-^-}V^{r,s;-}_{min;-}$ by $\oplus^{p_+^+}V^{r,s;+}_{min;+}$ if $p_-^-=p_+^+$. Hence we reduce the decompositions of an admissible module to the sums containing only $V^{r,s;\pm}_{min;+}$, and moreover we write simply $V^{r,s;\pm}_{min}$ for the simplicity. Thus
if $r=3\mod 4$ the Lie algebra $\n_{r,s}(U)$ essentially depends on the decomposition 
\begin{equation}\label{eq:isotyp} 
U=\big(\stackrel{p}\oplus V^{r,s;+}_{min}\big)\bigoplus \big(\stackrel{q}\oplus V^{r,s;-}_{min}\big),
\end{equation}
where the numbers $p,q$ are defined in items 2 and 3 of Theorems~\ref{main theorem 1}. We call admissible modules with decompositions~\eqref{eq:isotyp} {\it isotypic} if one of the numbers $p$ or $q$ vanishes. Otherwise the admissible module is called {\it non-isotypic} of type $(p,q)$.

Now we state the classification when the Lie algebras has opposite signatures $(r,s)$ and $(s,r)$ of the scalar products on the centres and arbitrary admissible modules, see~\cite[Theorems 4.6.2]{FM2}. 
We formulate here the revised version of the result obtained in~\cite[Theorem 4.6.2]{FM2}.

\begin{theorem}\label{th:finalclass}
Let $r=0,1,2\mod 4$ and 
$s=0,1,2\mod 4$. 
Then $\n_{r,s}(U^{r,s})\cong\n_{s,r}(U^{s,r})$ if $\dim(U^{r,s})=\dim(U^{s,r})$.

Let $r=3\mod 8$, $s=0,4,5,6\mod 8$
or 
$r=7\mod 8$, $s=0,1,2\mod 8$. 
Then $\n_{r,s}(U^{r,s})\cong\n_{s,r}(U^{s,r})$ if $\dim(U^{r,s})=\dim(U^{s,r})$ 
and 
$U^{r,s}=\big(\stackrel{p}\oplus V^{r,s;+}_{min}\big)\bigoplus \big(\stackrel{p}\oplus V^{r,s;-}_{min}\big).$

Let $r\equiv 3\,(\text{\em mod}~8)$ and $s\equiv 1,2,7\,(\text{\em mod}~8)$. 
Then $\n_{r,s}(U^{r,s})$ is never isomorphic to $\n_{s,r}(U^{s,r})$. 
\end{theorem}


\subsection{Periodicity of $\Aut(\n_{r,s}(U))$ in parameters $(r,s)$}\label{periodicity of iso}


In order to achieve the description of the groups $\Aut^0(\n_{r,s}(U))$ we need only to describe the basic cases, since the rest of the cases follow from the theorems that extend the periodicity properties in $(r,s)$ of the Clifford algebras to the counterpart on the pseudo $H$-type Lie algebras.

\begin{proposition}\label{periodicity all}\cite[Propositions 4.2.1 and 4.2.2]{FM2}
Let $(U^{r,s}_{min},\la\cdot\,,\cdot\ra_{U^{r,s}_{min}})$ 
be a minimal admissible module of $\Cl_{r,s}$ and
$J_{z_i}$, $i=1,\ldots,r+s$ 
the Clifford actions of the orthonormal basis $\{z_{i}\}$. Let also $(V^{\mu,\nu}_{min},\la\cdot\,,\cdot\ra_{V^{\mu,\nu}_{min}})$ 
be a minimal admissible module of $\Cl_{\mu,\nu}$ for $(\mu,\nu)\in\{(8,0),(0,8),(4,4)\}$ and
$J_{\zeta_i}$, $i=1,\ldots,8$ 
the Clifford actions of the orthonormal basis $\{\zeta_{i}\}$.Then
\begin{equation}\label{tensor product rep to direct sum}
U^{r,s}_{min}\otimes V_{min}^{\mu,\nu}
=(U^{r,s}_{min}\otimes E_{\mu,\nu})\bigoplus_{i=1}^{8}\big(U^{r,s}_{min}\otimes J_{\zeta_i}(E_{\mu,\nu})\big)\bigoplus_{j=2}^{8}\big(U^{r,s}_{min}\otimes J_{\zeta_1}J_{\zeta_{j}}(E_{\mu,\nu})\big)
\end{equation}
is a minimal admissible module $U^{r+\mu,s+\nu}_{min}$ of the Clifford algebra $\Cl_{r+\mu,s+\nu}$.

Conversely, if $U^{r+\mu,s+\nu}_{min}$ is a minimal admissible module of the algebra
$\Cl_{r+\mu,s+\nu}$, then the common 1-eigenspace $E_0$ of the involutions
$T_{i}$, $i=1,2,3,4$ from Example~\ref{ex:inv44} can be considered as a minimal admissible module $U^{r,s}_{min}$ of the algebra
$\Cl_{r,s}$. The action of the Clifford algebra $\Cl_{r,s}$ on
$E_{0}$ is the restricted action of $\Cl_{r+\mu,s+\nu}$ obtained by the natural inclusion 
$\Cl_{r,s}\subset \Cl_{r+\mu,s+\nu}$. 
\\

According to the correspondence of minimal admissible modules 
stated in Proposition~\ref{periodicity all}, there
is a natural injective map
\begin{equation}\label{eq:B}
\mathcal{B}\colon\C(\n_{r,s}(U^{r,s}_{min}))\to
\C(\n_{r+\mu,s+\nu}(U^{r+\mu,s+\nu}_{min})).
\end{equation}

Conversely, automorphisms of the form ${A}\oplus
C\in \C(\n_{r+\mu,s+\nu}(U^{r+\mu,s+\nu}_{min}))$ with the property that
$C(\zeta_{j})=\zeta_{j}$, $j=1,\ldots,8$,
defines an automorphism $A_{|E_{0}}\oplus C_{|\mathbb{R}^{r,s}}$ of 
the algebra $\n_{r,s}(E_0)$, where
the space $E_0$ is the common 1-eigenspace of the 
involutions $T_{j}$, $j=1, 2, 3, 4$, viewed as a minimal admissible module of $\Cl_{r,s}$.
\end{proposition}
Recall that the kernel of the map $\psi$
$$
\begin{array}{ccccc}
&\C(\n_{r,s}(U^{r,s})) &{\overset {\psi }{\longrightarrow}} &\Orth(r,s)
\\
& A\oplus C&\mapsto &C.
\end{array}
$$
is given by the inclusion $\iota\big(\Aut^0(\n_{r,s}(U^{r,s}))\big)\subset \C(\n_{r,s}(U^{r,s}))$, by Proposition~\eqref{automorohisms from Clifford algebra}.
\begin{corollary}\label{cor:kernel}
Let $U^{r,s}$ and $U^{r+\mu,s+\nu}=U^{r,s}\otimes V^{\mu,\nu}_{min}$ be admissible modules. 
Then 
\[
\Aut^0(\n_{r+\mu,s+\nu}((U^{r+\mu,s+\nu}))=\mathcal{B}\big(\Aut^0(\n_{r,s}(U^{r,s}))\big),
\]
that is the group $\Aut^0(\n_{r,s}(U^{r,s}))$ is invariant under the map $\mathcal{B}$ defined in~\eqref{eq:B}.
\end{corollary}
Finally, we state the result of the periodicity of isomorphisms for the Lie algebras.
\begin{theorem}\label{th:periodic1}\cite[Theorem 4.6.1]{FM2}
The Lie algebras $\n_{r,s}(U^{r,s})$ and $\n_{s,r}(U^{s,r})$ are isomorphic if and only if the Lie algebras $\n_{r+\mu,s+\nu}(U^{r+\mu,s+\nu})$ and $\n_{s+\nu,r+\mu}(U^{s+\nu,r+\mu})$ are isomorphic for $(\mu,\nu)\in \{(8,0),(0,8),(4,4)\}$.
\end{theorem}



\section{Definition of classical groups}\label{sec:groupG}

We aim to determine the subgroups $\mathbb A$ of $\SL(n,\mathbb R)$ such that if $A\in \mathbb A$, then $A\oplus \Id \in\Aut^0(\n_{r,s}(U)$. In what follows we will identify $\mathbb A$ and $\Aut^0(\n_{r,s}(U)$.
The maps $A\colon U\to U$ are linear maps over the field of real
numbers. From the other side the admissible modules $U$ carry complex
or quaternion structures such that the map $A$ commutes with
them. Thus, the map $A$ has to be linear with respect to these
additional algebras.
We recall the algebras $\mathbb C$, $\mathbb H$ and some useful embeddings into the space of real matrices.


\subsection{Algebras over $\mathbb R$}\label{sec:algebras}


We write 
$
\lambda=a+b{\bf i}$, ${\bf i}^2=-1$, 
for $\lambda\in \CC$ and 
$
h=a+b{\bf i}+c{\bf j}+d{\bf k}$ for $h\in\HH$. Recall that 
\begin{equation}\label{eq:AlmostQuaternion}
{\bf i}^2={\bf j}^2={\bf k}^2={\bf i}{\bf j}{\bf k}=-1.
\end{equation}

We describe here the embeddings of the algebras $\mathbb F=\CC,\HH$ and square matrices $M(n,\mathbb F)$ into the set of real square matrices $M(n,\mathbb R)$ and complex square matrices $M(n,\mathbb C)$, respectively.
We define an embedding 
\begin{equation}\label{eq:phoCC}
\begin{array}{ccllllll}
\rho_{\mathbb{C}}\colon&\mathbb{C}&\to&M(2,\mathbb{R})
\\
&\lambda=a+b{\bf i}&\mapsto& \begin{pmatrix}a&-b\\b&a\end{pmatrix}.
\end{array}
\end{equation}
Then naturally we have 
$$
A=\rho_{\CC}(A_{\CC})=\rho_{\CC}\Big(\begin{pmatrix}\lambda_{11}&\lambda_{12}
\\
\lambda_{21}&\lambda_{22}\end{pmatrix}\Big)=
\begin{pmatrix}
a_{11}&-b_{11}&a_{12}&-b_{12}
\\
b_{11}&a_{11}&b_{12}&a_{12}
\\
a_{21}&-b_{21}&a_{22}&-b_{21}
\\
b_{21}&a_{21}&b_{22}&a_{22}
\end{pmatrix}
$$
for $\lambda_{kl}=a_{kl}+b_{kl}\mathbf{i}$. The map $\rho_{\mathbb{C}}$ is the algebra
homomorphism:
\begin{equation*}
\rho_{\mathbb{C}}(A_{\CC}B_{\CC})=\rho_{\mathbb{C}}(A_{\CC}) \rho_{\mathbb{C}}(B_{\CC})=AB,
\end{equation*}
$$
\rho_{\mathbb{C}}(\overline{\lambda})=\big(\rho_{\mathbb{C}}(\lambda)\big)^T,\quad \lambda\in\mathbb{C},
$$
$$
\rho_{\mathbb{C}}(\overline{A_{\CC}}^T)=\big(\rho_{\mathbb{C}}(A_{\CC})\big)^T=A^T,\quad A_{\CC}\in M(n,\mathbb{C}),
$$
where superscript $A^T$ denotes the transposition of $A$.
Note also that if we denote by $\diag_nL$ a block-diagonal real matrix with the blocks $L$ on the diagonal, then
\begin{equation}\label{eq:conjugation1}
\diag_n\begin{pmatrix}
1&0\\0&-1
\end{pmatrix}
\rho_{\CC}(A_{\CC})
\diag_n\begin{pmatrix}
1&0\\0&-1
\end{pmatrix}=
\rho_{\CC}(\bar A_{\CC}).
\end{equation}

A quaternion number can be expressed by using the complex numbers by 
$$
h=a+b{\bf i}+c{\bf j}+d{\bf k}=\lambda+{\bf j}\mu, \quad \lambda=a+b{\bf i},\quad \mu=c+d{\bf i}.
$$
with the conjugation 
$
\bar h=a-b{\bf i}-c{\bf j}-d{\bf k}=\bar \lambda-{\bf j}\bar\mu$.
Thus we define 
\begin{equation*}
\begin{array}{ccllllll}
 \rho_{\mathbb{H}}\colon&\mathbb{H}&\to& M(2,\mathbb{C})
\\
&h=\lambda +{\bf j}\mu &\mapsto &\begin{pmatrix}\lambda&-\overline{\mu}\\\mu&\overline{\lambda}\end{pmatrix}.
\end{array}
\end{equation*}
Consider the space $\HH^n$ as right quaternion space.  Thus, $A_{\HH}(vh) = (A_{\HH}v)h$ for $h\in\HH$, a quaternion column vector $v\in\HH^n$ and quaternion matrix $A_{\HH}$. 
The column vector $h=(h_1,\ldots, h_n)^T\in\HH^n$ with $h_l=\lambda_l+{\mathbf j}\mu_l$ will be represented by the column vector $(\lambda_1,\ldots,\lambda_n,\mu_1,\ldots,\mu_n)^T\in\CC^{2n}$.
Then the quaternion matrix $Q_{\HH}\in M(n,\HH)$ written as $Q_{\HH}=\Lambda_{\CC}+{\bf j}\Psi_{\CC}$ with $\Lambda_{\CC},\Psi_{\CC}\in M(n,\CC)$ will be represented as
$$
\rho_{\HH}(Q_{\HH})=\begin{pmatrix}\Lambda_{\CC}&-\overline{\Psi_{\CC}}\\ \Psi_{\CC}&\overline{\Lambda_{\CC}}\end{pmatrix}\in M(2n,\CC).
$$
This representation is convenient by the following reason: if $\HH\ni h=\lambda + {\bf j}\mu$ is given as a column vector $\begin{pmatrix}\lambda\\ \mu\end{pmatrix}$, then multiplication from the left by a complex matrix representation of a quaternion produces a new column vector representing the correct quaternion. The map $\rho_{\mathbb{H}}$ is also the algebra
homomorphisms:
\begin{equation*}
\rho_{\mathbb{H}}(A_{\HH} B_{\HH})=\rho_{\mathbb{H}}(A_{\HH})\rho_{\mathbb{H}}(B_{\HH}),
\end{equation*}
$$
\rho_{\mathbb{H}}(\overline{h})=\overline{\big(\rho_{\mathbb{H}}(h)\big)}^T,\quad h\in\mathbb{H}
$$
$$
\rho_{\mathbb{H}}(\overline{B_{\HH}}^T)=\overline{\big(\rho_{\mathbb{H}}(B_{\HH})\big)}^T,\quad B_{\HH}\in M(n,\mathbb{H}).
$$

We recall the following definitions of the classical groups that will be used in the sequel. 
\noindent The \textbf{general linear group} $\GL(n,\F)$ of degree $n$ over the fields $\mathbb{F}=\R,\CC$ is 
\[
\GL(n, \mathbb{F}) \colonequals \{ M \in M(n, \mathbb{F}) \mid M \text{ is invertible} \}.
\] 

\noindent The \textbf{general orthogonal group} $\Orth(n, \mathbb{F})$ over the fields $\mathbb{F}=\mathbb{R},\mathbb{C}$ is
\[
\Orth(n, \mathbb{F})\colonequals \lbrace M \in \GL(n, \mathbb{F}) \mid  M^T M= \Id_n \rbrace, 
\]
where $\Id_n$ is the $(n \times n)$ identity matrix. In the case $\mathbb F=\mathbb R$ we also use the pseudo-orthogonal group
$\Orth(p,q)$ 
\[
\Orth(p,q)\colonequals \lbrace M \in \GL(p+q, \mathbb{R}) \mid  M^T \Id_{p,q} M= \Id_{p,q}\rbrace,\qquad
\Id_{p,q}= \begin{pmatrix}
\Id_p & 0 \\
0 & -\Id_q \\
\end{pmatrix}.
\]
All the groups over $\R$ preserving a symmetric bilinear form of index $(p,q)$ are isomorphic to $\Orth(p,q)$. The groups over $\CC$ preserving a symmetric bilinear form of index $(p,q)$ are isomorphic to $\Orth(n,\CC)$ with $n=p+q$, see~\cite[Chapter 3.1]{Rossmann}. 

The \textbf{symplectic group} $\Sp(2n, \F)$ of degree $2n$  over the fields $\F=\R,\CC$ is
\[
\Sp(2n, \mathbb{F}):= \lbrace M \in \GL(2n, \mathbb{F}) \mid M^T \Omega_n M= \Omega_n \rbrace, \qquad
\Omega_n = \begin{pmatrix}
0 &  -\Id_n \\
\Id_n & 0 \\
\end{pmatrix}.
\]
All the groups preserving a skew-symmetric bilinear form are isomorphic to $\Sp(2n, \mathbb{F})$. 

The \textbf{general unitary group} $\U(p,q)$ of degree $n$ is
\[
\U(p, q) \colonequals \lbrace M \in \GL(n, \mathbb{C}) \mid \overline{M}^T \Id_{p,q} M= \Id_{p,q} \rbrace.
\]
The subgroup $\U(p,0)\subset \U(p,q)$ is denoted by $\U(p)$. Note that from a qualitative point of view, consideration of skew-Hermitian forms (up to isomorphism) provides no new groups, since the multiplication by ${\bf i}$ renders a skew-Hermitian form Hermitian, and vice versa. Thus only the Hermitian case needs to be considered.

Now we turn to define the groups over the algebra $\mathbb F=\HH$. 
Under the identification described above 
$$
\GL(n,\HH)=\{M\in \GL(2n,\CC)\mid\ \Omega_n M=\overline M\Omega_n,\quad \det M\neq 0\}
$$
$$
\SL(n,\HH)=\{M\in \GL(n,\HH)\mid\ \det M=1\}
$$
\begin{eqnarray*}
\Sp(p,q)&=&\{M\in \GL(n,\HH)\mid\ \overline M^T \Id_{p,q}M=\Id_{p,q},\ p+q=n\}
\\
&=&
\Big\{M\in \GL(2n,\CC)\mid\ \overline M^T\diag\begin{pmatrix}\Id_{p,q}&0\\0&\Id_{p,q}\end{pmatrix} M=\diag\begin{pmatrix}\Id_{p,q}&0\\0&\Id_{p,q}\end{pmatrix}\Big\}.
\end{eqnarray*}
The group $\Sp(p,q)$ is called {\bf quaternionic unitary group}. 
If $p=0$ or $q=0$, then $\Sp(0,p)\cong\Sp(p,0)$ is denoted by $\U(n,\HH)$ and called {\bf hyperunitary group.} The reason for the notation $\Sp(p,q)$ is that this group can be represented, as a subgroup of $\Sp(n, \CC)$ preserving an Hermitian form of signature $(2p, 2q)$. 

The last group is the {\bf quaternionic orthogonal group} denoted by $\Orth^*(2n)=\Orth(n,\HH)$ and it is defined by
\begin{eqnarray*}
\Orth^*(2n)=\Orth(n,\HH)&=&\{M\in \GL(n,\HH)\mid\ \bar M^T \diag_n {\mathbf j} M=\ \diag_n {\mathbf j}\}
\\
&=&
\{M\in \GL(2n,\CC)\mid\ \bar M^T\Omega_n M=\Omega_n\}.
\end{eqnarray*}
Here ${\mathbf j}$ is the quaternionic unit represented by $\rho_{\HH}({\mathbf j})\in\begin{pmatrix}0&-1\\1&0\end{pmatrix}\in M(2,\CC)$. The definition of $\Orth^*(2n)=\Orth(n,\HH)$ can be given equivalently as
\begin{eqnarray*}
\Orth^*(2n)=\Orth(n,\HH)&=&\{M\in \GL(n,\HH)\mid\ \bar M^T \diag_n {\mathbf i} M=\ \diag_n {\mathbf i}\}
\\
&=&\{M\in \GL(n,\HH)\mid\ \bar M^T \diag_n {\mathbf k} M=\ \diag_n {\mathbf k}\}. 
\end{eqnarray*}
This is true due to the fact that by conjugation with some $h,\tilde h\in \Sp(1)$ we can get $h {\mathbf i}h^{-1}={\mathbf j}$ and analogously $\tilde h {\mathbf k}\tilde h^{-1}={\mathbf j}$.
The group $\Orth^*(2n)=\Orth(n,\HH)$ can be viewed as a subgroup of $\Orth(2n,\CC)$ that preserves an Hermitian form of index $(n,n)$. Particularly if $n=1$, then one needs to check the condition
$$
\begin{pmatrix}
\bar \lambda&\bar\mu
\\
-\mu&\lambda
\end{pmatrix}
\begin{pmatrix}
0&-1
\\
1&0
\end{pmatrix}
\begin{pmatrix}
\lambda&-\bar\mu
\\
\mu&\bar\lambda
\end{pmatrix}=
\begin{pmatrix}
0&-1
\\
1&0
\end{pmatrix},
$$ 
with $\lambda=a+{\mathbf i}b$, $\mu=c+{\mathbf i}d$. It leads to the solution of the system
$$
\begin{cases}
\im(\bar \lambda\mu)=0
\\
\lambda^2+\mu^2=1
\end{cases}\quad\Longrightarrow \quad
\begin{cases}
ad=bc
\\
ab+cd=0
\\
a^2-b^2+c^2-d^2=1
\end{cases}\quad\Longrightarrow \quad
\begin{cases}
a^2+c^2=1
\\
b=d=0.
\end{cases}
$$
Thus
$$
\begin{pmatrix}
\lambda&-\bar\mu
\\
\mu&\bar \lambda
\end{pmatrix}=
\begin{pmatrix}
a&-c
\\
c&a
\end{pmatrix}=\alpha=a+{\mathbf i}c\quad\text{and}\quad
a^2+c^2=|\alpha|^2=1
$$
Thus we conclude that $\Orth^*(2)\cong \U(1)$.


\section{Determination of $\Aut^0(\n_{r,s}(U))$}\label{Appendix0}



\subsection{Integral basis}


\begin{definition}
We fix the standard orthonormal basis $\{z_{k}\}$ for
$\mathbb{R}^{r,s}$. 
Then we call a basis $\{x_{i}\}$ of the minimal
admissible module $V_{min}$, an integral basis with respect to the
orthonormal basis $\{z_{k}\}$, if it
satisfies the conditions that
\begin{itemize}
\item {the basis $\{x_{i}\}$ is orthonormal with respect to the admissible scalar product,}
\item {for any $z_{k}$ and $x_{i}$, there exists a unique $x_{j}$ such
that either $J_{z_{k}}(x_{i})=x_{j}$ or $J_{z_{k}}(x_{i})=-x_{j}$.}
\end{itemize}
\end{definition}

One way to construct such a basis is given by taking a suitable vector
from $E_{r,s}$ and choosing an orthonormal basis of $V^{r,s}_{min}$ from the vectors
\[
\{v, \pm J_{z_{k}}v,\ldots,
\pm J_{z_{k_{1}}}J_{z_{k_{2}}}\ldots J_{z_{k_{l}}}v,\ldots,
\pm J_{z_{1}}J_{z_{2}}\ldots J_{z_{r+s}}v,\ 1\leq k_{1}<\ldots<
k_{l}\leq r+s
\}
\]
The choice of the basis is not unique. Nevertheless, once we fix a
basis, we denote by $\eta$ the matrix of the admissible scalar
product. Thus either $\eta=\Id_{2n}$ or
$\eta=\begin{pmatrix}\Id_n&0\\0&-\Id_n\end{pmatrix}$ according to the
ordering of positive vectors to negative vectors of the fixed integral basis.
The construction of the integral basis can be found in~\cite{FM}.

Recall that $J_{z_i}^{\tau}$ is the transposition with respect to an admissible scalar product $\la.\,,.\ra$, and $J_{z_i}^T$ the transposition with respect to a standard Euclidean scalar product. The relation between two transpositions is
$
J_{z_i}^{\tau}=\eta J_{z_i}^T\eta$.
\begin{lemma}\label{lem:symplectic}
If $ J_{z_i}^{\tau}=- J_{z_i}$, $ J_{z_i}^2=\pm\Id$, $i=1,2,3$, $ J_{z_i} J_{z_j}=- J_{z_j} J_{z_i}$, $i,j=1,2,3$, $i\neq j$, and $\eta^T=\eta$, $\eta^2=\Id$ is non-degenerate bi-linear form, then
\begin{itemize}
\item[(1.)]{
$(\eta J_{z_i})^T=-\eta J_{z_i}$;}
\item[(2.)]{$(\eta J_{z_k})^2=-\Id$;}
\item[(3.)]{$\eta J_{z_i}=
\begin{cases}
- J_{z_i}\eta\quad&\text{if}\quad  J_{z_i}^2=\Id
\\
 J_{z_i}\eta\quad&\text{if}\quad  J_{z_i}^2=-\Id
\end{cases}$;}
\end{itemize}
\end{lemma}
\begin{proof}
(1.) We obtain $
(\eta  J_{z_i})^T= J_{z_{z_i}}^T\eta^T=-\eta  J_{z_i}.
$ from $\eta  J_{z_i}=- J^T_{z_i}\eta$.

(2.) We consider four cases.

(a) Let $J_{z_i}^2=-\Id$ and $x_j$ an element of the integral basis such that $\la x_j,x_j\ra>0$. Then $\eta J_{z_i}x_j=J_{z_i}x_j$
$$
(\eta J_{z_{i}})^2(x_{j})=\eta J_{z_i}^2x_j=-x_j.
$$

(b) Let $J_{z_i}^2=-\Id$ and $\la x_j,x_j\ra<0$. Then $\eta J_{z_i}x_j=-J_{z_i}x_j$
$$
(\eta J_{z_{i}})^2(x_{j})=-\eta J_{z_i}^2x_j=\eta x_j=-x_j.
$$

(c) Let $J_{z_i}^2=\Id$ and $\la x_j,x_j\ra>0$. Then $\eta J_{z_i}x_j=-J_{z_i}x_j$
$$
(\eta J_{z_{i}})^2(x_{j})=-\eta J_{z_i}^2x_j=-\eta x_j=-x_j.
$$

(d) Let $J_{z_i}^2=\Id$ and $\la x_j,x_j\ra<0$. Then $\eta J_{z_i}x_j=J_{z_i}x_j$
$$
(\eta J_{z_{i}})^2(x_{j})=\eta J_{z_i}^2x_j=\eta x_j=-x_j.
$$

(3.) The property $\eta J_{z_i}\eta J_{z_i}=-\Id$ implies
$J_{z_i}\eta J_{z_i}^2=-\eta J_{z_i}$.
\end{proof}


\subsection{Description of the procedure  of determination of $\Aut^0(\n_{r,s}(U))$}\label{procedure}


Step 1. We determine the groups $\Aut^0(\n_{r,s}(V^{r,s}))$ for the
basic cases~\eqref{basic cases}. According to Corollary~\ref{cor:kernel} it provides the groups for all range of $(r,s)$. Thus, the next steps are explained only for basic cases.
\\

Step 2. We determine the groups $\Aut^0(\n_{r,s}(V^{r,s}_{min}))$ for minimal admissible modules.
\begin{itemize}
\item[2.1] We find $PI_{r,s}$ the sets of involutions of all types (1)-(5) and their subsets $P_{r,s}^*\subset PI_{r,s}$ that are involutions of types (1)-(3). We write $P_k$ for the operators from $PI_{r,s}$. We denote by $E_{r,s}^*$ the common $1$-eigenspace of involutions from $P_{r,s}^*$ and $E_{r,s}$ the common $1$-eigenspace of involutions from $P_{r,s}$. We find operators that commute with all involutions from $P_{r,s}^*$. These operators will leave the space $E_{r,s}^*$ invariant. Among these operators we denote by $\ii$ the almost complex structure, and by $\ii,\jj,\kk$ the almost quaternion structure, i.e. the operators satisfying~\eqref{eq:AlmostQuaternion} and the product of even number of maps $J_{z_k}$. We use the notation $\QQ$ for the negative operator $\QQ=J_{z_i}J_{z_j}$ such that $\QQ^2=\Id$. Apart of mentioned operators it could be at most one more, denoted by $\PP$ that is a product of an even number of $J_{z_k}$ commuting with all involutions from $P_{r,s}^*$. All these operators will be indicated for each case. We denote by $A$ an operator on $P_{r,s}^*$ that defines $\Aut^0(\n_{r,s}(V^{r,s}_{min}))$ by means of relations~\eqref{isomorphism relation}.
\item[2.2] We choose an integral basis generated from a vector $v\in E_{r,s}$, $\la v,v\ra_{E_{r,s}}=1$. Here we emphasise that $E_{r,s}\subset E_{r,s}^*$ is the common $1$-eigenspace of all types of involutions from $PI_{r,s}$. The details of the construction of the integral basis can be found in~\cite{FM1}. The basis of $E_{r,s}^*$ will be indicated for each case in a table. We use the black colour to denote the basis vectors $x_k$ such that $\la x_k,x_k\ra_{E_{r,s}^*}=1$ and by red colour the basis vectors $x_l$ such that $\la x_l,x_l\ra_{E_{r,s}^*}=-1$.
\item[2.3] In this step we distinguish 6 possible collections of operators on $E_{r,s}^*$ that leave it invariant.
\begin{itemize}
\item[2.3.1] {\bf The set $E_{r,s}^*$ has neither complex, quaternion structure, no operator $\QQ$.} In this case the operator $A\colon E_{r,s}^*\to E_{r,s}^*$ is real. In the presence of an operator $\PP$ we check the condition~\eqref{isomorphism relation}, that we write in the form:
\begin{equation}\label{eq:etaPP}
A^T\eta\PP A=\eta\PP.
\end{equation}
These are the cases 
$$
(r,s)\in\{(1,0),(0,1),(7,0),(0,7),(8,0),(0,8),(3,4),(4,3),(4,4)\}.
$$
\item[2.3.2] {\bf The set $E_{r,s}^*$ has only a complex structure, but neither quaternion structure, no operator $\QQ$.} Since $A$ commutes with $\ii$ we conclude that $A\in\GL(k,\CC)$, where $k=\dim_{\CC}(E_{r,s}^*)$. If there is no operator $\PP$ on $E_{r,s}^*$, then $\Aut^0(\n_{r,s}(V^{r,s}_{min}))=\GL(k,\CC)$. Otherwise we check the condition~\eqref{eq:etaPP}. There are two options: if the map $\eta\PP$ is complex liner ($\eta\PP$ commutes with $\ii$), then 
$$\Aut^0(\n_{r,s}(V^{r,s}_{min}))\cong\Sp(k,\CC)\quad\text{or}\quad
\Aut^0(\n_{r,s}(V^{r,s}_{min}))\cong\U(k). 
$$ If the $\eta\PP$ is not complex linear, then $\Aut^0(\n_{r,s}(V^{r,s}_{min}))\cong\Orth(k,\CC)$. These are the cases 
$$
(r,s)\in\{(2,0),(0,2),(6,0),(0,6),(2,4),(4,2),(3,5),(5,3),(7,1),(1,7)\}.
$$
\item[2.3.3] {\bf The set $E_{r,s}^*$ has quaternion structure, and has no operator $\QQ$.} Since $A$ commutes with $\ii,\jj,\kk$ we conclude $A\in\GL(k,\HH)$, where $k=\dim_{\HH}(E_{r,s}^*)$. All the operators $\eta\PP$ will be quaternion linear and by checking~\eqref{eq:etaPP} we make the conclusions in the cases
$$
(r,s)\in\{(3,0),(0,3),(4,0),(0,4),(5,0),(0,5),(4,1),(1,4),(5,2),(2,5),
$$
$$(6,1),(1,6),(6,2),(2,6),(6,3),(3,6), (7,2),(2,7)\}.
$$
\item[2.3.4] {\bf The set $E_{r,s}^*$ has operator $\QQ$ and neither has complex no quaternion structure.} In the presence of the operator $\QQ$ we decompose $E_{r,s}^*$ into eigenspaces of the involution $\QQ$ that we denote by $N_{\pm}$. Thus $E_{r,s}^*=N_+\oplus N_-$. Since $A$ commutes with $\QQ$, we get $A=A_+\oplus A_-$, where $A_{\pm}\colon N_{\pm}\to N_{\pm}$. We check~\eqref{eq:etaPP} and make the conclusion. Since in this case there are no other conditions on $A_{\pm}$ the group $\Aut^0(\n_{r,s}(V^{r,s}_{min}))$ will be given by a direct product of subgroups of $\GL(k,\R)$ with $k=\dim(N_{\pm})$. These are the cases 
$$
(r,s)\in\{(1,1),(3,3)\}.
$$
\item[2.3.5] {\bf The set $E_{r,s}^*$ has a complex structure and operator $\QQ$ but does not have a quaternion structure.} We start from the decompositions $E_{r,s}^*=N_+\oplus N_-$ and $A=A_+\oplus A_-$. In all these cases we have $\QQ\ii=-\ii \QQ$ and therefore we define $A_-=-\ii A_+\ii$. If it needs, we check~\eqref{eq:etaPP} on $N_+$ and make the conclusions. These are the cases 
$$
(r,s)\in\{(2,2),(3,2),(2,3),(2,1),(1,2)\}.
$$
\item[2.3.6] {\bf The set $E_{r,s}^*$ has a quaternion structure and the operator $\QQ$.} We start from the decompositions $E_{r,s}^*=N_+\oplus N_-$ and $A=A_+\oplus A_-$. The result depends on the situation whether $N_+$ carries the complex or quaternion structure. 
These are the cases 
$$
(r,s)\in\{(3,1),(1,3),(5,1),(1,5),(7,3),(3,7)\}.
$$
\end{itemize}
\item[2.4] Having in hands the operator $A\colon E_{r,s}^*\to E_{r,s}^*$, we can extend it to the operator $\overline A\colon V^{r,s}_{min}\to V^{r,s}_{min}$. The operator $\overline A$ is completely and uniquely determined by the operator $A$ according to Theorem~\ref{th:general}. To match the notation of the present description and Theorem~\ref{th:general} we note that $E_{r,s}^*=E^1$ and $A=A^1$ in Theorem~\ref{th:general}. The operators $G_I$ used for the construction of $\overline A$ are indicated for all the cases in tables. We emphasise that we present only some of the operators $G_I$, since the extention of $\overline A$ from $A$ does not depend on the choice of a specific operator $G_I$, but only on its existence. The map $\overline A$ will satisfy~\eqref{isomorphism relation} by Theorem~\ref{th:general}.  
Thus the group $\Aut^0(\n_{r,s}(V^{r,s}_{min}))$ is already defined in item 2.3.
\end{itemize}
We proceed to the next step.
\\

Step 3. We determine the groups $\Aut^0(\n_{r,s}(V^{r,s}))$ for arbitrary admissible modules $V^{r,s}=\oplus V^{r,s}_{min}$. It follows from the following procedure. We decompose the module $V^{r,s}$ into the orthogonal direct sum~\eqref{eq:isotyp} of minimal admissible modules following the classification of Theorem~\ref{main theorem 1}.
We write $V^{r,s}\supset E=\oplus^{l=p}_{l=1}(E_{r,s}^*)_l$, where $(E_{r,s}^*)_l\subset (V_{min}^{r,s})_l$. In each $(E_{r,s}^*)_l$ will be chosen a vector $v_l$, with $\la v_l,v_l\ra_{(E_{r,s})_l}=\pm 1$, generating an orthonormal basis on $(V_{min}^{r,s})_l$.  We draw the attention of the reader to the fact that $\la v_l,v_l\ra_{(E_{r,s})_l}= 1$ if $(E_{r,s})_l\in (V^{r,s;+}_{min})_l$, $\la v_l,v_l\ra_{(E_{r,s})_l}= -1$ if $(E_{r,s})_l\in (V^{r,s;-}_{min})_l$ and always $\la v_l,v_l\ra_{(E_{r,s})_l}= 1$ for $(E_{r,s})_l\in (V^{r,s;N}_{min})_l$.
We write $v=\oplus^{l=p}_{l=1} v_l$ for the generating vector on $E\subset V^{r,s}$. The result for $E\subset V^{r,s}$ is the direct sum of the results for $(E_{r,s}^*)_l\subset (V_{min}^{r,s})_l$, $l=1,\ldots,p$, that will allow us to make the conclusion in each case.

We list the final result of the determination and then we proceed to consider case by case.

\begin{table}[h]
\centering
\caption{Groups $\Aut^0(\n_{r,s}(U)$}
\scalebox{0.6}{
\begin{tabular}{| c || c| c| c| c| c| c| c|	c| c|}
\hline
8 & $\GL(p, \mathbb{R})$ & & & & & & & & \\
\hline

7 &$\Orth(p,p,\R)$& $\U(p,p)$ & $\Sp(p,p)$ & $\Sp(p,q)\times \Sp(p,q)$& & & & & \\
\hline

6 & $ \Orth(2p,\CC)$ & $\Orth^*(2p)$ & $\GL(p, \HH)$  & $\Sp(p,q)$ & & & & & \\
\hline

5 &$\Orth^*(4p)$  &$\Orth^*(2p)\times \Orth^*(2p)$  &$\Orth^*(2p)$&  $\U(p,q)$& & & & & \\
\hline

4 & $\GL(p, \mathbb{H})$ &  $\Orth^*(2p)$ & $\Orth(p, \mathbb{C})$ & $\Orth(p,q, \mathbb{R})$ & $\GL(p, \mathbb{R})$ & & & & \\
\hline

3 &$\Sp(p,p)$ & $\U(p,p)$ &$\Orth(p,p,\R)$ & $\Orth(p,q,\R)\times \Orth(p,q,\R)$  & $\Orth(p,p)$& $\U(p,p)$& $\Sp(p,p)$ & $\Sp(p,q)\times \Sp(p,q)$& \\
\hline

2 &$\Sp(2,\mathbb C)$ & $\Sp(2p,\R)$ & $\GL(2p,\R)$ & $\Orth(2p,2q,\R)$&$\Orth(2p,\CC)$ &$\Orth^*(2p)$ & $\GL(p, \HH)$ &  $\Sp(p,q)$&  \\
\hline

1 &$\Sp(2p,\mathbb R)$  &$\Sp(2p,\mathbb{R})\times \Sp(2p,\mathbb{R})$ & $\Sp(4p,\R)$&$\U(2p,2q)$&$\Orth^*(4p)$  &$\Orth^*(2p)\times\Orth^*(2p)$ & $\Orth^*(2p)$ &$\U(p,q)$ & \\
\hline

0 &  &$\Sp(2p,\mathbb R)$ &$\Sp(2p,\mathbb C)$ & $\Sp(p,q)$ &  $\GL(p, \mathbb{H})$ &  $\Orth^*(2p)$ & $\Orth(p, \mathbb{C})$& $\Orth(p,q,\mathbb{R})$ & $\GL(p, \mathbb{R})$ \\ 
\hline
\hline

 & 0 & 1 &2 & 3 & 4 & 5 & 6 & 7 & 8 \\
\hline
\end{tabular}
}
\label{tab:F-one}
\end{table}
In the following sections we will write the calculation in the order that was described in item 2.3. We write $J_k$ for $J_{z_k}$ for shortness.


\subsection{Modules over $\mathbb R$}



\subsubsection{$\dim_{\mathbb R}(E_{r,s}^*)=1$: cases $\n_{7,0}(U)$, $\n_{3,4}(U)$, $\n_{8,0}(U)$, $\n_{4,4}(U)$, $\n_{0,8}(U)$.}


\begin{center}
\centering
\scalebox{0.7}{
\begin{tabular}{|c|c|c|c|c|c|c|c|c|c|}
\hline
$V^{7,0}_{\min}$ & \multicolumn{8}{c|}{} & $\dim=8$ 
\\ 
\hline
$E_{P_1}^{\pm}$ & \multicolumn{4}{c|}{+}                         & \multicolumn{4}{c|}{-}                         & $\dim=4$  
\\ 
\hline
$E_{P_2}^{\pm}$ & \multicolumn{2}{c|}{+} & \multicolumn{2}{c|}{-} & \multicolumn{2}{c|}{+} & \multicolumn{2}{c|}{-} & $\dim=2$ 
\\ 
\hline
$E_{P_3}^{\pm}$ &   +   $E^*_{7,0}$     &     -      &   +        &      -     &    +       &     -      &    +       &      -     & $\dim=1$  
\\ 
\hline
Basis for $E^*_{7,0}$& $v$ & $\ldots$ & $\ldots$ & $\ldots$ & $\ldots$ & $\ldots$ & $\ldots$ & $\ldots$ & $\begin{array}{ll}&P_1=J_1J_2J_4J_5 \\ & P_2=J_1J_2J_6J_7\\ & P_3=J_1J_3J_4J_6\\ &P_4=J_1J_2J_3\\ &\PP=J_1J_2J_3\end{array}$
\\
\hline
$G_I$ &  & $J_3$ & $J_7$ & $J_6$ & $J_5$ & $J_4$ & $J_2$ & $J_1$ &  
\\ 
\hline
\end{tabular}
}
\end{center}

\vskip0.5cm

There are four types of minimal admissible modules: 
$$
V_{min;+}^{7,0;+},\quad 
V_{min;+}^{7,0;-},\quad
V_{min;-}^{7,0;+},\quad
V_{min;-}^{7,0;-}.
$$
According to the classification Theorem~\ref{main theorem 1} we can reduce the consideration to the non-isotypic  $(p,q)$-module
\begin{equation}\label{eq:non-isotypic07}
U=\big(\oplus^{p}V_{min;+}^{7,0;+}\big)\oplus \big(\oplus^{q}V_{min;+}^{7,0;-}\big).
\end{equation} 
We consider non-isotypic $(p,q)$-module~\eqref{eq:non-isotypic07} and 
a vector space $E=\big(\oplus^p (E_{7,0}^*)^+\big)\oplus \big(\oplus^q (E_{7,0}^*)^-\big)$, with $(E_{7,0}^*)^+\subset V_{min;+}^{7,0;+}$ and $(E_{7,0}^*)^-\subset V_{min;+}^{7,0;-}$. Note that $\PP$ acts as $\Id$ on $E$ and $\eta=\Id_{p,q}$.
The unique condition that needs to be checked is
$$
A^T\eta \PP A=\eta \PP\quad\Longleftrightarrow\quad
A^T\Id_{p,q} A=\Id_{p,q}.
$$
We conclude $\Aut^0(\n_{7,0}(U))=\Orth(p,q;\mathbb{R})$.

Structure of the minimal admissible modules and the involutions for $\n_{3,4}(U)$ are similar to $\n_{7,0}(U)$ and we conclude that $\Aut^0(\n_{3,4}(U))\cong \Orth(p,q,\mathbb{R})$ for a non-isotypic $(p,q)$-module $U$. 

\vskip0.5cm

\begin{center}
\centering
\scalebox{0.65}{
\begin{tabular}{|c|c|c|c|c|c|c|c|c|c|c|c|c|c|c|c|c|c|}
\hline
$V^{8,0}_{min}$ & \multicolumn{16}{c|}{}                                                                                                                                                                        & $\dim=16$ 
\\ 
\hline
$E_{P_1}^{\pm}$ & \multicolumn{8}{c|}{+}                                                                         & \multicolumn{8}{c|}{-}                                                                         & $\dim=8$ 
\\ 
\hline
$E_{P_2}^{\pm}$ & \multicolumn{4}{c|}{+}                         & \multicolumn{4}{c|}{-}                         & \multicolumn{4}{c|}{+}                         & \multicolumn{4}{c|}{-}                         & $\dim=4$ 
\\ 
\hline
$E_{P_3}^{\pm}$ & \multicolumn{2}{c|}{+} & \multicolumn{2}{c|}{-} & \multicolumn{2}{c|}{+} & \multicolumn{2}{c|}{-} & \multicolumn{2}{c|}{+} & \multicolumn{2}{c|}{-} & \multicolumn{2}{c|}{+} & \multicolumn{2}{c|}{-} & $\dim=2$ 
\\ 
\hline
$E_{P_4}^{\pm}$ &     +  $E_{8,0}^*$    &      -     &     +      &          - &      +     &      -     &     +      &  -         &       +    &          - &       +    &     -      &      +     &      -     &     +      &          - & $\dim=1$ 
\\ 
\hline
Basis for $E_{8,0}^*$& $v$ & $\ldots$ & $\ldots$ & $\ldots$ & $\ldots$ & $\ldots$ & $\ldots$ & $\ldots$ & $\ldots$ & $\ldots$ & $\ldots$ & $\ldots$ & $\ldots$ & $\ldots$ & $\ldots$ & $\ldots$ &  $\begin{array}{ll}&P_1=J_1J_2J_3J_4\\ &P_2=J_1J_2J_5J_6\\ &P_3=J_1J_2J_7J_8\\ & P_4=J_1J_3J_5J_7\end{array}$
\\ 
\hline
$G_I$ &  & $J_1J_2$ & $J_8$ & $J_7$ & $J_6$ & $J_5$ & $J_1J_3$ & $J_1J_4$ & $J_4$ & $J_3$ & $J_1J_5$ & $J_1J_6$ & $J_1J_7$ & $J_1J_8$ & $J_2$ & $J_1$ &  
\\ 
\hline
\end{tabular}
}
\end{center}

\vskip0.5cm

The tables for $(r,s)\in\{(0,8),(4,4)\}$ are the same. 
There are no operators $\PP$ leaving invariant the space $E=\oplus^pE_{r,s}^*$, $(r,s)\in\{(8,0),(0,8),(4,4)\}$ that means that there are no restrictions on group of automorphisms acting on an admissible module. We conclude that $\Aut^0(U)=\GL(p,\mathbb R)$ for $U=\oplus^p V^{r,s;+}_{min}$ and for $(r,s)\in\{(8,0),(0,8),(4,4)\}$.


\subsubsection{$\dim_{\mathbb R}(E_{r,s}^*)=2$: cases $\n_{1,0}(U)$, $\n_{0,1}(U)$; $\n_{0,7}(U)$, $\n_{4,3}(U)$}


\begin{center}
\scalebox{0.7}{
\begin{tabular}{|c|c|c|}
\hline
$\n_{1,0}$ & & dim=2 \\
\hline
Basis & $x_1=v$ & \\
& $x_2=J_1v$ & \\
\hline
\end{tabular}
\qquad
\begin{tabular}{|c|c|c|}
\hline
$\n_{0,1}$ & & dim=2 \\
\hline
Basis & $x_1=v$ & \\
& $\textcolor{red}{x_2=J_1v}$ & \\
\hline
\end{tabular}
}
\end{center}

\vskip0.5 cm

Let $U=\oplus^pV^{1,0;+}_{min}$. In this case $A\in\Aut^0(\n_{1,0}(U))$ has to fulfil the relation 
$A^TJ_{1}A=J_{1}$ for 
$J_{1}=\diag_p\begin{pmatrix}0&-1\\1&0\end{pmatrix}$. We conclude $\Aut^0(\n_{1,0}(U))\cong\Sp(2p;\mathbb R)$

Let $U=\oplus^pV^{0,1;N}_{min}$. Then $A^T\eta J_{1}A=\eta J_{1}$, where $\eta J_{1}=\diag_p\begin{pmatrix}0&1\\-1&0\end{pmatrix}$. 
It follows that $\Aut^0(\n_{0,1}(U))=\Sp(2p;\mathbb R)$ as in the previous case. 

\vskip0.5cm

\begin{center}
\centering
\scalebox{0.7}{
\begin{tabular}{|c|c|c|c|c|c|c|c|c|c|}
\hline
$V^{0,7}_{min}$ & \multicolumn{8}{c|}{}                                                                         & $\dim=16$ 
\\ 
\hline
$E_{P_1}^{\pm}$ & \multicolumn{4}{c|}{+}                         & \multicolumn{4}{c|}{-}                         & $\dim=8$  \\ \hline
$E_{P_2}^{\pm}$ & \multicolumn{2}{c|}{+} & \multicolumn{2}{c|}{-} & \multicolumn{2}{c|}{+} & \multicolumn{2}{c|}{-} & $\dim=4$ 
\\ 
\hline
$E_{P_3}^{\pm}$ &   +    $E_{0,7}^*$    &     -      &   +        &      -     &    +       &     -      &    +       &      -     & $\dim=2$  
\\ 
\hline
Basis for $E_{0,7}^*$ & $x_1=v$ & $\ldots$& $\ldots$& $\ldots$& $\ldots$& $\ldots$& $\ldots$& $\ldots$& $\begin{array}{ll} & P_1=J_1J_2J_3J_4\\& P_2=J_1J_2J_5J_6\end{array} $
\\
& \textcolor{red}{$x_2=J_1J_2J_7v$}& $\ldots$ & $\ldots$ & $\ldots$ & $\ldots$ & $\ldots$ & $\ldots$ &$\ldots$ & $\begin{array}{ll} &P_3=J_1J_3J_5J_7\\&\PP=J_1J_2J_7 \end{array}$
\\
\hline
$G_I$& & $J_7$ & $J_6$ & $J_5$ & $J_4$ & $J_3$ & $J_2$ & $J_1$ & \\
\hline
\end{tabular}
}
\end{center}

\vskip0.5cm

We need to check the condition
\begin{equation}\label{eq:07}
A^T\eta\PP A=\eta\PP\quad\Longleftrightarrow\quad A^T\diag_p\begin{pmatrix}0&1\\1&0\end{pmatrix}A=\diag_p\begin{pmatrix}0&1\\1&0\end{pmatrix}.
\end{equation}
In the basis $y_1=x_1+x_2$ and $y_2=x_1-x_2$ for $E_{0,7}^*\subset V^{0,7;N}_{min}$ condition~\eqref{eq:07} became
$$
A^T\diag_p\begin{pmatrix}1&0\\0&-1\end{pmatrix}A=\diag_p\begin{pmatrix}1&0\\0&-1\end{pmatrix}.
$$
We conclude that $\Aut^0(\n_{0,7}(U))\cong \Orth(p,p,\mathbb{R})$ for $U=\oplus^pV^{0,7;N}_{min}$.
\\

For the case $\n_{4,3}(U)$ the system of involutions and operators are similar. We conclude that $\Aut^0(\n_{4,3}(U))=\Orth(p,p,\mathbb R)$.


\subsection{Modules over $\mathbb C$}

In this section we first consider the cases when the operators $\eta \PP_k$ are complex linear, or in other words they commute with the almost complex structure $\ii$. In this cases the group of automorphisms is related to unitary  transformations. The last part of the cases is related to the situations when the operators  $\eta \PP_k$ are not complex linear.


\subsubsection{$\dim_{\mathbb C}(E_{r,s}^*)=1$: cases $\n_{7,1}(U)$, $\n_{3,5}(U)$; $\n_{6,0}(U)$, $\n_{2,4}(U)$}


\begin{center}
\centering
\scalebox{0.7}{
\begin{tabular}{|c|c|c|c|c|c|c|c|c|c|}
\hline
$V^{7,1}_{min}$ & \multicolumn{8}{c|}{}                                                                         & $\dim=16$ 
\\ 
\hline
$E_{P_1}^{\pm}$ & \multicolumn{4}{c|}{+}                         & \multicolumn{4}{c|}{-}                         & $\dim=8$  
\\ 
\hline
$E_{P_2}^{\pm}$ & \multicolumn{2}{c|}{+} & \multicolumn{2}{c|}{-} & \multicolumn{2}{c|}{+} & \multicolumn{2}{c|}{-} & $\dim=4$ 
\\ 
\hline
$E_{P_3}^{\pm}$ &   +    $E_{7,1}^*$    &     -      &   +        &      -     &    +       &     -      &    +       &      -     & $\dim=2$  
\\ 
\hline
Basis for $E_{7,1}^*$& $x_1=v$ & $\ldots$ & $\ldots$  & $\ldots$  & $\ldots$ & $\ldots$ & $\ldots$  & $\ldots$ & $\begin{array}{ll}&P_1=J_1J_2J_4J_5\\&P_2=J_1J_2J_6J_7\end{array}$
\\
 & \textcolor{red}{$x_2=\ii v$} & $\ldots$ & $\ldots$ & $\ldots$ & $\ldots$ &$\ldots$ &$\ldots$ &$\ldots$ & $\begin{array}{ll}&P_3=J_1J_3J_5J_7\\&P_4=J_1J_2J_3\\ &\ii=J_{1}J_{2}J_{3}J_{8}\\ &\PP=J_1J_2J_3\end{array}$
 \\
\hline
$G_I$& & $J_3$ & $J_6$ & $J_7$ & $J_4$ & $J_5$ &  $J_2$ & $J_1$  & \\
\hline
\end{tabular}
}
\end{center}
\vskip0.5cm

We have $E_{7,1}^*=E_{\PP}^{+1} \oplus E_{\PP}^{-1}$, with 
$
E_{\PP}^{+1} = \text{span}\lbrace v \rbrace$ and $E_{\PP}^{-1} = \text{span} \lbrace \ii v \rbrace$.
We let
$
U= (\stackrel{p}\oplus V_{min}^{7,1;+})\oplus (\stackrel{q}\oplus V_{min}^{7,1;-})$. Since $\eta\PP$ is complex linear, we need to check
\begin{equation}\label{eq:J_8}
A^T\eta \PP A=\eta \PP\quad\Longleftrightarrow\quad
\bar A^T_{\CC}\Id_{p,q}A_{\CC}=\Id_{p,q}.
\end{equation}
Here we used the embedding~\eqref{eq:phoCC} and denoted by $A_{\CC}$ the matrix with complex entries such that $\rho_{\CC}(A_{\CC})=A$.
It shows that $A\in \U(p,q)$ and $\Aut^0(\n_{7,1}(U))\cong \U(p,q)$.
\\

The table and calculations for $\n_{3,5}(U)$ are analogous to $\n_{7,1}(U)$ and we conclude $\Aut^0(\n_{3,5}(U))=\U(p,q)$ for $U= (\stackrel{p}\oplus V_{min}^{3,5;+})\bigoplus (\stackrel{q}\oplus V_{min}^{3,5;-})$.
\\

We consider now cases when the operators $\eta\PP$ are not complex linear.

\vskip0.5cm
\begin{center}
\centering
\scalebox{0.7}{
\begin{tabular}{|c|c|c|c|c|c|}
\hline
$V^{6,0}_{min}$ & \multicolumn{4}{c|}{}                         & $\dim=8$ 
\\ 
\hline
$E_{P_1}^{\pm}$ & \multicolumn{2}{c|}{+} & \multicolumn{2}{c|}{-} & $\dim=4$ 
\\ 
\hline
$E_{P_2}^{\pm}$  &      + $E^*_{6,0}$    &     -      &    +       &     -      & $\dim=2$ 
\\ 
\hline
Basis for $E^*_{6,0}$& $x_1=v$ &  $\ldots$ &  $\ldots$ & $\ldots$ & $\begin{array}{ll}&P_1=J_1J_2J_3J_4\\&P_2=J_1J_2J_5J_6\\& P_3=J_1J_3J_5\end{array}$
\\
& $x_2=\ii v$ & $\ldots$ &  $\ldots$ & $\ldots$ & $\begin{array}{ll}& \ii= J_1J_2\\& \PP=J_1J_3J_5\end{array}$\\
\hline
$G_I$ & & $ J_5$ & $J_3$ & $J_1$ & \\ 
\hline
\end{tabular}
}
\end{center}

\vskip0.5cm

We have $E_{6,0}^*=E_{\PP}^+ \oplus E_{\PP}^-$  with
$E_{\PP}^+ = \text{span}\lbrace v \rbrace$,  
$E_{\PP}^- = \text{span} \lbrace \ii v \rbrace$ and $A\in\GL(1,\CC)$. We also have that $\PP\ii=-\ii\PP$ with $\PP=\begin{pmatrix}1&0\\0&-1\end{pmatrix}$. We obtain
\begin{equation}\label{eq:J135}
A^T\PP A=\PP\quad\Longleftrightarrow\quad \PP\rho_{\CC}(\bar A_{\CC}^T) \PP \rho_{\CC}(A_{\CC})=\Id.
\end{equation} 
By making use of~\eqref{eq:conjugation1}, we conclude that
$
A^T_{\CC}A_{\CC}=\Id$. 
For general admissible module $U= (\stackrel{p}\oplus V_{min}^{6,0;+})$ we obtain
$\Aut^0(\n_{6,0}(U))=\Orth(p,\mathbb C)$.
\\

Calculations and the table for $\n_{2,4}(U)$ are similar to the case $\n_{6,0}(U)$. Thus $\Aut^0(\n_{2,4}(U))=\Orth(p,\mathbb C)$. 


\subsubsection{$\dim_{\mathbb C}(E_{r,s}^*)=2$: cases $\n_{1,7}(U)$, $\n_{5,3}(U)$; $\n_{2,0}(U)$, $\n_{0,2}(U)$; $\n_{0,6}(U)$, $\n_{4,2}(U)$}


\begin{center}
\centering
\scalebox{0.7}{
\begin{tabular}{|c|c|c|c|c|c|c|c|c|c|}
\hline
$V^{1,7}_{min}$ & \multicolumn{8}{c|}{}                                                                         & $\dim=32$ 
\\ 
\hline
$E_{P_1}^{\pm}$ & \multicolumn{4}{c|}{+}                         & \multicolumn{4}{c|}{-}                         & $\dim=16$ 
\\ 
\hline
$E_{P_2}^{\pm}$ & \multicolumn{2}{c|}{+} & \multicolumn{2}{c|}{-} & \multicolumn{2}{c|}{+} & \multicolumn{2}{c|}{-} & $\dim=8$ 
\\ 
\hline
$E_{P_3}^{\pm}$ &       + $E^*_{1,7}$    &     -      &          + &      -     &        +   &     -      &          + &    -       & $\dim=4$ 
\\ 
\hline
Basis for $E^*_{1,7}$& $x_1=v$ & \ldots & \ldots & \ldots & \ldots & \ldots & \ldots & \ldots & $\begin{array}{ll}&P_1=J_{2}J_{3}J_{4}J_{5}\end{array}$
\\
& \textcolor{red}{$x_2=\ii v$} & $\ldots$ & $\ldots$ & $\ldots$ &$\ldots$ & $\ldots$ & $\ldots$ & $\ldots$ & $\begin{array}{ll}& P_2=J_{2}J_{3}J_{6}J_{7}\\&P_3=J_{2}J_{4}J_{6}J_{8}\end{array}$
\\
 &  $x_3=J_{1}v$ & $\ldots$ & $\ldots$ & $\ldots$ &$\ldots$ & $\ldots$ & $\ldots$ & $\ldots$  &$\ii=J_{1}J_{6}J_{7}J_{8}$  
 \\
 & \textcolor{red}{$x_4=\ii J_1v$}  & $\ldots$ & $\ldots$ & $\ldots$ &$\ldots$ & $\ldots$ & $\ldots$ & $\ldots$ & $\PP=J_1$
 \\ 
\hline
$G_I$& & $J_{8}$ & $J_{7}$ & $J_{6}$ & $J_{5}$ & $J_{4}$ & $J_{3}$ &  $J_{2}$   & \\
\hline
\end{tabular}
}
\end{center}

\vskip0.5cm
We consider the minimal admissible module first. We have $A\in\GL(2,\CC)$ and $\eta\PP\ii=\ii\eta\PP$. Thus the complex linear map $\eta\PP=\begin{pmatrix}0&-1\\1&0\end{pmatrix}$  is skew-Hermitian. As it was noticed, from a qualitative point of view, consideration of skew-Hermitian forms (up to isomorphism) provides no new classical groups, since multiplication by $\mathbf i$ renders a skew-Hermitian form Hermitian, and vice versa. The form $\mathbf i \eta\PP$ is Hermitian of the signature $(1,1)$ and the condition
$
\bar A^T_{\CC}\mathbf i \eta\PP A_{\CC}=\mathbf i \eta\PP
$
leads to $\Aut^0(\n_{1,7}(V_{min}^{1,7;N}))\cong \U(1,1)$.  It shows that $\Aut^0(\n_{1,7}(U))\cong \U(p,p)$ for a general admissible module. 
\\

The calculations and the table for $\n_{5,3}(U)$ are similar to $\n_{1,7}(U)$ and we conclude that $\Aut^0(\n_{5,3}(U))\cong \U(p,p)$.

\vskip0.5cm

\begin{center}
\scalebox{0.7}{
\begin{tabular}{|c|c|c|}
\hline
$\n_{2,0}$ & & dim=4 \\
\hline
Basis & $x_1=v$ & \\
& $x_2=\ii x_1$ & $\ii=J_1J_2$\\
& $x_3=J_1v$ & $\PP=J_1$\\
& $x_4=\ii x_3$ & \\
\hline
\end{tabular}
\qquad
\begin{tabular}{|c|c|c|}
\hline
$\n_{0,2}$ & & $\dim=4$ \\
\hline
Basis & $x_1=v$ & \\
& $x_2=\ii v$ & $\ii=J_1J_2$\\
& $\textcolor{red}{x_3=J_1v}$ & $\PP=J_1$\\
& $\textcolor{red}{x_4=\ii v}$ & \\
\hline
\end{tabular}
}
\end{center}

\vskip0.5cm

We make calculations for $U=V_{min}^{2,0;+}$. We have  $A\in\GL(2,\CC)$, $\PP\ii=-\ii\PP$, and  
$$
\PP=
\begin{pmatrix}
0&0&-1&0
\\
0&0&0&1
\\
1&0&0&0
\\
0&-1&0&0
\end{pmatrix}=\diag_2\begin{pmatrix}1&0\\0&-1\end{pmatrix}\cdot \begin{pmatrix}0&-\Id_2\\\Id_2&0\end{pmatrix}
$$
The condition $
A^T \PP A=\PP$ is equivalent to
$$
\diag_2\begin{pmatrix}1&0\\0&-1\end{pmatrix}\rho_{\CC}(\bar A_{\CC}^T)\diag_2\begin{pmatrix}1&0\\0&-1\end{pmatrix}\begin{pmatrix}0&-\Id_2\\\Id_2&0\end{pmatrix}\rho_{\CC}(A_{\CC})=\begin{pmatrix}0&-\Id_2\\\Id_2&0\end{pmatrix}.
$$
Observation~\eqref{eq:conjugation1} implies that 
$$
A_{\CC}^T\begin{pmatrix}0&-1\\1&0\end{pmatrix}A_{\CC}=\begin{pmatrix}0&-1\\1&0\end{pmatrix}\quad\Longrightarrow\quad
\Aut^0(\n_{2,0}(V^{2,0;+}_{\min}))\cong\Sp(2;\CC).
$$
For the general module $U=\oplus^p V^{2,0;+}_{min}$, we obtain $\Aut^0(\n_{2,0}(U))\cong\Sp(2p;\CC)$.
\\

Let now $U=\oplus^pV^{0,2;N}_{min}$. For the neutral metric $\eta$ we obtain
$$\eta \PP=-
\begin{pmatrix}
0&0&-1&0
\\
0&0&0&1
\\
1&0&0&0
\\
0&-1&0&0
\end{pmatrix}.
$$
Thus by calculations for $
A^T \eta\PP A=\eta\PP$ as above we get
$\Aut^0(\n_{0,2}(U))=\Sp(2p,\CC)$.

\vskip0.5cm

\begin{center}
\centering
\scalebox{0.7}{
\begin{tabular}{|c|c|c|c|c|c|}
\hline
$V^{0,6}_{\min}$ & \multicolumn{4}{c|}{}                         & $\dim=16$ 
\\ 
\hline
$E_{P_1}^{\pm}$ & \multicolumn{2}{c|}{+} & \multicolumn{2}{c|}{-} & $\dim=8$ 
\\ 
\hline
$E_{P_2}^{\pm}$  &      +  $E^*_{0,6}$   &     -      &    +       &     -      & $\dim=4$ 
\\ 
\hline
Basis for $E^*_{0,6}$& $x_1=v$ & $\ldots$ & $\ldots$ & $\ldots$& $P_1=J_1J_2J_3J_4$
\\
 & $x_2=\ii v$ & $\ldots$ & $\ldots$ & $\ldots$ &  $P_2=J_1J_2J_5J_6$
 \\
& \textcolor{red}{$x_3=J_1J_3J_5v$} & $\ldots$ & $\ldots$  & $\ldots$ & $\ii=J_1J_2$
\\ 
& \textcolor{red}{$x_4=\ii x_3$} & $\ldots$ & $\ldots$  & $\ldots$ & $\PP=J_1J_3J_5$
\\
\hline
$G_I$& & $J_5$ & $J_3$ & $J_1$ & 
\\
\hline
\end{tabular}
}
\end{center}

\vskip0.5cm
We start from $U=V^{0,6}_{\min}$.
Note that $A\in\GL(2;\CC)$, $\eta\PP\ii=-\ii\eta\PP$ and 
$$
\eta\PP=
\begin{pmatrix}
0&0&-1&0
\\
0&0&0&1
\\
-1&0&0&0
\\
0&1&0&0
\end{pmatrix}=\diag_2\begin{pmatrix}1&0\\0&-1\end{pmatrix}\cdot \begin{pmatrix}0&\Id_2\\\Id_2&0\end{pmatrix}.
$$
Thus 
$$
A^T\eta\PP A=\eta\PP\quad\Longleftrightarrow\quad
A^T_{\CC}
\begin{pmatrix}0&1\\1&0\end{pmatrix}
A_{\CC}=
\begin{pmatrix}0&1\\1&0\end{pmatrix}.
$$
The matrix $\begin{pmatrix}0&1\\1&0\end{pmatrix}$ is symmetric of signature $(1,1)$. Thus $\Aut^0(\n_{0,6}(V_{min}^{0,6;N}))\cong\Orth(1,1;\CC)\cong\Orth(2;\CC)$.
We obtain $\Aut^0(\n_{0,6}(U))\cong\Orth(2p;\CC)$ for $U=\oplus^p V^{0,6;N}_{min}$, w.
\\

The calculations and the table for $\n_{4,2}(U)$ are similar to $\n_{0,6}(U)$ and we conclude that $\Aut^0(\n_{4,2}(U))=\Orth(2p,\CC)$.


\subsection{Modules over $\mathbb H$}



\subsubsection{$\dim_{\mathbb H}(E_{r,s}^*)=1$: cases $\n_{4,0}(U)$, $\n_{0,4}(U)$, $\n_{6,2}(U)$, $\n_{2,6}(U)$, $\n_{6,1}(U)$, $\n_{1,6}(U)$, $\n_{5,2}(U)$, $\n_{2,5}(U)$, $\n_{5,0}(U)$, $\n_{1,4}(U)$, $\n_{3,0}(U)$, $\n_{3,6}(U)$, $\n_{7,2}(U)$.}


\begin{center}
\scalebox{0.7}
{
\centering
\begin{tabular}{|c|c|c|c|}
\hline
$V^{4,0}_{min}$ & \multicolumn{2}{c|}{} & $\dim=8$ \\ \hline
$E_{P_1}^{\pm}$ &     +   $E^*_{4,0}$   &     -      & $\dim=4$ \\
\hline
Basis for $E^*_{4,0}$& $x_1=v$ &   $\ldots$ & $P_1=J_1J_2J_3J_4$ \\
 & $x_2=\ii v$ &  $\ldots$  & $\ii=J_1J_2$\\  
& $x_3=\jj v$ & $\ldots$ & $\jj=J_2J_3$\\
& $x_4=\kk v$ & $\ldots$ & $\kk=J_3J_1$\\
\hline
$G_I$& & $J_1$ & \\ 
 \hline
\end{tabular}
}
\end{center}
\vskip0.5cm

The table for $\n_{0,4}(V^{0,4}_{min})$ is analogous, with
$\ii=J_1J_2 $, $\jj=J_2J_3$, $\kk=J_1J_3$.

\vskip05cm
\begin{center}
\centering
\scalebox{0.7}{
\begin{tabular}{|c|c|c|c|c|c|c|c|c|c|}
\hline
$V^{6,2}_{min}$ & \multicolumn{8}{c|}{}& $\dim=32$ 
\\ 
\hline
$E_{P_1}^{\pm}$ & \multicolumn{4}{c|}{+}                         & \multicolumn{4}{c|}{-}                         & $\dim=16$ 
\\ 
\hline
$E_{P_2}^{\pm}$ & \multicolumn{2}{c|}{+} & \multicolumn{2}{c|}{-} & \multicolumn{2}{c|}{+} & \multicolumn{2}{c|}{-} & $\dim=8$ 
\\ 
\hline
$E_{P_3}^{\pm}$ &       +  $E^*_{6,2}$  &     -      &          + &      -     &        +   &     -      &          + &    -       & $\dim=4$ 
\\ 
\hline
Basis for $E^*_{6,2}$& $x_1=v$ & $\ldots$ & $\ldots$ & $\ldots$ & $\ldots$ & $\ldots$ & $\ldots$ & $\ldots$ & $\begin{array}{ll}&P_1=J_1J_2J_3J_4\\&P_2=J_1J_2J_5J_6\\&P_3=J_1J_2J_7J_8\end{array}$
\\
 & $x_2=\ii v$ & $\ldots$ & $\ldots$ & $\ldots$ & $\ldots$ & $\ldots$ & $\ldots$ & $\ldots$ & $\ii=J_1J_2$
 \\
& \textcolor{red}{$x_3=\jj v$} & $\ldots$ & $\ldots$ & $\ldots$ & $\ldots$ &$\ldots$ & $\ldots$ & $\ldots$ & $\jj=J_1J_3J_5J_7$
\\
& \textcolor{red}{$x_4=\kk v$} & $\ldots$ & $\ldots$ & $\ldots$ & $\ldots$ & $\ldots$ & $\ldots$ & $\ldots$ & $\kk=J_2J_3J_5J_7$
\\ 
\hline
$G_I$& & $J_7$ & $J_5$ & $J_1J_3$ & $J_3$ & $J_1J_5$ &  $J_1J_7$ & $J_1$  & \\
\hline
\end{tabular}
}
\end{center}

\vskip0.5cm

The table for $\n_{2,6}$ is similar.
In all 4 cases there are no conditions except of requirement to commute with the quaternion structure. We conclude that 
$
\Aut^0(\n_{4,0}(U))=\Aut^0(\n_{0,4}(U))=\Aut^0(\n_{6,2}(U))=\Aut^0(\n_{2,6}(U))=\GL(p,\HH)$.

\vskip0.5cm

\begin{center}
\centering
\scalebox{0.7}{
\begin{tabular}{|c|c|c|c|c|c|}
\hline
$V^{1,6}_{min}$ & \multicolumn{4}{c|}{}                         & $\dim=16$ 
\\ 
\hline
$E_{P_1}^{\pm}$ & \multicolumn{2}{c|}{+} & \multicolumn{2}{c|}{-} & $\dim=8$ 
\\ 
\hline
$E_{P_2}^{\pm}$  &      +  $E^{*}_{1,6}$   &     -      &    +       &     -      & $\dim=4$ 
\\ 
\hline
Basis for $E^{*}_{1,6}$
& $x_1=v$ &\ldots & \ldots & \ldots&  $\begin{array}{ll}& P_1=J_2J_3J_4J_5\\&P_2=J_2J_3J_6J_7\\&P_3=J_1J_2J_3\end{array}$
\\
& \textcolor{red}{$x_2=\ii v$} & \ldots & \ldots& \ldots & $\ii=J_1J_2J_4J_6$
\\
 & $x_3=\jj v$ & \ldots & \ldots & \ldots & $\jj=J_2J_3$  
 \\
& \textcolor{red}{$x_4=\kk v$} & \ldots & \ldots  & \ldots &$\begin{array}{ll}&\kk=J_1J_3J_4J_6\\&\PP=J_1\end{array}$  \\
\hline
$G_I$& & $J_6$ & $J_4$ & $J_2$ & \\
\hline
\end{tabular}
}
\end{center}

\vskip0.5cm 

Observe that $P_3=-\Id$ on $E_{1,6}^*$ according to the agreement that $E_{1,6}^*\subset V_{min;+}^{1,6;N}$ where $\Omega^{1,6}=\Id$ on $V_{min;+}^{1,6;N}$. 
We consider  $U=\oplus^p V^{1,6;N}_{min;+}$. Since 
$
\eta \PP=
\diag_p\mathbf{j}$
and 
$$
A^T\eta \PP A=\eta \PP\quad\Longleftrightarrow\quad \bar A^T_{\HH}\diag_p\mathbf{j} A_{\HH}=\diag_p\mathbf{j},
$$
we conclude $\Aut^0(\n_{1,6})\cong\Orth^*(2p)$.

\vskip0.5cm

\begin{center}
\centering
\scalebox{0.7}{
\begin{tabular}{|c|c|c|c|c|c|}
\hline
$V^{5,2}_{min}$ & \multicolumn{4}{c|}{}                         & $\dim=16$ 
\\ 
\hline
$E_{P_1}^{\pm}$ & \multicolumn{2}{c|}{+} & \multicolumn{2}{c|}{-} & $\dim=8$ 
\\ 
\hline
$E_{P_2}^{\pm}$  &      +  $E^*_{5,2}$   &     -      &    +       &     -      & $\dim=4$ 
\\ 
\hline
Basis for $E^*_{5,2}$& $x_1=v$ & \ldots & \ldots  & \ldots  &  $\begin{array}{ll}&P_1=J_1J_2J_3J_4\\ &P_2=J_1J_2J_6J_7\\&P_3=J_1J_2J_5\end{array}$
\\
& \textcolor{red}{$x_2=\ii v$} & \ldots  & \ldots  & \ldots  &$\ii=J_2J_3J_5J_6$ 
\\
 & $x_3=\jj v$ & \ldots  & \ldots  & \ldots  & $\jj=J_1J_2$ 
 \\
& \textcolor{red}{$x_4=\kk v$} & \ldots  & \ldots  & \ldots  & $\begin{array}{ll}&\kk=J_1J_3J_5J_6\\&\PP=J_5\end{array}$ 
\\
\hline
$G_I$& & $J_7$ & $J_3$ & $J_1$ & 
\\
\hline 
\end{tabular}
}
\end{center}

\vskip0.5cm

Observe that $P_3=-\Id$ according to $E_{5,2}^*\subset V_{min;+}^{5,2}$ where $\Omega^{5,2}=\Id$ on $V_{min;+}^{5,2;N}$. The calculations are similar to the case of $\n_{1,6}(U)$ show that $\Aut^0(\n_{5,2})\cong\Orth^*(2p)$.

\vskip0.5cm

\begin{center}
\centering
\scalebox{0.7}{
\begin{tabular}{|c|c|c|c|c|c|}
\hline
$V^{6,1}_{min}$ & \multicolumn{4}{c|}{}                         & $\dim=16$ 
\\ 
\hline
$E_{P_1}^{\pm}$ & \multicolumn{2}{c|}{+} & \multicolumn{2}{c|}{-} & $\dim=8$ 
\\ 
\hline
$E_{P_2}^{\pm}$  &      +  $E^*_{6,1}$   &     -      &    +       &     -      & $\dim=4$ 
\\ 
\hline
Basis for $E^*_{6,1}$ & $x_1=v$ &  \ldots  & \ldots &  \ldots &  $\begin{array}{ll}&P_1=J_1J_2J_3J_4\\&P_2=J_1J_2J_5J_6\\&P_3=J_1J_3J_5\end{array}$
\\
& $x_2=\ii v$ & \ldots & \ldots & \ldots & $\ii=J_1J_2$
\\ 
 & \textcolor{red}{$x_3=\jj v$} & \ldots  & \ldots & \ldots& $\jj=J_1J_3J_5J_7$
\\
 & \textcolor{red}{$x_4=\kk v$}  & \ldots  & \ldots & \ldots & $\begin{array}{ll}&\kk=J_2J_3J_5J_7\\& \PP=J_7\end{array}$  
\\
\hline
$G_I$& & $J_5$ & $J_3$ & $J_1$ & \\
\hline
\end{tabular}
}
\end{center}


Observe that $E_{6,1}^*=E_{P_3}^+ \oplus E_{P_3}^-$, with 
$
E_{P_3}^+ = \text{span}\lbrace v, \kk v \rbrace$ and  $E_{P_3}^- = \text{span} \lbrace \ii v, \jj v \rbrace$.
We obtain
$
\eta \PP=
\diag_p\mathbf{j}
$ for $U=\oplus^p V^{6,1;N}_{min}$.
Thus, $\Aut^0(\n_{6,1})\cong\Orth^*(2p)$.
\\

The calculations and the table for $\n_{2,5}(U)$ are similar to $\n_{6,1}(U)$ and we conclude $\Aut^0(\n_{2,5}(U))\cong\Orth^*(2p)$.

\vskip0.5cm

\begin{center}
\scalebox{0.7}
{
\begin{tabular}{|c|c|c|c|}
\hline
$V^{5,0}_{min}$ & \multicolumn{2}{c|}{} & $\dim=8$ 
\\ 
\hline
$E_{P_1}$ &     +  $E^*_{5,0}$    &     -      & $\dim=4$ 
\\
\hline
 Basis for $E^*_{5,0}$ & $x_1=v$ &  $\ldots$ & $\begin{array}{ll}&P_1=J_2J_3J_4J_5\\&P_2=J_1J_2J_3\end{array}$
 \\
 & $x_2=\ii v$ &   $\ldots$& $\ii=J_3J_4$
 \\ 
& $x_3=\jj v$ &  $\ldots$ & $\jj=J_3J_2$
\\
& $x_4=\kk v$ & $\ldots$ &  $\begin{array}{ll}&\kk=J_4J_2\\ &\PP=J_1\end{array}$
\\
\hline
$G_I$& & $J_5$ & \\
\hline 
\end{tabular}
}
\end{center}

\vskip0.5cm

Note that $E_{5,0}^*=E_{P_2}^+ \oplus E_{P_2}^-$ with
$
E_{P_2}^+ = \text{span}\lbrace v, \jj v \rbrace$ and 
$E_{P_2}^- = \text{span} \lbrace \ii v, \kk v \rbrace$.
Thus $\PP=\diag_p
\mathbf{j}$ and 
we conclude that $\Aut^0(\n_{5,0}(U))\cong \Orth^*(2p)$ for $U=\oplus^p V^{5,0;+}_{min}$.
\\

For the case $\n_{1,4}(U)$ we use the quaternion structure $
\ii=J_3J_4$,
$\jj=J_3J_2$,
$\kk=J_2J_4$.
The rest of calculations are similar to $\n_{5,0}(U)$ and we obtain $\Aut^0(\n_{1,4}(U))\cong \Orth^*(2p)$.

\vskip0.5cm

\begin{center}
\scalebox{0.7}
{
\begin{tabular}{|c|c|c|}
\hline
$V^{3,0}_{min}$ & & dim=4 
\\
\hline
Basis & $x_1=v$ & $P_1=J_1J_2J_3$ 
\\
& $x_2=\ii v$ & $\ii=J_1J_2$
\\
& $x_3=\jj v$ & $\jj=J_2J_3$
\\
& $x_4=\kk v$ & $\begin{array}{ll}&\kk=J_3J_1\\&\PP=J_1J_2J_3\end{array}$
\\
\hline
\end{tabular}
}
\end{center}

\vskip0.5cm

Observe that $\PP=\Omega^{3,0}=\Id$. We obtain that 
$
A^T\PP A=\PP$ is equivalent to $\bar A^T_{\HH}\Id_{p,q} A_{\HH}=\Id_{p,q}$.
Thus $\Aut^0(\n_{3,0}(U))=\Sp(p,q)$ for $U=\oplus^p \big(V_{\min;+}^{3,0;+}\big)\oplus\big(V_{\min;+}^{3,0;-}\big)$.

\vskip0.5cm

\begin{center}
\centering
\scalebox{0.7}{
\begin{tabular}{|c|c|c|c|c|c|c|c|c|c|}
\hline
$V^{3,6}_{min}$ & \multicolumn{8}{c|}{}                                                                         & $\dim=32$ 
\\ 
\hline
$E_{P_1}^{\pm}$ & \multicolumn{4}{c|}{+}                         & \multicolumn{4}{c|}{-}                         & $\dim=16$ 
\\ 
\hline
$E_{P_2}^{\pm}$ & \multicolumn{2}{c|}{+} & \multicolumn{2}{c|}{-} & \multicolumn{2}{c|}{+} & \multicolumn{2}{c|}{-} & $\dim=8$ 
\\ 
\hline
$E_{P_3}^{\pm}$ &       +  $E^{8}_{3,6}$  &     -      &          + &      -     &        +   &     -      &          + &    -       & $\dim=4$ 
\\ 
\hline
Basis for $E^{8}_{3,6}$& $x_1=v$ & $\ldots$ & $\ldots$& $\ldots$& $\ldots$& $\ldots$ & $\ldots$ & $\ldots$ & $\begin{array}{ll}&P_1=J_1J_2J_8J_9\\&P_2=J_4J_5J_8J_9\end{array}$
\\
& $x_2=\ii v$ & $\ldots$ & $\ldots$ & $\ldots$ & $\ldots$ & $\ldots$ & $\ldots$ & $\ldots$ & $\begin{array}{ll}&P_3=J_6J_7J_8J_9\\&P_4=J_3J_8J_9\end{array}$
\\ 
 & \textcolor{red}{$x_3=\jj v$}& $\ldots$ & $\ldots$ & $\ldots$ &$\ldots$ & $\ldots$ & $\ldots$ & $\ldots$ & $\begin{array}{ll}&\ii=J_8J_9\\&\jj=J_1J_4J_7J_8\end{array}$
 \\
 & \textcolor{red}{$x_4=\kk v$}& $\ldots$ & $\ldots$ & $\ldots$ & $\ldots$ & $\ldots$ & $\ldots$ & $\ldots$ & $\begin{array}{ll}&\kk=-J_1J_4J_7J_9\\&\PP=J_3J_8J_9\end{array}$
 \\
\hline
$G_I$& & $J_7$ & $J_4$ & $J_5$ & $J_1$ & $J_1J_6$ & $J_7J_8$ & $J_9$ & \\
\hline
\end{tabular}
}
\end{center}

\vskip0.5cm

We have $E_{3,6}^*=E_{P_4}^+ \oplus E_{P_4}^-$,  
with
$E_{P_4}^+ = \text{span}\lbrace v, \ii v \rbrace$, 
$E_{P_4}^- = \text{span} \lbrace \jj v, \kk v \rbrace$. Since $\eta \PP=\Id_{p,q}$, we obtain 
$$
A^T\eta \PP A=\eta \PP\quad\Longleftrightarrow\quad \bar A_{\HH}\Id_{p,q} A_{\HH}=\Id_{p,q}.
$$
So $\Aut^0(\n_{3,6}(U))=\Sp(p,q)$ for $U=\oplus^p \big(V_{\min}^{3,6;+}\big)\oplus\big(\oplus^q V_{\min}^{3,6;-}\big)$.
\\

The calculation and the table for $\n_{7,2}(U)$ are similar to the case $\n_{3,6}(U)$ and we conclude that
$\Aut^0(\n_{7,2}(U))=\Sp(p,q)$.


\subsubsection{$\dim_{\mathbb{H}}(E_{r,s}^*)=2$: cases $\n_{0,3}(U)$, $\n_{6,3}(U)$, $\n_{2,7}(U)$, $\n_{0,5}(U)$, $\n_{4,1}(U)$}


\begin{center}
\scalebox{0.7}{
\begin{tabular}{|c|c|c|}
\hline
$\n_{0,3}$ & & $\dim=8$ \\
\hline
Basis & $x_1=v$ & \\
& $x_2=\ii v$ &$\ii=J_2J_1$ \\
& $x_3=\jj v$ & $\jj=J_3J_2$\\
&  $x_4=\kk v$ & $\kk=J_1J_3$\\
&  $\textcolor{red}{x_5=J_1J_2J_3 v}$& \\
& $\textcolor{red}{x_6=\ii x_5}$& \\
& $\textcolor{red}{x_7=\jj x_5}$ & $\PP=J_1J_2J_3$\\
& $\textcolor{red}{x_8=\kk x_5}$ & \\
\hline
\end{tabular}
}
\end{center}

\vskip0.5cm

We make calculations on $V^{0,3}_{min}$ and note that $\eta\PP=
\begin{pmatrix}
0&1
\\
1&0
\end{pmatrix}
$. In the basis 
\begin{equation}\label{eq:basis03}
\begin{array}{lllllll}
&y_1=x_{1}+x_{5},\ &y_2=\ii y_1,\ &y_3=\jj y_1,\ &y_4=\kk y_1,
\\
&y_5=x_{1}-x_{5},\ &y_6=\ii y_5,\ &y_7=\jj y_5,\ &y_8=\kk y_5
 \end{array}
\end{equation}
the operator $\eta\PP$ takes the form $\Id_{1,1}$.
Thus
$$
A^T \eta \PP A=\eta \PP\quad\Longleftrightarrow \bar A^T_{\HH}\Id_{1,1}A_{\HH}=\Id_{1,1}\quad\Longrightarrow A\in\Sp(1,1).
$$ 
We conclude that  
$\Aut^0(\n_{0,3})\cong\Sp(p,p)$ for $U=\oplus^p V_{min}^{0,3;N}$.

\vskip0.5cm 

\begin{center}
\centering
\scalebox{0.7}{
\begin{tabular}{|c|c|c|c|c|c|c|c|c|c|}
\hline
$V^{6,3}_{min}$ & \multicolumn{8}{c|}{}                                                                         & $\dim=64$ \\ \hline
$E^{\pm}_{P_1}$ & \multicolumn{4}{c|}{+}                         & \multicolumn{4}{c|}{-}                         & $\dim=32$ \\ \hline
$E^{\pm}_{P_2}$ & \multicolumn{2}{c|}{+} & \multicolumn{2}{c|}{-} & \multicolumn{2}{c|}{+} & \multicolumn{2}{c|}{-} & $\dim=16$ \\ \hline
$E^{\pm}_{P_3}$ &   +   $E_{6,3}^*$    &     -      &          + &      -     &        +   &     -      &          + &    -       & $\dim=8$ \\ \hline
Basis for $E_{6,3}^*$& $x_1=v$ & $\ldots$ & $\ldots$& $\ldots$ & $\ldots$ & $\ldots$ & $\ldots$& 
$\ldots$ & $P_1=J_1J_2J_3J_4$
\\
& \textcolor{red}{$x_2=\ii v$} & $\ldots$ & $\ldots$ &$\ldots$ & $\ldots$ & $\ldots$ & $\ldots$ & $\ldots$ &$P_2=J_1J_2J_5J_6$
\\ 
& $x_3=\jj v$ & $\ldots$ & $\ldots$ & $\ldots$ & $\ldots$ & $\ldots$ & $\ldots$ & $\ldots$ & $P_3=J_1J_2J_7J_8$
\\ 
& \textcolor{red}{$x_4=\kk v$} & $\ldots$ & $\ldots$ & $\ldots$ & $\ldots$ & $\ldots$ & $\ldots$ & $\ldots$ & $ \ii =J_1J_3J_6J_8$ 
\\ 
 & \textcolor{red}{$x_5=J_2J_1J_9 v$}& $\ldots$ &$\ldots$ & $\ldots$ & $\ldots$ & $\ldots$ & $\ldots$ &$\ldots$ & 
$\jj=J_2J_1$ 
\\
 & $x_6=\ii J_2J_1J_9v$& $\ldots$ & $\ldots$ & $\ldots$ & $\ldots$ & $\ldots$ & $\ldots$ & $\ldots$& $\kk=J_2J_3J_6J_8$
 \\
 &\textcolor{red}{$x_7=\jj J_2J_1J_9v$} & $\ldots$ & $\ldots$ & $\ldots$ &$\ldots$ & $\ldots$ & $\ldots$ & $\ldots$ &  $\PP=J_2J_1J_9$
 \\
 & $x_8=\kk J_2J_1J_9v$& $\ldots$ & $\ldots$ &$\ldots$ & $\ldots$ & $\ldots$& $\ldots$ & $\ldots$ & 
 \\
\hline
$G_I$& & $J_7$ & $J_5$ & $J_1J_3$ & $J_3$ & $J_1J_5$ & $ J_1J_7$ & $J_1$ & \\
\hline
\end{tabular}
}
\end{center}

\vskip0.5cm

We have that 
$\eta\PP=\Id_{1,1}$ in the basis~\eqref{eq:basis03}. It leads to
$ \Aut^0(\n_{6,3}(U))\cong\Sp(p,p)$.
\\

The calculations are similar to $\n_{6,3}(U)$ and we conclude $\Aut^0(\n_{2,7}(U))\cong\Sp(p,p)$.

\vskip0.5cm

\begin{center}
\scalebox{0.7}{
\begin{tabular}{|c|c|c|c|}
\hline
$V^{0,5}_{min}$ & \multicolumn{2}{c|}{} & $\dim=16$ \\ \hline
$E_{P_1}$ &     +      &     -      & $\dim=8$ \\
\hline
Basis & $x_1=v$ & $\ldots$ & $P_1=J_1J_2J_3J_4$
\\
& $x_2=\ii v$ & $\ldots$  & 
\\
& $x_3=\jj v$& $\ldots$ &$\ii=J_1J_2$ 
\\ 
& $x_4=\kk v$  & $\ldots$ & $\jj=J_1J_3$
\\
&  \textcolor{red}{$x_5=J_5v$}  & $\ldots$ & $\kk=J_3J_2$
\\
& \textcolor{red}{$x_6=\ii J_5v$} & $\ldots$  & \\
 & \textcolor{red}{$x_7=\jj J_5v$} & $\ldots$  & $\PP=J_5$\\
& \textcolor{red}{$x_8=\kk J_5v$} &  $\ldots$ & \\
\hline
$G_I$& & $J_1$ & \\
\hline
\end{tabular}
}
\end{center}

\vskip0.5cm

We have 
$
\eta \PP=\begin{pmatrix}0&1\\-1&0\end{pmatrix}
$ on $V^{0,5}_{min}$, and $\eta\PP=\begin{pmatrix}\mathbf{j}&0\\0&{\mathbf j}\end{pmatrix}$ in the basis 
\begin{equation}\label{eq:basis05}
\begin{array}{lllllll}
&y_1=x_{1}+x_{3}-x_5+x_7,\ &y_2=\ii y_1,\ &y_3=\jj y_1,\ &y_4=\kk y_1,
\\
&y_5=x_{2}+x_{4}+x_6-x_8,\ &y_6=\ii y_5,\ &y_7=\jj y_5,\ &y_8=\kk y_5
 \end{array}
\end{equation}
It leads to
$
A^T \eta \PP A=\eta \PP$ that is equivalent to $\bar A^T_{\HH}\diag_p\begin{pmatrix}\mathbf{j}&0\\0&{\mathbf j}\end{pmatrix}A_{\HH}=\diag_p\begin{pmatrix}\mathbf{j}&0\\0&{\mathbf j}\end{pmatrix}$.
Thus we showed that 
$ \Aut^0(\n_{0,5}(U))\cong\Orth^*(4p)$.
\\

In the case $\n_{4,1}(U)$ we use the quaternion structure 
$\ii=J_1J_2$, $\jj=J_4J_2$, $\kk=J_1J_4$.
The rest of calculations are similar to $\n_{0,5}(U)$. Thus $\Aut^0(\n_{4,1}(U))\cong \Orth^*(4p)$.


\subsection{Modules over $\R$ caring a negative involution}



\subsubsection{Cases $\n_{1,1}(U)$, $\n_{3,3}(U)$}

In these cases there are no complex or quaternion structures, but only a negative involution leaving invariant the space $E^*_{r,s}$. We denote it by $\QQ$, and write in the tables. The involution $\QQ$ commutes with involutions of type (1)-(3) and therefore decomposes the space $E^*_{r,s}$ into its eigenspaces: $E^*_{r,s}=N_+\oplus N_-$. The admissible scalar product is degenerate in both $N_{\pm}$, but the decomposition still orthogonal with respect to the admissible product. In these cases the determination of $\Aut^0(\n_{r,s}(U))$ reduces to the calculations on $N_{\pm}$.

\begin{center}
\scalebox{0.7}{
\begin{tabular}{|c|c|c|}
\hline
$\n_{1,1}$ & & $\dim=4$ \\
\hline
Basis & $x_1=v$ & \\
& $\textcolor{red}{x_2=\QQ v}$ & $\QQ=J_1J_2$\\
& $x_3=J_1v$ & \\
& $\textcolor{red}{x_4=\QQ J_1v}$ & $\PP=J_1$\\
\hline
\end{tabular}
}
\end{center}

\vskip0.5cm

We have $V^{1,1;N}_{min}=N_+\oplus N_-$ with the bases
$$
N_+=\spn\{y_1=\frac{x_1+x_2}{2},\ y_2=\frac{x_3+x_4}{2}\},\quad
N_-=\spn\{y_3=\frac{x_1-x_2}{2},\ y_4=\frac{x_3-x_4}{2}\}
$$
Since $A\QQ=\QQ A$ we can decompose $A=A_+\oplus A_-$ such that $A_{\pm}\colon N_{\pm}\to N_{\pm}$. We have $\eta \PP \QQ=\QQ\eta \PP$ and 
$
\eta \PP=\diag_2\begin{pmatrix}
0&-1
\\
1&0\end{pmatrix}
$
in the basis $\{y_k\}_{k=1}^{4}$.
Thus the condition
$
A^{T}\eta \PP A=\eta \PP$ is equivalent to two independent conditions
$$
A^T_{\pm}\begin{pmatrix}
0&-1
\\
1&0
\end{pmatrix}A_{\pm}=\begin{pmatrix}
0&-1
\\
1&0
\end{pmatrix}.
$$
We conclude
$\Aut^0(\n_{1,1}(V^{1,1;N}_{min}))\cong\Sp(2,\R)\times \Sp(2,\R)$.  We obtain $\Aut^0(\n_{1,1}(U)\cong\Sp(2p,\R)\times \Sp(2p,\R)$ for a general admissible module $U=\oplus^pV^{1,1;N}_{min}$.

\vskip0.5cm

\begin{center}
\scalebox{0.7}
{
\begin{tabular}{|c|c|c|c|c|c|}
\hline
$V^{3,3}_{min}$ & \multicolumn{4}{c|}{}                         & $\dim=8$ 
\\ 
\hline
$E_{P_1}^{\pm}$ & \multicolumn{2}{c|}{+} & \multicolumn{2}{c|}{-} & $\dim=4$ 
\\ 
\hline
$E_{P_2}^{\pm}$  &      +  $E^*_{3,3}$   &     -      &    +       &     -      & $\dim=2$ 
\\ 
\hline
Basis for $E^*_{3,3}$ & $x_1=v$ & $\ldots$ & $\ldots$ & $\ldots$ &  $\begin{array}{ll} &P_1=J_1J_2J_5J_6 \\& P_2=J_1J_3J_4J_6\\ &P_3=J_1J_2J_3 \end{array}$
\\
& \textcolor{red}{$x_2=\QQ v$} & $\ldots$ & $\ldots$ & $\ldots$  & $\begin{array}{ll} & \QQ=J_1J_6 \\ & \PP=P_3 \end{array}$
\\
\hline
$G_I$& & $J_3$ & $J_2$ & $J_6$ & 
\\
\hline
\end{tabular}
}
\end{center}

\vskip0.5cm
We have 
$E_{3,3}^*=E_{\PP}^+ \oplus E_{\PP}^-$,  
with
$E_{\PP}^+ = \text{span}\lbrace v \rbrace$, 
$E_{\PP}^- = \text{span} \lbrace \QQ v \rbrace$. Thus
$$
E_{3,3}^*=N_+\oplus N_-,\quad
N_+=\spn\{y_1=\frac{x_1+x_2}{2}\},\quad
N_-=\spn\{y_2=\frac{x_1-x_2}{2}\}.
$$
Since $A\QQ=\QQ A$ we can decompose $A=A_+\oplus A_-$ such that $A_{\pm}\colon N_{\pm}\to N_{\pm}$. We have $\eta \PP \QQ=\QQ\eta \PP$ and 
$
\eta \PP=\Id_2
$
in the basis $\{y_k\}$, $k=1,2$.
Thus the condition
$
A^{T}\eta \PP A=\eta \PP$ is equivalent to two independent conditions
$
A^T_{\pm}A_{\pm}=\Id$.
We conclude
$\Aut^0(\n_{3,3}(V^{3,3;+}_{min}))\cong\Orth(1,\R)\times \Orth(1,\R)$.  We obtain $\Aut^0(\n_{3,3}(U)\cong\Orth(p,q;\R)\times \Orth(p,q;\R)$ for $U=(\oplus^pV^{3,3;+}_{min})\oplus (\oplus^qV^{3,3;-}_{min})$.


\subsection{Modules over $\mathbb{C}$, caring a negative involution}


In this cases we continue consider eigenspaces of the negative involution $\QQ$. The presence of the complex structure can preserve eigenspaces of $\QQ$ or reverse them. It leads to the different results.


\subsubsection{Cases $\n_{2,2}(U)$, $\n_{3,2}(U)$, $\n_{2,3}(U)$, $\n_{1,2}(U)$}


\begin{center}
\scalebox{0.7}
{
\begin{tabular}{|c|c|c|c|}
\hline
$V^{2,2}_{min}$ & \multicolumn{2}{c|}{} & $\dim=8$ 
\\ 
\hline
$E_{P_1}^{\pm}$ &     +   $E^*_{2,2}$   &     -      & $\dim=4$ 
\\
\hline
Basis for $E^*_{2,2}$& $x_1=v$ &   $\ldots$ & 
\\
 & $x_2=\ii v$ &  $\ldots$  & $P_1=J_1J_2J_3J_4$
 \\ 
 & $\textcolor{red}{x_3=J_2J_3v}$ & $\ldots$ & $\ii=J_1J_2$
 \\

& $\textcolor{red}{x_4=\ii J_2J_3v}$ & $\ldots$ & $\QQ=J_2J_3$
\\
\hline
$G_I$& & $J_3$ & \\ 
 \hline
\end{tabular}
}
\end{center}

\vskip0.5cm

We have the decomposition $E_{2,2}^*=N_+\oplus N_{-}$ with the bases:
\begin{equation}\label{eq:basis22}
N_+=\spn\{y_1=\frac{x_1+x_3}{2},\
y_2=\frac{x_2-x_4}{2}\},\
N_-=\spn\{y_3=\frac{x_1-x_3}{2},\
y_4=\frac{x_2+x_4}{2}\}.
\end{equation}

We write $A=A_+\oplus A_-$, where $A_+\in\GL(2,\R)$, $A_+\colon N_+\to N_+$. The map $A_-\colon N_-\to N_-$ can be found from the relation
$
A_-=J_1J_2A_+J_2J_1$.
We conclude that for minimal admissible module $A\in\GL(2;\R)$. In general $\Aut^0(\n_{2,2}(U))=\GL(2p,\R)$ for $U=\oplus^pV^{2,2;N}_{min}$. 

\vskip0.5cm

\begin{center}
\scalebox{0.7}
{
\begin{tabular}{|c|c|c|}
\hline
$V^{1,2}_{min}$ & & $\dim=4$ \\
\hline
Basis & $x_1=v$ & \\
& $x_2=\ii v$ & $\ii=J_2J_3$\\
& $\textcolor{red}{x_3=J_1J_2v}$ & $\QQ=J_1J_2$
\\
& $\textcolor{red}{x_4=\ii J_1J_2v}$ & $\PP=J_1J_2J_3$
\\
\hline
\end{tabular}
}
\end{center}

\vskip0.5cm

In this case there are two minimal admissible modules but they are metrically isotypic and we set $\PP v=v$. We start from a minimal admissible module and write $V^{1,2;N}_{min}=N_+\oplus N_-$. We also write $A=A_+\oplus A_-$, where $A_+\in\GL(2;\R)$ and $A_-=J_2J_3A_+J_3J_2$. We obtain
$
\eta\PP=\begin{pmatrix}
0&\Id_2
\\
\Id_2&0
\end{pmatrix}
$ in the basis~\eqref{eq:basis22}.
The condition
$$
A^T\eta\PP A=\eta\PP\quad \Longleftrightarrow\quad A_+^T\oplus A_-^T\begin{pmatrix}
0&\Id_2
\\
\Id_2&0
\end{pmatrix}A_+\oplus A_-=\begin{pmatrix}
0&\Id_2
\\
\Id_2&0
\end{pmatrix}
$$
is equivalent to  
$$
A^T_+ A_-=\Id_2\quad \Longleftrightarrow\quad A_+^TJ_2J_3A_+J_3J_2=\Id
\quad \Longleftrightarrow\quad 
A_+^T\begin{pmatrix}
0&-1
\\
1&0
\end{pmatrix}A_+=\begin{pmatrix}
0&-1
\\
1&0
\end{pmatrix}.
$$
Thus we conclude that 
$\Aut^0(\n_{1,2}(V^{1,2;N}_{min}))=\Sp(2;\R)$. For a general admissible module $U=\oplus^pV^{1,2;N}_{min}$ we obtain $\Aut^0(\n_{1,2}(U)=\Sp(2p;\R)$. 

\vskip0.5cm

\begin{center}
\scalebox{0.7}{
\begin{tabular}{|c|c|c|c|}
\hline
$V^{3,2}_{min}$ & \multicolumn{2}{c|}{} & $\dim=8$ 
\\ 
\hline
$E_{P_1}$ &     +   $E^*_{3,2}$   &     -      & $\dim=4$ 
\\
\hline
Basis for $E^*_{3,2}$ &$x_1=v$ &  $\ldots$ & $\begin{array}{ll}&P_1=J_1J_2J_4J_5\\&P_2=J_3J_4J_5\end{array}$
\\
 & $x_2=\ii v$  &   $\ldots$ & $\ii=J_4J_5$
\\
& \textcolor{red}{$x_3=\QQ v$} & $\ldots$ & $\QQ=J_1J_4$
\\ 
& \textcolor{red}{$x_4=\ii \QQ v$} & $\ldots$ & $\PP=J_3J_4J_5$
\\
\hline
$G_I$& & $J_1$ & 
\\  
\hline
\end{tabular}
}
\end{center}

\vskip0.5cm

We have 
$E_{3,2}^*=E_{\PP}^+ \oplus E_{\PP}^-$,  
with 
$E_{\PP}^+ = \text{span}\lbrace v, \ii v \rbrace$ and 
$E_{\PP}^- = \text{span} \lbrace \QQ v, \QQ\ii v \rbrace$,
and $\eta\PP=\Id$ in the basis~\eqref{eq:basis22}. As before we decompose $A=A_+\oplus A_-$ on $E^*_{3,2}$ with $A_+\in\GL(2(p+q);\R)$ and $A_-=-\ii A_+\ii$. The condition
$$
A^T_+\Id_{2p,2q} A_+=\Id_{2p,2q}\quad\text{on}\quad
U=\big(\oplus^pV_{min}^{3,2;+}\big)\oplus\big(\oplus^qV_{min}^{3,2;-}\big)
$$
leads to the conclusion that 
$\Aut^0(\n_{3,2}(U))=\Orth(2p,2q;\R)$. 

\vskip0.5cm

\begin{center}
\scalebox{0.7}
{
\begin{tabular}{|c|c|c|c|}
\hline
$V^{2,3}_{min}$ & \multicolumn{2}{c|}{} & $\dim=8$ 
\\ 
\hline
$E_{P_1}^{\pm}$ &     +  $E^*_{2,3}$    &     -      & $\dim=4$ 
\\
\hline
Basis for $E^*_{2,3}$ & $x_1=v$  &   $\ldots$ & $\begin{array}{ll}&P_1=J_1J_2J_4J_5\\&P_2=J_1J_3J_5\end{array}$
\\
& $x_2=\ii v$ &  $\ldots$ & $\ii=J_4J_5$
\\
& \textcolor{red}{$x_3=\QQ v$} & $\ldots$ & $\QQ=J_1J_4$
\\
& \textcolor{red}{$x_4=\ii \QQ v$} &   $\ldots$  & $\PP=J_1J_3J_5$
\\
\hline
$G_I$& & $J_1$ & \\
\hline
\end{tabular}
}
\end{center}

\vskip0.5cm

Arguing as in the case $V^{3,2}_{min}$ and by making use the basis~\eqref{eq:basis22} we come to condition
$$
A_+^T
\begin{pmatrix}
1&0
\\
0&-1
\end{pmatrix}
A_+=
\begin{pmatrix}
1&0
\\
0&-1
\end{pmatrix},\quad\text{for}\quad V_{min}^{2,3;N}
$$
For a general module $U=\oplus^pV_{min}^{2,3;N}$
we conclude that 
$\Aut^0(\n_{2,3}(U))=\Orth(p,p;\R)$. 

\vskip0.5cm

\begin{center}
\scalebox{0.7}{
\begin{tabular}{|c|c|c|}
\hline
$\n_{2,1}$ & & $\dim=8$ 
\\
\hline
Basis & $x_1=v$ & 
\\
& $x_2=\ii v$ & 
\\
& $\textcolor{red}{x_3=\QQ v}$ & 
\\
& $\textcolor{red}{x_4=\ii \QQ v}$& $\ii=J_1J_2$ 
\\
& $\textcolor{red}{x_5=J_1J_2J_3v}$ & $\QQ=J_2J_3$
\\
& $\textcolor{red}{x_6=\ii J_1J_2J_3v}$ & $\PP=J_1J_2J_3$ 
\\
& $x_7=\QQ J_1J_2J_3v$ & $$ 
\\
& $x_8=\ii \QQ J_1J_2J_3v$ & 
\\
\hline
\end{tabular}
}
\end{center}
\vskip0.5cm

We use the basis for:
\begin{equation}\label{eq:basis21}
\begin{array}{lllllllll}
&y_1=\frac{x_1+x_3}{2},\quad
&y_2=\frac{x_2-x_4}{2},\quad
&y_3=\frac{x_5+x_7}{2},\quad
&y_4=\frac{x_6-x_8}{2},\quad
\\
&y_5=\frac{x_1-x_3}{2},\quad
&y_6=\frac{x_2+x_4}{2},
&y_7=\frac{x_5-x_7}{2},\quad
&y_8=\frac{x_6+x_8}{2},
\end{array}
\end{equation}
for $V^{2,1;N}_{min}=N_+\oplus N_-$ with
$
N_+=\spn\{y_1,y_2,y_3,y_4\}$, $N_-=\spn\{y_5,y_6,y_7,y_8\}$.
We write $A=A_+\oplus A_-$ with $A_+\in\GL(4;\R)$ and $A_-=J_1J_2A_+J_2J_1$ in the basis~\eqref{eq:basis21}. Then
$$
\eta\PP=
\begin{pmatrix}
0&S
\\
S&0
\end{pmatrix}\quad\text{with}\quad
S=-\begin{pmatrix}
0&\Id_2
\\
\Id_2&0\end{pmatrix}.
$$
Thus we need to check the condition 
$$
A^T_+SA_-=S\quad\Longleftrightarrow\quad A^T_+SJ_1J_2A_+=SJ_1J_2
\quad\text{with}\quad 
SJ_1J_2=\begin{pmatrix}
0&0&0&1
\\
0&0&-1&0
\\
0&1&0&0
\\
-1&0&0&0
\end{pmatrix}.
$$
Finally, we conclude that 
$\Aut^0(\n_{2,1}(U))=\Sp(4p;\R)$.


\subsection{Modules over $\mathbb{H}$ caring a negative involution}


\subsubsection{Cases $\n_{1,3}(U)$, $\n_{3,1}(U)$, $\n_{1,5}(U)$, $\n_{5,1}(U)$, $\n_{7,3}(U)$, $\n_{3,7}(U)$,}


\begin{center}
\scalebox{0.7}{
\begin{tabular}{|c|c|c|}
\hline
$\n_{1,3}$ & & $\dim=8$ \\
\hline
Basis & $x_1=v$ & 
\\
& $x_2=\ii v$ & $P=J_1J_2J_3$
\\
& $x_3=\jj v$ & 
\\
& $x_4=\kk v$ & $\ii=J_2J_3$ 
\\
& $\textcolor{red}{x_5=J_4v}$ &  $\jj=J_3J_4$
\\
& $\textcolor{red}{x_6=\ii J_4v}$ &$\kk=J_2J_4$  
\\
& \textcolor{red}{$x_7=\jj J_4v$} &  $\QQ=J_1J_2$
\\
& \textcolor{red}{$x_8=\kk J_4v$} & $\PP=J_1$
\\
\hline
\end{tabular}
}
\end{center}

\vskip0.5cm

We choose $P_1 v=v$ and the basis for $V^{1,3;N}_{min}=N_+\oplus N_-$: 
\begin{equation}\label{eq:basis13}
\begin{split}
&N_+=\spn\{y_1=\frac{x_1+x_7}{2},\
y_2=\frac{x_3-x_5}{2},\
y_3=\frac{x_4+x_6}{2},\
y_4=\frac{x_2-x_8}{2}\},
\\
&N_-=\spn\{y_5=\frac{x_1-x_7}{2},\
y_6=\frac{x_3+x_5}{2},\
y_7=\frac{x_4-x_6}{2},\
y_8=\frac{x_2+x_8}{2}\}.
\end{split}
\end{equation}
Since $A\QQ=\QQ A$ and $\QQ\ii=-\ii \QQ$ we write $A=A_+\oplus A_-$, where $A_+\in\GL(4;\R)$ and $
A_-=J_2J_3 A_+ J_3J_2$.
Since 
$\QQ\jj=\jj \QQ$ we deduce that $A_+\in\GL(2;\CC)$. Moreover
\begin{equation}\label{eq:basisu}
N_+=\spn\{y_1,\quad
y_2=\jj y_1,\quad
y_3= y_3,\quad
y_4=\jj y_3\}.
\end{equation}

We also have $\eta \PP \QQ=\QQ\eta \PP$,  $\eta \PP\jj=\jj\eta\PP$ with the matrix $\eta \PP\vert_{N_+}=\begin{pmatrix}
0&\mathbf{i}\\\mathbf{i}&0
\end{pmatrix}$. It leads to
$$
A_+^T \eta \PP A_+=\eta \PP 
\quad\Longleftrightarrow\quad
\overline{(A_+)_{\CC}^T}\begin{pmatrix}
0&1\\1&0
\end{pmatrix} (A_+)_{\CC}=\begin{pmatrix}
0&1\\1&0
\end{pmatrix}.
$$
The matrix $\begin{pmatrix}
0&1\\1&0
\end{pmatrix}$ is Hermitian of index $(1,1)$. We conclude that $\Aut^0(\n_{1,3}(V^{1,3;N}_{min}))\cong \U(1,1)$ and $\Aut^0(\n_{1,3}(U)\cong \U(p,p)$ for $U=\oplus^p V^{1,3;N}_{min}$.

\vskip0.5cm

\begin{center}
\scalebox{0.7}{
\begin{tabular}{|c|c|c|}
\hline
$\n_{3,1}$ & & $\dim=8$ 
\\
\hline
Basis & $x_1=v$ & 
\\
& $x_2=\ii v$ & $P_1=J_1J_2J_3$
\\
& $x_3=\jj v$ & \\
& $x_4=\kk v$ & $\ii=J_2J_3$
\\
& $\textcolor{red}{x_5=J_4v}$ & $\jj=J_1J_2$
\\
& $\textcolor{red}{x_6=\ii J_4v}$ & $\kk=J_1J_3$
\\
& $\textcolor{red}{x_7=\jj J_4v}$ &$\QQ=J_3J_4$ 
\\
& $\textcolor{red}{x_8=\kk J_4v}$ & $\PP=J_4$
\\
\hline
\end{tabular}
}
\end{center}

\vskip0.5cm

We have $V^{3,1;+}_{min}=E_{P_1}^+ \oplus E_{P_1}^-$, with 
$$
E_{P_1}^+ = \text{span}\lbrace v, \ii v,\jj v, \kk v \rbrace,\quad 
E_{P_1}^- = \text{span} \lbrace J_4v, \ii J_4v,\jj J_4v, \kk J_4v \rbrace.
$$ 
The negative involution $\QQ$ decomposes $
V^{3,1;+}_{min}=N_+\oplus N_-$ with the basis given by~\eqref{eq:basis13}.
Since $A\QQ=\QQ A$ and $\QQ\ii=-\ii \QQ$ we write $A=A_+\oplus A_-$, where $A_+\in\GL(4;\R)$ and $
A_-=-\ii A_+\ii$.
The condition
$\QQ\jj=\jj \QQ$ implies $A_+\in\GL(2;\CC)$. 
We also have $\eta \PP \QQ=\QQ\eta \PP$ and $\eta \PP \jj=\jj\eta \PP$ with $\eta \PP\vert_{N_+}=\diag_2\begin{pmatrix}
0&-1\\1&0
\end{pmatrix}$ in the basis~\eqref{eq:basisu}. It leads to
$$
A_+^T \eta \PP A_+=\eta \PP 
\quad\Longleftrightarrow\quad
\overline{(A_+)_{\CC}^T} (A_+)_{\CC}=\Id_2.
$$
The conclusion is that $\Aut^0(\n_{3,1}(V^{3,1;+}_{min}))\cong \U(2)$ and $\Aut^0(\n_{3,1}(U)\cong \U(2p,2q)$ for $U=(\oplus^p V^{3,1;+}_{min})\oplus (\oplus^q V^{3,1;-}_{min})$.

\vskip0.5cm

\begin{center}
\scalebox{0.7}{
\begin{tabular}{|c|c|c|c|}
\hline
$V^{5,1}_{min}$ & \multicolumn{2}{c|}{} & $\dim=16$ \\ \hline
$E_{P_1}^{\pm}$ &     +  $E^*_{5,1}$    &     -      & $\dim=8$ \\
\hline
Basis for $E^*_{5,1}$ & $x_1=v$ & $\ldots$ &  $P_1=J_1J_2J_3J_4$
\\
 & $x_2=\ii v$ &   $\ldots$ &  $P_2=J_1J_2J_5$
 \\
& $x_3=\jj v$ & $\ldots$  & 
\\
&  $x_4=\kk v$ & $\ldots$ & $\ii=J_1J_3$
\\
& \textcolor{red}{$x_5=J_6v$} & $\ldots$  & $\jj=J_1J_2$ 
\\
& \textcolor{red}{$x_6=\ii J_6v$} & $\ldots$ & $\kk=J_3J_2$
\\
& \textcolor{red}{$x_7=\jj J_6v$} &  $\ldots$  &  $\QQ=J_5J_6$
\\
& \textcolor{red}{$x_8=\kk J_6v$} &  $\ldots$ &  $\PP=J_6$ 
\\
\hline
$G_I$& & $ J_1$ & \\
\hline
\end{tabular}
}
\end{center}

\vskip0.5cm

We have $E_{5,1}^*=E_{P_2}^+ \oplus E_{P_2}^-$, with 
$$
E_{P_2}^+ = \text{span}\lbrace v, \jj v,\ii J_6v, \kk J_6v \rbrace,\quad
E_{P_2}^- = \text{span} \lbrace J_6v, \ii v,\jj J_6v, \kk v,  \rbrace.
$$
The negative involution $\QQ$ decomposes $E_{5,1}^*$ into two eigenspaces 
$
E_{5,1}^*=N_+\oplus N_-$
with the basis
\begin{equation}\label{eq:basis51}
\begin{split}
&N_+=\spn\{y_1=\frac{x_1+x_7}{2},\
y_2=\frac{x_2+x_8}{2},\
y_3=\frac{x_3-x_5}{2},\
y_4=\frac{x_4-x_6}{2}\},
\\
&N_-=\spn\{y_5=\frac{x_1-x_7}{2},\
y_6=\frac{x_2-x_8}{2},\
y_7=\frac{x_3-x_5}{2},\
y_8=\frac{x_4+x_6}{2}\}.
\end{split}
\end{equation}
Since $A\QQ=\QQ A$ and $\QQ\ii=\ii \QQ$, $\QQ\jj=\jj \QQ$ we write $A=A_+\oplus A_-$, where $A_{\pm}\in\GL(1;\HH)$. Moreover
\begin{equation}\label{eq:basis51+}
N_+=\spn\{y_1,\quad
y_2=\ii y_1,\quad
y_3= y_3,\quad
y_4=\ii y_3\}.
\end{equation}
We also have $\eta \PP \QQ=\QQ\eta \PP$,  $\eta \PP\ii=\ii\eta \PP$, $\eta \PP\jj=\jj\eta \PP$. Thus $\eta \PP$ is quaternion linear and $\eta \PP \vert_{N_{\pm}}={\mathbf j}$, written in the basis~\eqref{eq:basis51+}. It leads to
$$
A_{\pm}^T \eta \PP A_{\pm}=\eta \PP \vert_{N_+}
\quad\Longleftrightarrow\quad
\overline{(A_{\pm})_{\HH}^T} {\mathbf j}(A_{\pm})_{\HH}={\mathbf j}.
$$
The conclusion is that $\Aut^0(\n_{5,1}(V^{5,1;N}_{min}))\cong \Orth^*(2)\times \Orth^*(2)$ and $\Aut^0(\n_{5,1}(U)\cong \Orth^*(2p)\times \Orth^*(2p)$ for $U=\oplus^p V^{5,1;N}_{min}$.

\vskip0.5cm

\begin{center}
\scalebox{0.7}{
\begin{tabular}{|c|c|c|c|}
\hline
$V^{1,5}_{min}$ & \multicolumn{2}{c|}{} & $\dim=16$ \\ \hline
$E_{P_1}^{\pm}$ &     +  $E^*_{1,5}$    &     -      & $\dim=8$ \\
\hline
Basis for $E^*_{1,5}$ & $x_1=v$ &  $\ldots$ &$P_1=J_2J_3J_4J_5$ 
\\
& $x_2=\ii v$ &  $\ldots$ & $P_2=J_1J_2J_3$ 
\\
& $x_3=\jj v$  & $\ldots$ & 
\\
& $x_4=\kk v$ & $\ldots$ &$\ii=J_3J_4$ 
\\
& \textcolor{red}{$x_5=J_6v$} & $\ldots$  & $\jj=J_2J_3$
\\
& \textcolor{red}{$x_6=\ii J_6v$} &   $\ldots$ &$\kk=J_4J_2$ 
\\
& \textcolor{red}{$x_7=\jj J_6v$}  &   $\ldots$  & $\QQ=J_1J_6$
\\
& \textcolor{red}{$x_8=\kk J_6v$} & $\ldots$ & $\PP=J_6$
\\
\hline
$G_I$& & $J_5$ & \\
\hline 
\end{tabular}
}
\end{center}

\vskip0.5cm

With the chosen operators $\ii,\jj, \QQ, \PP$ the calculations are identical to the case of $\n_{5,1}$ and we conclude that $\Aut^0(\n_{1,5}(U)\cong \Orth^*(2p)\times \Orth^*(2p)$ for $U=\oplus^p V^{1,5;N}_{min}$.

\vskip0.5cm

\begin{center}
\centering
\scalebox{0.7}{
\begin{tabular}{|c|c|c|c|c|c|c|c|c|c|}
\hline
$V^{7,3}_{min}$ & \multicolumn{8}{c|}{}                                                                         & $\dim=64$ 
\\ 
\hline
$E_{P_1}^{\pm}$ & \multicolumn{4}{c|}{+}                         & \multicolumn{4}{c|}{-}                         & $\dim=32$ 
\\ 
\hline
$E_{P_2}^{\pm}$ & \multicolumn{2}{c|}{+} & \multicolumn{2}{c|}{-} & \multicolumn{2}{c|}{+} & \multicolumn{2}{c|}{-} & $\dim=16$ 
\\ 
\hline
$E_{P_3}^{\pm}$ &       +  $E^*_{7,3}$  &     -      &          + &      -     &        +   &     -      &          + &    -       & $\dim=8$ \\ \hline
Basis for $E^*_{7,3}$& $x_1=v$ & $\ldots$& $\ldots$ & $\ldots$& $\ldots$ & $\ldots$ & $\ldots$& $\ldots$& $P_1=J_1J_2J_4J_5$
\\
& $x_2=\ii v$ & $\ldots$& $\ldots$& $\ldots$ & $\ldots$ & $\ldots$ &$\ldots$ & $\ldots$ & $P_2=J_1J_2J_6J_7$
\\
&  \textcolor{red}{$x_3=\jj v$} & $\ldots$ & $\ldots$ & $\ldots$ & $\ldots$ &$\ldots$ & $\ldots$ & $\ldots$ & $P_3=J_1J_2J_8J_9$
\\
&  \textcolor{red}{$x_4=\kk v$} & $\ldots$& $\ldots$& $\ldots$ & $\ldots$ &  $\ldots$ & $\ldots$ & $\ldots$ & $P_4=J_1J_2J_3$ 
\\
 & \textcolor{red}{$x_5=J_{10}v$} & $\ldots$ & $\ldots$ & $\ldots$ & $\ldots$ & $\ldots$ &$\ldots$& $\ldots$ & $\ii=J_1J_2$ 
 \\
& \textcolor{red}{$x_6=\ii J_{10}v$} & $\ldots$ & $\ldots$ & $\ldots$ & $\ldots$ & $\ldots$ & $\ldots$& $\ldots$ & $\jj=J_1J_4J_6J_8$  
\\
& $x_7=\jj J_{10}v$ & $\ldots$ & $\ldots$ & $\ldots$ & $\ldots$ & $\ldots$ & $\ldots$& $\ldots$ &  $\kk=J_2J_4J_6J_8$
\\
& $x_8=\kk J_{10}v$ & $\ldots$ & $\ldots$ & $\ldots$ &  $\ldots$ &  $\ldots$ &$\ldots$& $\ldots$ & $\begin{array}{ll}&\QQ=J_3J_{10}\\&\PP=J_{10}\end{array}$
\\
\hline
$G_I$ & & $J_8$ & $J_6$ & $J_1J_4$ & $J_4$ & $J_1J_6$ & $J_1J_8$ & $ J_1$ & \\
\hline
\end{tabular}
}
\end{center}

\vskip0.5cm

Observe that $E_{7,3}^*=E_{P_4}^+ \oplus E_{P_4}^-$, with 
\[
E_{P_4}^+ = \text{span}\lbrace v, \ii v,\jj J_{10}v, \kk J_{10}v \rbrace,
\quad
E_{P_4}^- = \text{span} \lbrace J_{10}v, \ii J_{10}v, \jj v,\kk v \rbrace.
\]

We start from the minimal admissible module. The negative involution $\QQ$ decomposes $E_{7,3}^*$ into two eigenspaces 
$
E_{7,3}^*=N_+\oplus N_-$
with the bases
\begin{equation}\label{eq:basis73}
\begin{split}
&N_+=\spn\{y_1=\frac{x_1+x_6}{2},\
y_2=\frac{x_2-x_5}{2},\
y_3=\frac{x_4+x_7}{2},\
y_4=\frac{x_8-x_3}{2}\},
\\
&N_-=\spn\{y_5=\frac{x_1-x_6}{2},\
y_6=\frac{x_2+x_5}{2},\
y_7=\frac{x_7-x_4}{2},\
y_8=\frac{x_8+x_3}{2}\}.
\end{split}
\end{equation} 
Since $A\QQ=\QQ A$ and $\QQ\ii=\ii \QQ$, $\QQ\jj=\jj \QQ$ we write $A=A_+\oplus A_-$, where $A_{\pm}\in\GL(1;\HH)$. 
Moreover
\begin{equation}\label{eq:basis73+}
N_+=\spn\{y_1,\quad
y_2=\ii y_1,\quad
y_3= y_3,\quad
y_4=\ii y_3\}.
\end{equation} 
We also have $\eta \PP \QQ=\QQ\eta \PP$,  $\eta \PP\ii=\ii\eta \PP$, $\eta \PP\jj=-\jj\eta \PP$ with $\big(\eta \PP \vert_{N_{\pm}}\big)_{\CC}={\mathbf i}\Id_2$, written in the basis~\eqref{eq:basis73+}. It leads to
$$
A_{\pm}^T \eta \PP A_{\pm}=\eta \PP \vert_{N_+}
\quad\Longleftrightarrow\quad
\overline{(A_{\pm})_{\CC}^T} \Id_2(A_{\pm})_{\CC}=\Id_2.
$$

Thus we conclude $\Aut^0(\n_{7,3}(V^{7,3;+}_{min})\cong \Sp(1)\times \Sp(1)$ for a minimal admissible module. If $U=(\oplus^p V^{7,3;+}_{min})\oplus (\oplus^q V^{7,3;-}_{min})$, then $\Aut^0(\n_{7,3}(U)\cong \Sp(p,q)\times \Sp(p,q)$.
\\

The calculations and the table for $\n_{3,7}(U)$ are identical to $\n_{7,3}(U)$ and we conclude that 
$\Aut^0(\n_{3,7}(U)\cong \Sp(p,q)\times \Sp(p,q)$  for $U=(\oplus^p V^{3,7;+}_{min})\oplus (\oplus^q V^{3,7;-}_{min})$.


\section{Appendix}\label{App}



\subsection{Comparison of $\Aut^0(\n_{r,s}(U))$ for isomorphic algebras}\label{comparison}


{\bf Cases $\n_{1,0}(U)$, $\n_{0,1}(V)$; $\n_{2,0}(U)$, $\n_{0,2}(V)$; $\n_{5,1}(U)$, $\n_{1,5}(V)$}. 
$$
\n_{1,0}(\oplus^p V^{1,0;+}_{min})\cong \n_{0,1}(\oplus^p V^{0,1;N}_{min}),\quad
\n_{2,0}(\oplus^p V^{2,0;+}_{min})\cong \n_{0,2}(\oplus^p V^{0,2;N}_{min})
$$
$$
\n_{5,1}(\oplus^p V^{5,1;N}_{min})\cong \n_{1,5}(\oplus^p V^{1,5;N}_{min})
$$
$$
\Aut^0(\n_{1,0}(U))=\Aut^0(\n_{0,1}(V))\cong\Sp(2p,\mathbb R),\quad
\Aut^0(\n_{2,0}(U))=\Aut^0(\n_{0,2}(V))=\Sp(2p,\CC),
$$
$$
\Aut^0(\n_{1,5}(U))=\Aut^0(\n_{5,1}(V))=\Orth^*(2p)\times \Orth^*(2p).
$$
\\

{\bf Cases $\n_{4,0}(V)$, $\n_{0,4}(U)$; $\n_{2,6}(U)$, $\n_{6,2}(V)$; $\n_{8,0}(U)$, $\n_{0,8}(V)$, $\n_{4,4}(W)$; $\n_{1,6}(U)$, $\n_{6,1}(U)$; $\n_{2,5}(U)$, $\n_{5,2}(V)$}. 
$$
\n_{4,0}(\oplus^p V^{4,0;+}_{min})\cong \n_{0,4}(\oplus^p V^{0,4;N}_{min}),\qquad
\n_{2,6}(\oplus^p V^{2,6;N}_{min})\cong \n_{6,2}(\oplus^p V^{6,2;N}_{min})
$$
$$
\n_{1,6}(\oplus^p V^{1,6;N}_{min})\cong \n_{6,1}(\oplus^p V^{6,1;N}_{min}),\qquad
\n_{2,5}(\oplus^p V^{2,5;N}_{min})\cong \n_{5,2}(\oplus^p V^{5,2;N}_{min})
$$
$$
\n_{8,0}(\oplus^p V^{8,0;+}_{min})\cong \n_{0,8}(\oplus^p V^{0,8;+}_{min})\not\cong
\n_{4,4}(\oplus^p V^{4,4;+}_{min})
$$
$$
\Aut^0(\n_{4,0}(V))=\Aut^0(\n_{0,4}(U))=\Aut^0(\n_{2,6}(U))=\Aut^0(\n_{6,2}(V)=\GL(p,\HH);
$$
$$
\Aut^0(\n_{1,6}(U))=\Aut^0(\n_{6,1}(V))=\Aut^0(\n_{2,5}(U))=\Aut^0(\n_{5,2}(V))=\Orth^*(2p).
$$
$$
\Aut^0\big(\n_{8,0}(U)\big)=\Aut^0\big(\n_{0,8}(V)\big)=\Aut^0\big(\n_{4,4}(W)\big)=\GL(p,\mathbb R).
$$ 
\\

{\bf Cases $\n_{5,0}(U)$, $\n_{0,5}(V)$; $\n_{1,4}(U)$, $\n_{4,1}(V)$; $\n_{6,0}(U)$, $\n_{0,6}(U)$; $\n_{2,4}(U)$, $\n_{4,2}(U)$; $\n_{1,2}(U)$, $\n_{2,1}(U)$} Here we have that 
$$
\n_{5,0}(\oplus^{2p}V^{5,0;+}_{min})\cong \n_{0,5}(\oplus^{p}V^{0,5;N}_{min}),\quad
\n_{1,4}(\oplus^{2p}V^{1,4;+}_{min})\cong \n_{4,1}(\oplus^{p}V^{4,1;N}_{min})
$$
$$
\n_{6,0}(\oplus^{2p} V_{min}^{6,0;+})\cong \n_{0,6}(\oplus^p V_{min}^{0,6;N}),
\quad
\n_{2,4}(\oplus^{2p} V_{min}^{2,4;+})\cong \n_{4,2}(\oplus^p V_{min}^{4,2;N}),
$$
$$
\n_{1,2}(\oplus^{2p} V_{min}^{1,2;N})\cong \n_{2,1}(\oplus^p V_{min}^{2,1;N}),
$$
We also showed
$$
\Aut^0\big(\n_{5,0}(\oplus^p V_{min}^{5,0;+})\big)\cong\Orth^*(2p)
\quad
\text{and}
\quad
\Aut^0(\n_{0,5}(\oplus^p V_{min}^{0,5;N}))\cong\Orth^*(4p).
$$
$$
\Aut^0\big(\n_{1,4}(\oplus^p V_{min}^{1,4;+})\big)\cong\Orth^*(2p)
\quad
\text{and}
\quad
\Aut^0(\n_{4,1}(\oplus^p V_{min}^{4,1;N}))\cong\Orth^*(4p).
$$
$$
\Aut^0\big(\n_{6,0}(\oplus^p V_{min}^{6,0;+})\big)\cong\Orth(p;\CC)
\quad
\text{and}
\quad
\Aut^0\big(\n_{0,6}(\oplus^p V_{min}^{0,6;N})\big)\cong\Orth(2p;\CC).
$$
$$
\Aut^0\big(\n_{2,4}(\oplus^p V_{min}^{2,4;+})\big)\cong\Orth(p;\CC)
\quad
\text{and}
\quad
\Aut^0\big(\n_{4,2}(\oplus^p V_{min}^{4,2;N})\big)\cong\Orth(2p;\CC).
$$
$$
\Aut^0(\n_{1,2}(\oplus^p V_{min}^{1,2;N}))\cong\Sp(2p;\CC)
\quad
\text{and}
\quad
\Aut^0\big(\n_{2,1}(\oplus^p V_{min}^{2,1;N})\big)\cong\Sp(4p;\R).
$$
\\

{\sc Cases $\n_{3,0}(V)$, $\n_{0,3}(U)$; $\n_{7,0}(U)$, $\n_{0,7}(V)$; $\n_{3,4}(U)$, $\n_{4,3}(V)$.} 
$$
\n_{0,3}(V^{0,3;N}_{min})\cong \n_{3,0}\Big(\big(\oplus^{p}V^{3,0;+}_{min;+}\big)\oplus\big(\oplus^p V^{3,0;-}_{min;+}\big)\Big),
$$ 
$$
\n_{0,7}(\oplus^{p} V^{0,7;N}_{min})\cong\n_{7,0}\Big(\big(\oplus^{p}V^{7,0;+}_{min;+}\big)\oplus\big(\oplus^p V^{7,0;-}_{min;+}\big)\Big),
$$
$$
\n_{4,3}(\oplus^{p} V^{4,3;N}_{min})\cong\n_{3,4}\Big(\big(\oplus^{p}V^{3,4;+}_{min;+}\big)\oplus\big(\oplus^p V^{3,4;-}_{min;+}\big)\Big),
$$
$$
\Aut^0(\n_{0,3}\big(\oplus^{p} V^{0,3;N}_{min}\big)=\Sp(p,p),
\quad
\Aut^0(\n_{3,0}\Big(\big(\oplus^{p}V^{3,0;+}_{min;+}\big)\oplus\big(\oplus^q V^{3,0;-}_{min;+}\big)\Big)=\Sp(p,q),
$$ 
$$
\Aut^0(\n_{0,7}\big(\oplus^{p} V^{0,7;N}_{min}\big)\cong\Aut^0(\n_{4,3}\big(\oplus^{p} V^{4,3;N}_{min}\big)\cong\Orth(p,p),
$$ 
$$
\Aut^0(\n_{7,0}\Big(\big(\oplus^{p}V^{7,0;+}_{min;+}\big)\oplus\big(\oplus^q V^{7,0;-}_{min;+}\big)\Big)\cong\Aut^0(\n_{3,4}\Big(\big(\oplus^{p}V^{3,4;+}_{min;+}\big)\oplus\big(\oplus^q V^{3,4;-}_{min;+}\big)\Big)\cong\Orth(p,q).
$$ 
\\

{\sc Cases $\n_{1,7}(U)$, $\n_{7,1}(V)$; $\n_{5,3}(U)$, $\n_{3,5}(V)$; $\n_{2,7}(U)$, $\n_{7,2}(V)$; $\n_{6,3}(U)$, $\n_{3,6}(V)$.} 

$$\n_{1,7}(\oplus^{p} V^{1,7;N}_{min})\cong\n_{7,1}\big((\oplus^pV^{7,1;+}_{\min})\oplus(\oplus^pV^{7,1;-}_{\min})\big).
$$
$$\n_{5,3}(\oplus^{p} V^{5,3;N}_{min})\cong\n_{3,5}\big((\oplus^pV^{3,5;+}_{\min})\oplus(\oplus^pV^{3,5;-}_{\min})\big).
$$
$$\n_{2,7}(\oplus^{p} V^{2,7;N}_{min})\cong\n_{7,2}\big((\oplus^pV^{7,2;+}_{\min})\oplus(\oplus^pV^{7,2;-}_{\min})\big),
$$
$$\n_{6,3}(\oplus^{p} V^{6,3;N}_{min})\cong\n_{3,6}\big((\oplus^pV^{3,6;+}_{\min})\oplus(\oplus^pV^{3,6;-}_{\min})\big),
$$
$$
\Aut^0\big(\n_{5,3}(\oplus^{p} V^{5,3;N}_{min})\big)\cong\Aut^0\big(\n_{1,7}(\oplus^{p} V^{1,7;N}_{min})\big)\cong\U(p,p)
$$
$$
\Aut^0(\n_{7,1}\big(\oplus^pV^{7,1;+}_{\min})\oplus(\oplus^qV^{7,1;-}_{\min})\big)\cong
\Aut^0(\n_{3,5}\big(\oplus^pV^{3,5;+}_{\min})\oplus(\oplus^qV^{3,5;-}_{\min})\big)\cong\U(p,q).
$$
$$
\Aut^0\big(\n_{2,7}(\oplus^{p} V^{2,7;N}_{min})\big)\cong\Aut^0\big(\n_{6,3}(\oplus^{p} V^{6,3;N}_{min})\big)\cong\Sp(p,p),
$$
$$
\Aut^0(\n_{7,2}\big(\oplus^pV^{7,2;+}_{\min})\oplus(\oplus^qV^{7,2;-}_{\min})\big)\cong
\Aut^0(\n_{3,6}\big(\oplus^pV^{3,6;+}_{\min})\oplus(\oplus^qV^{3,6;-}_{\min})\big)\cong\Sp(p,q).
$$
\\

{\bf Cases $\n_{1,3}(U)$, $\n_{3,1}(V)$; $\n_{2,3}(U)$, $\n_{3,2}(V)$; $\n_{3,7}(U)$, $\n_{7,3}(V)$} In all these cases the pairs of the Lie algebras are not isomorphic for any choice of admissible modules. We have 
$$
\Aut^0(\n_{1,3}(U))=\U(p,p),\quad \Aut^0(\n_{3,1}(V))=\U(2p,2q);
$$
$$
\Aut^0(\n_{2,3}(U))\cong \Orth(p,p;\R),\quad
\Aut^0(\n_{3,2}(V))=\Orth(2p,2q;\R);
$$
$$
\Aut^0(\n_{3,7}(U))\cong \Aut^0(\n_{7,3}(V))\cong \Sp(p,q)\times \Sp(p,q).
$$


\subsection{Some isomorphisms} 

In the work~\cite[Theorem 11]{FM1} it was shown the existence of an isomorphism $\n_{1,7}(U^{1,7;N}_{\min})\cong \n_{7,1}(V_{min}^{7,1;\pm}\oplus V_{min}^{7,1;\pm})$. The proof was not constructive and did not show how the metric changes under the isomorphism. Therefore we propose here the constructive proof of $\n_{1,7}(U^{1,7;N}_{\min})\cong \n_{7,1}(V_{min}^{7,1;+}\oplus V_{min}^{7,1;-})$. We will construct the isomorphism only for minimal dimensional module. 
Thus we choose the basis $(z_1,\ldots,z_8)$
$$
\la z_k,z_k\ra_{1,7}=-1,\quad k=1,\ldots,7,\quad \la z_8,z_8\ra_{1,7}=1\quad\text{for}\quad
\mathbb R^{1,7}.
$$
$$
y_{1}=u,\quad y_{2}=\ii y_{l1},\quad y_{3}=J_{z_1}J_{z_2}J_{z_7}y_1,\quad y_{4}=\ii y_3=J_{z_8}y_{1}\quad\text{for}\quad
E^{1,7}\subset U^{1,7;N}_{\min}
$$
with $\la u,u\ra_{E^{1,7}}=1$ and the complex structure $\ii=J_{z_1}J_{z_2}J_{z_7}J_{z_8}$. 
We also choose the basis $(w_1,\ldots,w_8)$
$$
\la w_k,w_k\ra_{7,1}=1,\quad k=1,\ldots,7,\quad \la w_8,w_8\ra_{7,1}=-1\quad\text{for}\quad
\mathbb R^{7,1}.
$$
$$
x_{1}=v_1,\quad x_{2}=\tilde\ii x_{1}=-J_{w_8}v_{1},\quad x_{3}=J_{w_1}J_{w_2}J_{w_7}v_2,\quad x_{4}=\tilde\ii x_3=-J_{w_8}v_{2}
$$
for $E^{7,1;+}\oplus E^{7,1;-}\subset V_{min}^{7,1;+}\oplus V_{min}^{7,1;-}
$
with $\la v_1,v_1\ra_{E^{7,1;+1}_{J_{w_1}J_{w_2}J_{w_7}}}=-\la v_2,v_2\ra_{E^{7,1;+1}_{J_{w_1}J_{w_2}J_{w_7}}}=1$ and the complex structure $\tilde \ii=J_{w_1}J_{w_2}J_{w_7}J_{w_8}$.
According to~\cite[Corollary 5, Theorem 3]{FM1} we define $C\colon \mathbb R^{1,7}\to \mathbb R^{7,1}$ by $C(z_k)=w_k$ and $C^{\tau}(w_k)=-z_k$, $k=1,\ldots,8$. The complex structure $\ii$ will correspond the complex structure $\tilde\ii$.

We define $A\colon E^{1,7;N}\to E^{7,1;+}\oplus E^{7,1;-}$ by setting 
$$
Ay_1=\sum_{m=1}^{4}a_mx_m,\quad Ay_3=\sum_{m=1}^{4}b_mx_m.
$$
Using the properties $A\ii=\tilde\ii A$ we deduce that 
$$A_{\CC}=
\begin{pmatrix}
\bar\lambda_1&\bar\mu_1
\\
\bar\lambda_2&\bar\mu_2
\end{pmatrix},\quad
\eta^{1,7} J_{z_8}=
\begin{pmatrix}
0&-i
\\
-i&0
\end{pmatrix},\quad
\eta^{7,1}J_{w_8}=
\begin{pmatrix}
-i&0
\\
0&i
\end{pmatrix}.
$$
We need to check the condition 
$$
A^T\eta^{7,1}J_{w_8}A=-\eta^{1,7}J_{z_8}
\quad\Longleftrightarrow\quad 
\bar A^T_{\CC}
\begin{pmatrix}
-1&0
\\
0&1
\end{pmatrix}
A_{\CC}=
\begin{pmatrix}
0&1
\\
1&0
\end{pmatrix}
$$
It leads to finding the solution of the system 
$$
\begin{cases}
-|\lambda_1|^2+|\lambda_2|^2=0
\\
-|\mu_1|^2+|\mu_2|^2=0
\\
-\lambda_1\bar\mu_1+\lambda_2\bar\mu_2=1.
\end{cases}
\quad\Longrightarrow\quad
\begin{cases}
-\lambda_1=\lambda_2=\frac{1}{2},
\\
\mu_1=\mu_2=1.
\end{cases}
$$
As we see the Lie algebras $\n_{1,7}(V^{1,7;N}_{\min})$ and $\n_{7,1}(V_{min}^{7,1;+}\oplus V_{min}^{7,1;-})$ are isomorphic.

The isomorphism is extended to any modules and the algebras $\n_{1,7}(\oplus^{p} V^{1,7;N}_{min})$ and $\n_{7,1}\big((\oplus^pV^{7,1;+}_{\min})\oplus(\oplus^pV^{7,1;-}_{\min})\big)$.
Analogously we can show 
$$\n_{2,7}(\oplus^{p} V^{2,7;N}_{min})\cong\n_{7,2}\big((\oplus^pV^{7,2;+}_{\min})\oplus(\oplus^pV^{7,2;-}_{\min})\big),
$$
$$\n_{l,3}(\oplus^{p} V^{l,3;N}_{min})\cong\n_{3,l}\big((\oplus^pV^{3,l;+}_{\min})\oplus(\oplus^pV^{3,l;-}_{\min})\big),\quad l=5,6,
$$
$$
\n_{0,l}(V^{0,l;N}_{min})\cong \n_{l,0}\Big(\big(\oplus^{p}V^{l,0;+}_{min;+}\big)\oplus\big(\oplus^p V^{,0;-}_{min;+}\big)\Big), \quad l=3,7,
$$ 
$$
\n_{4,3}(\oplus^{p} V^{4,3;N}_{min})\cong\n_{3,4}\Big(\big(\oplus^{p}V^{3,4;+}_{min;+}\big)\oplus\big(\oplus^p V^{3,4;-}_{min;+}\big)\Big).
$$


\end{document}